\documentclass[12pt]{article}
\usepackage{graphicx, tikz}
\usepackage{amsmath}
\usepackage{amssymb}
\usepackage{theorem}
\usepackage{enumerate} 
\usepackage{color}

   \usepackage{hyperref}

\numberwithin{equation}{section}

\newtheorem{thm}{Theorem}[section]
\newtheorem{lemma}[thm]{Lemma}

\newtheorem{prop}[thm]{Proposition}
\newtheorem{cor}[thm]{Corollary}
{\theorembodyfont{\rmfamily}

\newtheorem{rmk}[thm]{Remark}
}

\makeatletter
\newcommand{\labeltext}[2]{%
  \@bsphack
  \csname phantomsection\endcsname 
  \def\@currentlabel{#1}{\label{#2}}%
  \@esphack
}
\makeatother

\newcommand{\qed}{\hfill \mbox{\raggedright \rule{.07in}{.1in}}}
 
\newenvironment{proof}{\vspace{1ex}\noindent{\bf
Proof}\hspace{0.5em}}{\hfill\qed\vspace{1ex}}
\newenvironment{pfof}[1]{\vspace{1ex}\noindent{\bf Proof of
#1}\hspace{0.5em}}{\hfill\qed\vspace{1ex}}

\def\subjclass#1{\par\medskip
\noindent\textbf{Mathematics Subject Classification (2010):} #1}
\def\keywords#1{\par\medskip
\noindent\textbf{Keywords.} #1}

\newcommand{\R}{{\mathbb R}}

\newcommand{\C}{{\mathbb C}}

\newcommand{\N}{{\mathbb N}}
\newcommand{\cA}{{\mathcal A}}
\newcommand{\cB}{{\mathcal B}}

\newcommand{\cN}{{\mathcal N}}

\newcommand\cM{{\mathcal M}}

\newcommand{\T}{{\mathbb T}}
\newcommand\cW{{\mathcal W}}
\newcommand\cY{{\mathcal Y}}
\newcommand\cP{{\mathcal P}}

\newcommand\Id{{\bf 1}}

\newcommand{\supp}{\operatorname{supp}}

\newcommand{\eps}{\epsilon}
\newcommand{\rf}{r}

\newcommand{\vertiii}[1]{{\left\vert\kern-0.25ex\left\vert\kern-0.25ex\left\vert #1
    \right\vert\kern-0.25ex\right\vert\kern-0.25ex\right\vert}}
\def\cM{\mathcal{M}}
\def\cP{\mathcal P}
\def\s{\mathbf{s}}
\def\u{\mathbf{u}}
\def\w{g}
\def\bgamma{\mathbf{\gamma}}
\def\aa{\chi}

\title{Pressure function and limit theorems for almost Anosov flows}

\begin{document}
\author{Henk Bruin 
\thanks{Faculty of Mathematics, University of Vienna, 
Oskar Morgensternplatz 1, 1090 Vienna, Austria, {\it henk.bruin@univie.ac.at}}
\and Dalia Terhesiu 
\thanks{Institute of Mathematics, University of Leiden,
Niels Bohrweg 1, 2333 CA Leiden, The Netherlands,
{\it daliaterhesiu@gmail.com}}
\and Mike Todd
\thanks{Mathematical Institute, University of St Andrews
North Haugh, St Andrews KY16 9SS, Scotland, {\it m.todd@st-andrews.ac.uk}}
}

\maketitle

\begin{abstract}
We obtain limit theorems (Stable Laws and Central Limit Theorems, both standard and non-standard)
and thermodynamic properties for a class of 
non-uniformly hyperbolic flows: almost Anosov flows, constructed here.
The link between the pressure function and limit theorems 
is studied in an abstract functional analytic framework, 
which may be applicable to other classes of non-uniformly hyperbolic flows.
\end{abstract}

\section{Introduction and summary of the main results}
\label{sec-intro}

This paper is a contribution to the theory of limit theorems and thermodynamic formalism in the context of 
non-uniformly hyperbolic flows on manifolds.  In particular we introduce a new class of `almost Anosov flows', 
a natural analogue of almost Anosov diffeomorphisms introduced in~\cite{HY95}, and prove limit theorems for 
natural observables $\psi$ with respect to the SRB measure, also giving the form of the associated pressure function. 
This example is presented in the context of a general framework, which may be applicable to a range of 
other non-uniformly hyperbolic flows. We recall that various statistical properties for several classes 
of Anosov flows are known~\cite{Dolgopyat98, Dolgopyat03, Liverani04}, but none of these results apply 
to the class of `almost Anosov flows' considered here.

Given a flow $(\Phi_t)_t$ with a real-valued potential $\phi$, we suppose that there is an equilibrium state $\mu_\phi$.  
A standard way of studying the behaviour of averages of a real-valued observable $\psi$ is to consider 
a Poincar\'e section $Y$ and study the induced first return map $F:Y\to Y$, 
the induced potentials $\bar\phi$ and $\bar\psi$, with the induced measure $\mu_{\bar\phi}$.  
In the discrete time case, in fact in the setting where all the dynamical systems are countable 
Markov shifts, \cite[Section 2]{Sar06} showed a one-to-one relation between limit laws 
for $\psi$ and the asymptotic form of the pressure function $\mathcal P(\phi+s\psi)$,
as $s\to 0$. This function experiences a phase transition at $s=0$, and its precise form determines the type of limit laws
(in particular, the index of the stable law), see \cite{MTor, Zwe07} and \cite[Theorem 7]{Sar06}. 
Furthermore, under certain conditions on $\psi$, \cite[Theorem 8]{Sar06} gave an asymptotically linear relation between the induced pressure $\mathcal P(\overline{\phi+s\psi})$ and $\mathcal P(\phi+s\psi)$.   
On the limit theorems side, \cite{MTor, Zwe07} and \cite[Theorem 7]{Sar06} show how one can go between results on the induced and on the original system: the first of these also applies in the flow setting.  
In these cases, the tail of the roof function/return time to the Poincar\'e section 
(both for the flow and the map) plays a major role in determining the form of the results.
Here we will follow this paradigm in the setting of a new class of flows.  The proofs  of the main theorems are facilitated by corresponding theorems in an abstract functional analytic framework. Applying this to the considered example requires  precise estimates on the tails of the roof function, which we prove for our main example.

Our central example is an  almost Anosov flow, which is a flow having a continuous flow-invariant splitting of the tangent bundle 
$T\cM = E^{\u} \oplus E^{\mathbf{c}} \oplus E^{\s}$
(where $E^{\mathbf{c}}_q$ is the one-dimension flow direction) such that we have exponential 
expansion/contraction in the directions
$E^{\u}_q, E^{\s}_q$, except for a finite number of periodic orbits  (in our case a single orbit).
We can think of these as perturbed Anosov flows, where the perturbation is local around these periodic orbits, making them neutral. A precise description is given in
Section~\ref{sec-Aaflow}. Almost Anosov diffeomorphisms have been introduced in~\cite{HY95}
and sufficiently precise estimates on the tail of the return function to a `good' set have been obtained in~\cite{BT17}. 
(These estimates are for both finite and infinite measure
preserving almost Anosov diffeomorphisms.)  We emphasise that the roof function is not constant on the stable direction, which is a main source of difficulty. When considering these examples we choose $\phi$ so that $\mu_\phi$ is the SRB measure.
We build 
on the construction in~\cite{BT17} to obtain
the asymptotic of the tail behaviour of the roof functions for almost Anosov flows (when viewed as suspension flows). The main technical results for these systems are Propositions~\ref{prop:w-integral}
and~\ref{prop:tailtau}.  These are then used to prove the main result Theorem~\ref{thm-conclAn}.

For a major part of the statements of the abstract theorems in this paper (in contexts not restricted to almost Anosov flows) 
we do not require any Markov structure for our system: it is only when we want to see our results in terms of the pressure function (making the connection between the leading eigenvalue 
of the  twisted transfer operator of the base map and pressure as in \cite[Theorem 4]{Sar99}) that this is needed.  Our setup requires good functional analytic properties of the  
Poincar{\'e} map in terms of abstract Banach spaces of distributions. Using the rather mild  abstract functional  assumptions described in Section~\ref{sec-abstr},
 in Sections~\ref{sec-lmtF} we obtain stable laws, standard and non-standard CLT; this is the content of Proposition~\ref{prop-limthF}.
In Section~\ref{sec-lmtflow} we recall~\cite[Theorem 7]{Sar06} and~\cite[Theorem 1.3]{MTor} to lift Proposition~\ref{prop-limthF} to the flow, which allow us to prove
Proposition~\ref{prop-limthflow}.

In Section~\ref{sec-pressure} we do exploit the assumption of the Markov structure to relate 
the definition of the pressure $\mathcal{P}(\bar\phi+s\bar\psi), s\geq 0$, 
with that of the family of eigenvalues of the family of twisted transfer operators of the  Poincar{\'e} map; the twist is in terms of the roof function of the suspension flow 
and the potential $\psi$.
Using this type of identification, in Theorem~\ref{thm-relpres} 
we relate the induced pressure $\mathcal{P}(\bar\phi+s\bar\psi)$ with the original pressure. 
Using the main result in Section~\ref{sec-pressure}, in Section~\ref{sec-concl} we summarise the results for the abstract framework in the concluding Theorem~\ref{thm-concl},  which gives
the equivalence between the asymptotic behaviour  of the pressure function $P(\phi+s\psi)$ and limit theorems. It is this summarising result that can be viewed as a 
version of Theorems 2--4 and Theorem 7 of \cite{Sar06} combined, for flows.

We note that limit theorems for the almost Anosov flows studied here could have been obtained via the very 
recent results of limit laws for invertible Young towers
as in \cite[Theorem 3.1]{MV19} together with the arguments of lifting limit laws from
the suspension to the flow in \cite{MTor, Zwe07} and \cite[Theorem 7]{Sar06}.
Notably \cite{MV19} applies to invertible Young towers and as such to several classes of billiard maps/flows.

\paragraph{Organisation of the paper:}
In Section~\ref{sec-Aaflow} we give the setup of the almost Anosov flow as ODE, and then
translate it to a suspension flow over a Poincar\'e section. Proposition~\ref{prop:tailtau} 
computes the tails of the associated roof function. Subsection~\ref{subsec-mainres} gives the main limit theorems
for the almost Anosov flow.
Section~\ref{sec-back} gives the preliminaries of equilibrium states for flows.
In Section~\ref{sec-abstr} we give the abstract background, including abstract hypotheses, of the transfer operator approach.
In this abstract setting, the main result Proposition~\ref{prop-limthF} is stated and proved in Section~\ref{sec-lmtF}.
Section~\ref{sec-lmtflow} states and proves the limit law for the flow,
and Section~\ref{sec-pressure} deals with the asymptotic shape of the pressure function.
In Section~\ref{sec-concl} we formulate and prove the limit laws for equilibrium states of the flow.
Finally, the appendix gives some technical proofs for Section~\ref{sec-lmtF}, and then concludes by checking all the
hypotheses of the abstract results.
In Appendix~\ref{sec-ver} we verify the abstract hypotheses of Proposition~\ref{prop-limthF} and Theorem~\ref{thm-concl} for the almost Anosov flows introduced in Section~\ref{sec-Aaflow}
which allows us to complete the proofs of Theorems~\ref{thm-conclAn}.

\paragraph{Notation:}
We use ``big O'' and $\ll$ notation interchangeably, writing $a_n=O(b_n)$ or $a_n\ll b_n$
if there is a constant $C>0$ such that $a_n\le Cb_n$ for all $n\ge1$.
We write $a_n \sim b_n$ if $\lim_n a_n/b_n = 1$.
Throughout we let $\to^d$ stand for convergence in distribution.
By $F$ being piecewise H\"older, etc.\ we mean that $F$ is H\"older on the elements $a \in \cA$, 
with a uniform exponent, but the H\"older norm of $F|_a$ is allowed to depend on $a$.

\paragraph{Acknowledgements:} 
HB gratefully acknowledges the support of FWF grant P31950-N45.
DT was partially supported by EPSRC grant EP/S019286/1.
The authors  would like to thank the Erwin Schr\"odinger Institute
where this paper was initiated during a ``Research in Teams'' project. 
The vigilant remarks of the referees are gratefully acknowledged.

\section{The setup for almost Anosov flows}\label{sec-Aaflow}

\subsection{Description via ODEs}
\label{subsec-ode}

Let $\cM$ be an odd-dimensional compact connected manifold without boundary.
An almost Anosov flow is a flow having continuous flow-invariant splitting of the tangent bundle 
$T\cM = E^{\u} \oplus E^{\mathbf{c}} \oplus E^{\s}$
(where $E^{\mathbf{c}}_q$ is the one-dimension flow direction) such that we have 
exponential expansion/contraction in the direction of
$E^{\u}_q, E^{\s}_q$, except for a finite number of periodic orbits (in our case a single orbit $\Gamma$).
After restricting to an irreducible component if necessary,
Anosov flows are automatically transitive, cf.\ \cite[5.10.3]{BS02}. 
As noted in the introduction, we can think of almost Anosov flows as perturbed Anosov flows, where the perturbation is local around $\Gamma$, 
making $\Gamma$ neutral.

Let $\Phi_t:\cM \times\R \to \cM$ be an almost Anosov flow on
a $3$-dimensional manifold. 
We assume that the flow in local Cartesian coordinates near the neutral periodic orbit $\Gamma := (0,0) \times \T$ is 
determined by a vector field $X:\cM \to T\cM$ defined as:
\begin{equation}\label{eq:polvf}
 \begin{pmatrix}
  \dot x \\ \dot y\\ \dot z
 \end{pmatrix}
 =  X \begin{pmatrix} x \\ y \\ z \end{pmatrix}  =
 \begin{pmatrix} x(a_0 x^2  + a_2 y^2) \\ -y(b_0 x^2 + b_2 y^2) \\ 1 + w(x,y) \end{pmatrix} 
\end{equation}
where\footnote{The notation and absence of mixed terms $a_1xy$ and $b_1xy$ goes back to \cite{Hu00},
who could not treat mixed terms. In \cite{BT17}, the mixed terms are absent too, in order to
make the explicit form of the first integrals possible. Importantly, \eqref{eq:polvf} ensure that the 
two vertical coordinate planes are local stable and unstable manifold of $\Gamma$. The only linear transformation 
that preserves this are trivial scalings in the coordinate directions. Such scalings
reduce $b_2/a_2$ and $a_0/b_0$ to single parameters, but we didn't do this to maintain the possibility
to compare proofs with \cite{BT17,Hu00} more easily.
It is possible to treat mixed terms to some extent, see \cite{Bruin}, but the additional required
technicalities are beyond the purpose of this paper.}
$a_0, a_2, b_0, b_2 \geq 0$ with $\Delta = a_2b_0 - a_0b_2 \neq 0$, and $w$ is a homogeneous function
with exponent $\rho \ge 0$ as leading term, cf.\ Remark~\ref{rem:psi}.
For the Limit Theorem~\ref{thm-conclAn}, we additionally require $a_2 > b_2$ 
(so that $\beta > 1$, the finite measure case).
We will use a Poincar\'e section $\Sigma$ which is $\{ z = 0\}$ in local coordinates near $\Gamma$, 
but which is thought to be defined across the whole manifold so that the Poincar\'e maps
is defined everywhere on $\Sigma$.  The Poincar\'e map of the Anosov flow (i.e., before the perturbation rendering $\Gamma$ neutral)
is Anosov itself, and hence has a Markov partition, for instance based on arcs in $W^s(p) \cup W^u(p)$ for $p = (0,0)$, but not 
 $W_{loc}^s(p) \cup W_{loc}^u(p)$ because we want $p$ in the interior of a partition element.
As the perturbation is local around $\Gamma$ and leaves local stable and unstable manifolds of $\Gamma$ unchanged,
the same Markov partition can be used for the almost Anosov flow.

Let us call the horizontal (i.e., $(x,y)$-component) of the flow
$\Phi_t^{hor}$.
This is the vector field of \cite[Equation (4)]{BT17}, with $\kappa = 2$ and the 
vertical component is added as a skew product.
Therefore we can take some crucial estimates from the estimates
of $\Phi_t^{hor}$ in \cite[Proposition 2.1]{BT17}.

The flow $\Phi_t$ has a periodic orbit $\Gamma = \{ p \} \times \T$ of period $1$
(which is neutral because $DX$ is zero on $\Gamma$),
and it has local stable/unstable manifolds 
$W^{\s}_{loc}(\Gamma) = \{ 0 \} \times (-\eps,\eps) \times  \T^1$ and
$W^{\u}_{loc}(\Gamma) = (-\eps,\eps) \times \{ 0 \} \times \T^1$. 
It is an equilibrium point of neutral saddle type if we only consider $\Phi^{hor}_t$.
The time-$1$ map $\hat f$ of $\Phi_t^{hor}$ is an almost Anosov map,
with Markov partition $\{ \hat P_i\}_{i \geq 0}$, where we assume that $p$ is an interior point of $\hat P_0$.

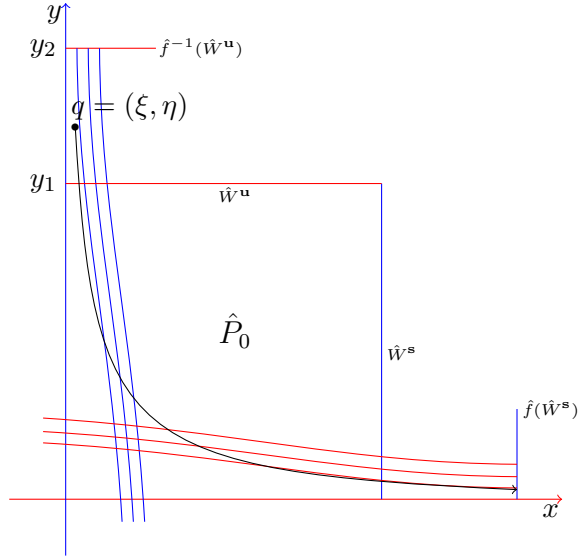
\begin{figure}[ht]
\begin{center}
\begin{tikzpicture}[scale=1.5]
\node at (4.3,-0.1) {\small $x$}; 
\node at (-0.1,4.3) {\small $y$}; \node at (-0.2,2.8) {\small $y_1$};\node at (-0.2,4) {\small $y_2$};
\draw[->, draw=red] (-0.5,0)--(4.4,0);
\draw[->, draw=blue] (0,-0.5)--(0,4.4);
\node at (1.5,1.5) {$\hat P_0$};
\node at (1.5,2.7) {\tiny $\hat W^{\u}$}; \node at (1.2,4) {\tiny $\hat f^{-1}(\hat W^{\u})$};
\node at (2.98,1.3) {\tiny $\hat  W^{\s}$}; \node at (4.3,0.8) {\tiny $\hat f(\hat W^{\s})$};
\draw[-, draw=red] (0,4)--(0.8,4); \draw[-, draw=blue] (4,0)--(4,0.8); 
\draw[-, draw=blue] (0.1,4) .. controls (0.11,2.5) and (0.4,1.5) .. (0.5,-0.2);
\draw[-, draw=blue] (0.2,4) .. controls (0.21,2.5) and (0.5,1.5) .. (0.6,-0.2);
\draw[-, draw=blue] (0.3,4) .. controls (0.31,2.5) and (0.6,1.5) .. (0.7,-0.2);
\draw[-, draw=blue] (2.8,2.8)--(2.8,-0.0);
\draw[-, draw=red] (4, 0.1) .. controls (2.5, 0.11) and (1.5, 0.4) .. (-0.2, 0.5);
\draw[-, draw=red] (4, 0.2) .. controls (2.5, 0.21) and (1.5, 0.5) .. (-0.2, 0.6);
\draw[-, draw=red] (4, 0.31) .. controls (2.5, 0.32) and (1.5, 0.61) .. (-0.2, 0.72);
\draw[-, draw=red] (2.8,2.8)--(-0.0,2.8);
\node at (0.57,3.48) {\small $q = (\xi,\eta)$};
\node at (0.085,3.3) {\tiny $\bullet$};
\draw[->, draw=black] (0.085,3.3) .. controls (0.25,0.5) and (0.5,0.25) .. (4,0.085);
\end{tikzpicture}
\caption{The first quadrant of the rectangle $\hat P_0$, with stable and unstable foliations of 
time-$1$ map $\hat f = \Phi_1^{hor}$
drawn vertically and horizontally, respectively. Also the integral curve of $q$ is drawn.}
\label{fig:leaves1}
\end{center}
\end{figure}

Given $q \in \hat f^{-1}(\hat P_0)$, define
\begin{equation}\label{eq:hattau}
 \hat\tau(q) := \min\{ t > 0 : \Phi_t^{hor}(q) \in \hat W^{\s}\},
\end{equation}
where $\hat  W^{\s}$ is the stable boundary leaf of $\hat  P_0$, see Figure~\ref{fig:leaves1}.
Let $\hat W^{\u}(y)$ denote the unstable leaf of $\hat f$ intersecting $(0,y) \in \hat  W^{\s}(p)$.
The function $\hat\tau$ is strictly monotone on $\hat W^{\u}(y)$.
For $y > 0$ and $T \geq 1$, let $\xi(y,T)$ denote the distance between $(0,y)$ and the (unique) point $q \in \hat W^{\u}(y)$ with 
$\hat\tau(q) = T$.
The crucial information from \cite[Proposition 2.1]{BT17} is 
\begin{equation}\label{eq:asymp0}
|\xi(y,T) - \xi(y,T+1)| = \beta \xi_0(y)^{-\beta} T^{-(\beta+1)} (1+o(1)) \quad \text{ as } T \to \infty,
\end{equation}
for $\beta = (a_2+b_2)/(2b_2)$ and specific values $\xi_0(y)$ given in \cite[Proof Proposition 2.1]{BT17}.
The parameters $a_2,b_2 \geq 0$ can be chosen freely, so \eqref{eq:asymp0} holds for $\beta \in (\frac12,\infty)$, 
but in this paper we choose $a_2 > b_2$, and therefore $\beta \in (1,\infty)$.
Also the ``small tail'' estimates of \eqref{eq:asymp0} are stronger than the ``big tail'' estimates needed
in our main theorems, but are required for particular results in the infinite measure setting of \cite{BT17}.

\begin{rmk}\label{rem:small_tails}
In fact,  for the most general vector fields treated in \cite[Equation (3)]{BT17} 
 (i.e., vector fields with higher order terms), we cannot do better than ``big tails'' estimates: 
 \begin{equation}\label{eq:asympt2}
 \xi(y,T) = \xi_0(y) T^{-\beta}(1+ O(T^{-\beta_*}+T^{-\frac12} \log T)),
 \end{equation}
 for $\beta_* = \min\{1, a_2/b_2, b_0/a_0\}$, see \cite[Lemma 2.3]{BT17}.
This is due to the perturbation arguments in \cite[Section 2.4]{BT17} that become necessary when
higher order terms are present and
a precise first integral from \cite[Lemma 2.2]{BT17} is not available.
 Since the horizontal part of vector field of \eqref{eq:polvf} here is the `ideal' vector field 
 from \cite[Equation (4)]{BT17}, our small tail estimate \eqref{eq:asymp0} becomes possible.
\end{rmk}

The next proposition gives an estimate of integrals along such curves.
This allows us to estimate the tail of the roof function (when viewing $\Phi_t$ as a suspension flow) 
in Proposition~\ref{prop:tailtau}
and also, the tail of induced potentials (see Remark~\ref{rem:psi}). 
In the first instance we use it to estimate the vertical component of the flow
$\Phi_t$ (compared to $t$).

\begin{prop}\label{prop:w-integral}
 Let $\theta:\R^2 \to \R$ be a homogeneous function with exponent $\rho \in \R$ such that
 $\theta(x,0) \not\equiv 0 \not\equiv \theta(0,y)$ for $x\neq 0 \neq y$.
Then there is a constant $C_\rho > 0$ (given explicitly in the proof) such that 
for every $q$ with $\hat\tau(q) = T$ as in \eqref{eq:hattau}, 
\begin{equation}\label{eq:w-integral}
\Theta(T) := \int_0^T \theta( \Phi_t(q,0))\, dt =
\begin{cases}
C_\rho T^{1-\frac{\rho}{2}} (1+o(1)) & \text{ if } \rho < 2, \\
C_\rho  \log(T) (1+o(1)) & \text{ if } \rho = 2, \\
C_\rho (1+o(1)) & \text{ if } \rho > 2.
\end{cases}
\end{equation}
\end{prop}

For the proof of Proposition~\ref{prop:w-integral} we recall some notation and the form of the first integral $L(x,y)$ from
\cite[Lemma 2.2]{BT17}, replacing $\kappa$ from that paper by $2$.
Recall that $\Delta := a_2b_0 - a_0b_2 \neq 0$.
Let $u,v \in \R$ be the solutions of the linear equation
\begin{equation}\label{eq:uv}
\begin{cases}
(u+2) a_0 = v b_0 \\
(v+2)b_2 = u a_2
\end{cases}
\quad \text{ that is:}\quad
\begin{cases}
u = \frac{2 b_2}{\Delta}(a_0+b_0), \\[2mm]
v = \frac{2 a_0}{\Delta} (a_2+b_2).
\end{cases}
\end{equation}
Note that $u,v$ and $\Delta$ all have the same sign. Define:
\begin{equation}\label{eq:gamma}
\beta_0 := \frac{a_0+b_0}{2 a_0} = \frac{u+v+2 }{2 v}, \quad 
\beta := \beta_2 := \frac{a_2+b_2}{2 b_2} = \frac{u+v+2 }{2 u}, \quad
\begin{array}{l} c_0 = a_0 + b_0 \\ c_2 = a_2+b_2 \end{array}.
\end{equation}
Note that $\frac{\beta_0}{\beta_2} = \frac{u}{v}$, $\frac{a_0 u}{b_2v} = \frac{c_0}{c_2}$ and $\beta_0, \beta_2 > \frac{1}{2}$ (or $=\frac{1}{2}$ if 
we allow $b_0=0$ or $a_2=0$ respectively).
The content of \cite[Lemma 2.2]{BT17} is that
\begin{equation}\label{eq:lf}
L(x,y) = 
\begin{cases}
  x^u y^v ( \frac{a_0}{v}\ x^2  + \frac{b_2}{u}\ y^2 ) & \text{ if } \Delta > 0;\\
  x^{-u} y^{-v} ( \frac{a_0}{v}\ x^2  + \frac{b_2}{u}\ y^2 )^{-1} & \text{ if } \Delta < 0,
\end{cases}
\end{equation}
is a first integral of (i.e., preserved by) $\Phi_t^{hor}$ (and therefore of $\Phi_t$).

For the proof of Proposition~\ref{prop:w-integral}, we follow the proof of \cite[Proposition 2.1]{BT17}.
In comparison, we have $\kappa$ from \cite[Proposition 2.1]{BT17} equal to $2$, 
our current integrand is more complicated, but we only need first order error terms.

\begin{pfof}{Proposition~\ref{prop:w-integral}}
Fix $\eta$ such that the local unstable leaf $\hat W^{\u}_{loc}(0,\eta)$ intersects 
$\overline{\hat f^{-1}(\hat P_0) \setminus \hat P_0}$.
Recall from the text below \eqref{eq:hattau} that, given $T$ large enough, 
there is a unique point $(\xi(\eta,T), \eta') \in \hat W^{\u}_{loc}(0,\eta)$ 
such that
$$
(\zeta'_0, \omega(\eta,T)) := \Phi_T^{hor}((\xi(\eta,T), \eta')) \in \hat W_{loc}^s(\zeta_0,0),
$$
where $\hat W_{loc}^s(\zeta_0,0)$ is the local stable leave forming the right boundary of $P_0$,
see Figure~\ref{fig:leaves3}.
Assuming $\hat W^{\u}_{loc}(0,\eta)$ and $\hat W_{loc}^s(\zeta_0,0)$ are straight horizontal and vertical lines
respectively, induces a negligible error in these vertical coordinates,
so in the rest of the proof, we write $\eta' = \eta$ and $\zeta_0' = \zeta_0$.

\begin{figure}[ht]
\begin{center}
\begin{tikzpicture}[scale=1.5]
\draw[->, draw=red] (-0.2,0)--(4,0);
\draw[->, draw=blue] (0,-0.2)--(0,4);
\node at (1.5,1.5) {$\hat P_0$};
\node at (3.37,1.3) {\tiny $\hat  W^{\s}(\zeta'_0,\omega(\eta,T))$}; 
\node at (2.8,-0.2)  {\small $\zeta_0$};
\draw[-, draw=red] (0,3.3)--(1.8,3.3); \node at (2,3.4) {\tiny $\hat W^{\u}_{loc}(0,\eta)$};
\draw[-, draw=blue] (2.8,2.8)--(2.8,-0.0);
\draw[-, draw=red] (2.8,2.8)--(-0.0,2.8);
\node at (-0.1,3.35) {\tiny $\eta$};
\node at (0.56,3.43) {\tiny $(\xi(\eta,T),\eta')$};
\node at (0.085,3.3) {\tiny $\bullet$}; \node at (2.8,0.085) {\tiny $\bullet$};
\draw[-, draw=black] (0.085,3.3) .. controls (0.25,0.5) and (0.5,0.25) .. (2.8,0.085);
\node at (3.3,0.095) {\tiny $(\zeta_0',\omega(\eta,T))$};
\end{tikzpicture}
\caption{The first quadrant of the rectangle $\hat P_0$, with the integral curve
$I(T)$ connecting $(\xi(\eta,T),\eta) \in \hat W^{\u}_{loc}(0,\eta))$ and 
$(\zeta_0',\omega(\eta,T)) \in \hat  W^{\s}(\zeta'_0,\omega(\eta,T))$.
}
\label{fig:leaves3}
\end{center}
\end{figure}
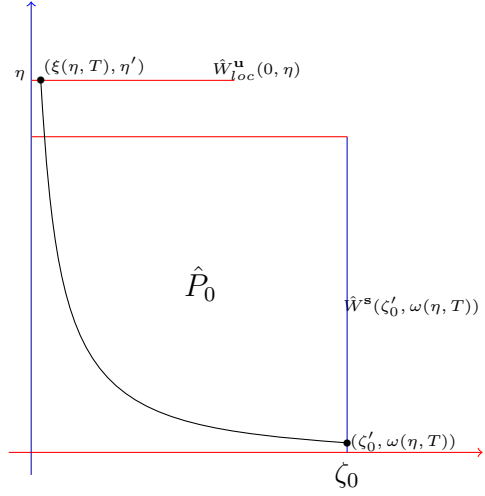 
For simplicity of notation, we will suppress the $\eta$ and $T$ in $\xi(\eta,T)$ and $\omega(\eta,T)$.
For $0 \leq t \leq T$, let $\Phi_t^{hor}(\xi,\eta) = (x(t), y(t))$,
so $(x(0)), y(0)) = (\xi,\eta)$ and $(x(T), y(T)) = (\zeta_0, \omega)$.
We need to compute
$$
\Theta(T) = \int_0^T\theta(x(t), y(t)) \, dt,
$$
where $\theta:\R^2 \to \R$ is homogeneous of exponent $\rho$.

{\bf New coordinates:}
We compute $\Theta(T)$ by introducing new coordinates $(L,M)$ where $L$ is the first integral from \eqref{eq:lf}
(and hence independent of $t$), and  $M(t) = y(t)/x(t)$, so $y(t) = M(t)x(t)$. 
For simplicity of notation, we suppress $t$ in $x,y$ and $M$.
Differentiating to  $\dot y = \dot M x + M \dot x(t)$, and inserting the values for $\dot x$ and $\dot y$ from \eqref{eq:polvf}, we get
\begin{equation}\label{eq:M0}
\dot M = -M( c_0 + c_2 M^2) x^2.
\end{equation}
We carry out the proof for $\Delta > 0$ (the case $\Delta < 0$ goes likewise), so 
the first integral is $L(x,y) =  x^u y^v ( \frac{a_0}{v}\, x^2  + \frac{b_2}{u}\, y^2)$ in \eqref{eq:lf}.
Since $(x,y)$ is in the level set 
$L(x,y) =  L(\xi,\eta) = \xi^u \eta^v(\frac{a_0}{v} \xi^2  + \frac{b_2}{u}\eta^2)$, 
we can solve for $x^2 $ in the expression
$$
\xi^u\eta^v\left(\frac{a_0}{v} \xi^2  + \frac{b_2}{u}\eta^2 \right) = x^u y^v \left(\frac{a_0}{v} x^2  + \frac{b_2}{u} y^2 \right) 
= x^{u+v+2 } M^v \left(\frac{a_0}{v} + \frac{b_2}{u} M^2\right).
$$
Use \eqref{eq:uv} and \eqref{eq:gamma} to obtain
$$
\frac{a_0}{v} + \frac{b_2}{u} M^2  = \frac{\Delta}{2 c_0c_2} ( c_0+c_2M^2 )
\quad \text{ and }\quad 
\frac{a_0 \xi^2}{v}  + \frac{b_2 \eta^2 }{u} = \frac{\Delta}{2 c_0c_2} ( c_0 \xi^2 +c_2 \eta^2 ).
$$
Note also from \eqref{eq:gamma} that $\frac{2v}{u+v+2} = \frac{1}{\beta_0}$,  $\frac{2u}{u+v+2} = \frac{1}{\beta_2}$
and $\frac{2}{u+v+2} = 1-\frac{1}{2\beta_0} - \frac{1}{2\beta_2}$.
This plus the previous two equations together gives
\begin{equation}\label{eq:x}
 x^2 = G(T) M^{-\frac{1}{\beta_0} } (c_0+c_2M^2)^{\frac{1}{2\beta_0}+\frac{1}{2\beta_2}-1},
 \end{equation}
 with 
 \begin{equation}\label{eq:G(T)}
 G(T) := G(\xi(\eta, T), \eta) = \xi(\eta, T)^{\frac{1}{\beta_2}} \eta^{\frac{1}{\beta_0}}
 (c_0 \xi(\eta,T)^2 + c_2 \eta^2)^{1-\frac{1}{2\beta_0} - \frac{1}{2\beta_2}}.
 \end{equation} 
We claim that $G(T) \sim G_0 T^{-1}$ as $T \to \infty$ for some $G_0 > 0$.

{\bf Proving the claim:} First note that
\begin{equation}\label{eq:M0I}
T = \int_0^T dt = \int_{M(0)}^{M(T)} \, \frac{dM}{\dot M} = -\int_{M(T)}^{M(0)} \, \frac{dM}{\dot M}.
\end{equation}
Recall that $M(0) = \eta/\xi$ and $M(T) = \omega/\zeta_0$.
Inserting this and \eqref{eq:M0} and \eqref{eq:x} into \eqref{eq:M0I} gives
\begin{equation}\label{eq:M4}
\int_{\omega/\zeta_0}^{\eta/\xi} \frac{dM}
{  M^{1-\frac{1}{\beta_0}}  \left(c_0+c_2 M^2  \right)^{ \frac{1}{2 \beta_0} + \frac{1}{2 \beta_2} } }
= G(T) T.
\end{equation}
As $T$ increases, the integral curve connecting $(\xi(\eta,T), \eta)$ to $(\zeta_0,\omega(\eta,T))$
tends to the union of the local stable and unstable manifolds of $(0,0)$,
whilst $M(T) = \omega(\eta,T)/\zeta_0 \to 0$ and $M(0) = \eta/\xi(\eta,T) \to \infty$.
From their definition, $\xi(\eta,T)$ and $\omega(\eta,T)$ are decreasing in $T$,
so their $T$-derivatives $\xi'(\eta,T), \omega'(\eta,T) \leq 0$. 

Since $c_0, c_2 > 0$ (otherwise $\Delta = 0$), the integrand of \eqref{eq:M4} is $O(M^{\frac{1}{\beta_0}-1})$ as $M \to 0$ and 
$O(M^{-\frac{1}{\beta_2}-1})$ as $M \to \infty$. Hence the integral is increasing and bounded in $T$.
But this means that $G(T) T$ is increasing in $T$ and bounded as well.
Since by \eqref{eq:G(T)}
$$
G(T) T  = \left(\xi(\eta,T) T^{\beta_2} \right)^{\frac{1}{\beta_2}} 
\eta^{\frac{1}{\beta_0}} (c_0 \xi(\eta,T)^2 + c_2 \eta^2 )^{1-\frac{1}{2 \beta_0} - \frac{1}{2 \beta_2}},
$$
we find by combining with \eqref{eq:M4}  that 
\begin{eqnarray*}
\xi_0(\eta) := \lim_{T\to\infty}  \xi(\eta,T) T^{\beta_2} 
&=& \lim_{T \to \infty} (G(T)T)^{\beta_2}
\eta^{-\frac{\beta_2}{\beta_0}} (c_0 \xi(\eta,T)^2 + c_2 \eta^2 )^{-\beta_2(1-\frac{1}{2 \beta_0} - \frac{1}{2 \beta_2})} \\
&=& c_2^{-\frac1u} \eta^{- \frac{a_2}{b_2} } \left( \int_0^\infty \frac{ dM }
{ M^{1-\frac{1}{\beta_0}}  \left( c_0 + c_2 M^2  \right)^{ \frac{1}{2 \beta_0} + \frac{1}{2 \beta_2} } }
\right)^{\beta_2},
\end{eqnarray*}
where we have used $-\beta_2(1-\frac{1}{2 \beta_0} - \frac{1}{2 \beta_2}) = -\frac{2 \beta_2}{u+v+2 } = -\frac{1}{u}$
for the exponent of $c_2$, and $-\frac{\beta_2}{\beta_0} - 2\beta_2(1-\frac{1}{2\beta_0}-\frac{1}{2\beta_2}) 
= 1-2\beta_2 = -\frac{a_2}{b_2}$ for the exponent of $\eta$.
This shows that
\begin{equation}\label{eq:G(T)2}
G(T) \sim c_2^{1-\frac{1}{2\beta_0}-\frac{1}{2\beta_2}} 
\xi_0(\eta)^{\frac{1}{\beta_2}} \eta^{2-\frac{1}{\beta_2}}\, T^{-1} =: G_0 T^{-1} \quad \text{ as } T \to \infty.
\end{equation}
 
{\bf Estimating the integral $\Theta(T)$:}
Recalling $M = y/x$ and using homogeneity of $\theta$, we have $\theta(x,y) = x^\rho \theta(1,M)$, and 
$\theta_0 := \theta(1,0)$ and $\theta_\infty := \theta(0,1) = \lim_{M \to \infty} M^{-\rho} \theta(1,M)$
are non-zero by assumption.
Inserting the above into the integral of \eqref{eq:M4}, and using \eqref{eq:x} to rewrite $x$, we obtain
\begin{equation}\label{eq:Theta-integral}
\Theta(T) = G(T)^{\frac{\rho}{2}-1} 
\int_{\omega/\zeta_0}^{\eta/\xi} 
\frac{ M^{-\frac{\rho}{2\beta_0}} (c_0+c_2M^2)^{-\frac{\rho}{2}(1-\frac{1}{2\beta_0}-\frac{1}{2\beta_0})} 
\theta(1,M) }
{  M^{1-\frac{1}{\beta_0}}  \left(c_0+c_2 M^2  \right)^{ \frac{1}{2\beta_0} + \frac{1}{2\beta_2} } }
\, dM.
\end{equation}
Set $\rho_0 := \frac{\rho}{2} - (1+\frac{\rho}{2}) (\frac{1}{2\beta_0} + \frac{1}{2\beta_2})$.
The leading terms in the integrand of \eqref{eq:Theta-integral} are
\begin{equation}\label{eq:lead12}
\begin{cases}
\theta_0 c_0^{\rho_0} M^{-1+(1-\frac{\rho}{2})\frac{1}{\beta_0}} & \text { as } M \to 0;\\
\theta_\infty c_2^{\rho_0} M^{-1-(1-\frac{\rho}{2})\frac{1}{\beta_2}} & \text { as } M \to \infty.
\end{cases}
\end{equation}

{\bf The case $\rho < 2$:}
By \eqref{eq:lead12}, we have for $\rho < 2$ that the exponent of $M$ is $> -1$ as $M \to 0$
and $< -1$ as $M \to \infty$.
This means that the integral in \eqref{eq:Theta-integral} converges to some constant 
$$
C^* = \int_0^\infty
\frac{ M^{-\frac{\rho}{2\beta_0}} (c_0+c_2M^2)^{-\frac{\rho}{2}(1-\frac{1}{2\beta_0}-\frac{1}{2\beta_0})} \theta(1,M) }
{  M^{1-\frac{1}{\beta_0}}  \left(c_0+c_2 M^2  \right)^{ \frac{1}{2\beta_0} + \frac{1}{2\beta_2} } }
\, dM
$$
as $T \to \infty$,
and $\Theta \sim C^*  G(T)^{\frac{\rho}{2}-1}  \sim C_\rho T^{1-\frac{\rho}{2}}$
for $C_\rho = \left( c_2^{1-\frac{1}{2\beta_0} - \frac{1}{2\beta_2}} \xi_0(\eta)^{\frac{1}{\beta_0}} \eta^{2-\frac{1}{\beta_2}} \right)^{\frac{\rho}{2}-1} C^*$
by \eqref{eq:G(T)}.
This finishes the proof for $\rho < 2$.

{\bf The case $\rho > 2$:} The value of $\Theta(T)$ based on the leading terms \eqref{eq:lead12}
of the integrand only is
$$
\Theta(T) \sim  \frac{G(T)^{\frac{\rho}{2}-1}}{\frac{\rho}{2}-1} \left(
\beta_0 \theta_0 c_0^{\rho_0} \left(\frac{\omega}{\zeta_0} \right)^{(1-\frac{\rho}{2})\frac{1}{\beta_0}} 
+
\beta_2 \theta_\infty c_2^{\rho_0} 
\left(\frac{\eta}{\xi} \right)^{-(1-\frac{\rho}{2})\frac{1}{\beta_2}}  \right).
$$
Insert $\xi = \xi_0(\eta) T^{-\beta_2} (1+o(1))$ and $\omega = \omega_0(\eta) T^{-\beta_0} (1+o(1))$ from \cite[Proposition 2.1]{BT17}:
$$
\Theta(T) \sim  \frac{ G(T)^{\frac{\rho}{2}-1}}{\frac{\rho}{2}-1} \left(
\beta_0 \theta_0 c_0^{\rho_0} 
\left(\frac{\zeta_0}{\omega_0(\eta)} \right)^{\frac{\rho/2-1}{\beta_0}} 
+
\beta_2 \theta_\infty c_2^{\rho_0} 
\left(\frac{\eta}{\xi_0(\eta)} \right)^{\frac{\rho/2-1}{\beta_2}} \right)
T^{\frac{\rho}{2}-1}.
$$
Next, by inserting the asymptotics of $G(T)$ from \eqref{eq:G(T)2}, the factor $T^{\frac{\rho}{2}-1}$ cancels:
$$
\Theta(T) \sim C_\rho := \frac{ G_0^{\frac{\rho}{2}-1}}{\frac{\rho}{2}-1} 
 \left(
\beta_0 \theta_0 c_0^{\rho_0} 
\left(\frac{\zeta_0}{\omega_0(\eta)} \right)^{\frac{\rho/2-1}{\beta_0}} 
+
\beta_2 \theta_\infty c_2^{\rho_0} 
\left(\frac{\eta}{\xi_0(\eta)} \right)^{\frac{\rho/2-1}{\beta_2}} 
\right).
$$

{\bf The case $\rho = 2$:} The factor $G(T)^{\frac{\rho}{2}-1}$ in \eqref{eq:Theta-integral} now
disappears and the leading terms \eqref{eq:lead12}
are $\theta_0 c_0^{\rho_0} M^{-1}$ and $\theta_\infty c_2^{\rho_0} M^{-1}$ respectively.
This gives
$$
\Theta(T) \sim \int_{\omega/\zeta_0}^{\eta/\xi} \frac{\theta_0 c_0^{\rho_0}+\theta_\infty c_2^{\rho_0}}{M} \, dM
=  \left( \theta_0 c_0^{\rho_0}+\theta_\infty c_2^{\rho_0} \right) 
\left(\log \frac{\eta}{\xi} - \log \frac{\omega}{\zeta_0}\right).
$$
Inserting again the values of $\xi$ and $\omega$ from \cite[Proposition 2.1]{BT17} gives
$$
\Theta(T) \sim  \left( \theta_0 c_0^{\rho_0}+\theta_\infty c_2^{\rho_0} \right) 
(\beta_0+\beta_2) \log T \quad \text{ as } T \to \infty.
$$
This completes the proof.
\end{pfof}

\subsection{Description via suspension flows, tail estimates of the roof function}
\label{subsec-suspAn}
The $3$-dimensional time-$1$ map $\Phi_1$ preserves no $2$-dimensional submanifold of $\cM$. 
Yet in order to model $\Phi_t$ as a suspension flow over a $2$-dimensional map, we need a genuine
Poincar\'e map.
For this we choose a section $\Sigma$ transversal to $\Gamma$ and containing a neighbourhood $U$ of $p$.
As an example, $\Sigma$ could be $\T^2 \times \{ 0 \}$, and the Poincar\'e map to $\T^2 \times \{ 0 \}$ 
could be (a local perturbation of) Arnol'd's cat map; in this case (and most cases)
$\cM$ is not homeomorphic to $\T^3$ because the homology is more complicated, see \cite{BF13, N76}.

Let $h:\Sigma \to \R^+$, $h(q) = \min\{ t > 0 : \Phi_t(q)  \in \Sigma\}$ be the first return time.
Assuming that $\sup_\Sigma |w(x,y)| < 1$, the
first return time $h$ is bounded and bounded away from zero, i.e.,
$0 < \inf_{\Sigma} h < \sup_{\Sigma} h$. There is no loss of generality in assuming that $\inf_{\Sigma}h \geq 1$.

The Poincar\'e map $f := \Phi_h: \Sigma \to \Sigma$ 
has a neutral saddle point $p$ at the origin. Its local stable/unstable manifolds
are $W^{\s}_{loc}(p) = \{ 0 \} \times (-\eps,\eps)$ and 
$W^{\u}_{loc}(p) = (-\eps,\eps) \times \{ 0 \}$.
Because the flow $\Phi_t$ is a perturbation of an Anosov flow, and $f$ is 
a Poincar\'e map, it has a finite Markov partition $\{P_i\}_{i \geq 0}$
and we can assume that $p$ is in the interior of $P_0$.
In the sequel, let $U$ be a neighbourhood of $p$ that is small enough that 
\eqref{eq:polvf} is valid on $U \times [0,1]$ but also that $f(U) \supset \hat P_0 \cup P_0$.

In order to regain the hyperbolicity lacking in $f$, let 
\begin{equation}\label{eq:firstreturn}
\rf(q) := \min\{ n \geq 1 : f^n(q) \in Y \}
\end{equation}
be the first return time to $Y := \Sigma \setminus P_0$.
Then the Poincar\'e map $F = f^{\rf} = \Phi_\tau$ of $\Phi_t$ to
$Y \times \{ 0 \}$ is hyperbolic (see \cite[Section 7]{Hu00}), where
\begin{equation*}
\tau(q) = \min\{ t > 0 : \Phi_t((q,0)) \in Y \times \{ 0 \} \} = \sum_{j=0}^{\rf-1} h \circ f^j 
\end{equation*}
is the corresponding first return time. 

Consequently, the flow $\Phi_t:\cM \times \R \to \cM$ can be modelled as a suspension flow
on $Y^\tau = \left( \bigcup_{q \in Y} \{ q \} \times [0,\tau(q)) \right)/(q,\tau(q)) \sim (F(q),0)$.
Since the flow and section $Y \times \{ 0 \}$ are $C^1$ smooth,
$\tau$ is $C^1$ on each piece $\{\rf = k\}$.

\begin{lemma}\label{lem:tau}
In the notation of Proposition~\ref{prop:w-integral} with $\theta = w$ from the $z$-component in \eqref{eq:polvf}, we have
$\tau(q) = \hat\tau(q) + O(1)$, $\rf = \hat\tau(q)+\Theta(\hat\tau(q))+O(1)$
and $\Theta(\hat\tau(q)) = o(\hat\tau(q))$ as $q \to 0$.
\end{lemma}

\begin{proof}
By the definition of $\hat\tau$ we have $\Phi_{\hat\tau}^{hor}(q) \in \hat W^{\s}$. Therefore it takes a bounded amount of
 time (positive or negative) for $\Phi_{\hat\tau}(q,0)$ to hit $Y \times \{ 0 \}$, so $|\tau(q)-\hat\tau(q)| = O(1)$.
 
If in \eqref{eq:w-integral} we set $\theta = w$, then $\hat\tau(q) + \Theta(\hat\tau(q))$ indicates the vertical 
displacement under the flow $\Phi_t$. In particular, it gives the number of times the flow-line intersects 
$\Sigma$, and hence $\rf = \hat\tau(q) + \Theta(\hat\tau(q)) + O(1)$. 
\end{proof}

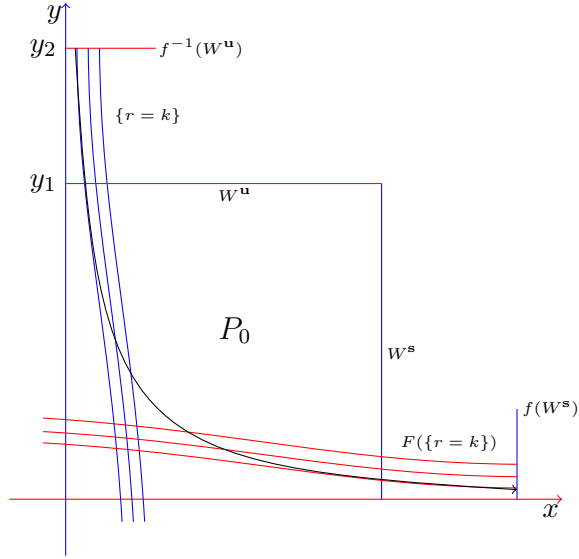
\begin{figure}[ht]
\begin{center}
\begin{tikzpicture}[scale=1.5]
\node at (4.3,-0.1) {\small $x$}; 
\node at (-0.1,4.3) {\small $y$}; \node at (-0.2,2.8) {\small $y_1$};\node at (-0.2,4) {\small $y_2$};
\draw[->, draw=red] (-0.5,0)--(4.4,0);
\draw[->, draw=blue] (0,-0.5)--(0,4.4);
\node at (0.73,3.4) {\tiny $\{ \rf = k\}$};
\node at (3.4,0.5) {\tiny $F(\{\rf = k\})$};
\node at (1.5,1.5) {$P_0$};
\node at (1.5,2.7) {\tiny $W^{\u}$}; \node at (1.2,4) {\tiny $f^{-1}(W^{\u})$};
\node at (2.98,1.3) {\tiny $W^{\s}$}; \node at (4.3,0.8) {\tiny $f(W^{\s})$};
\draw[-, draw=red] (0,4)--(0.8,4); \draw[-, draw=blue] (4,0)--(4,0.8); 
\draw[-, draw=blue] (0.1,4) .. controls (0.11,2.5) and (0.4,1.5) .. (0.5,-0.2);
\draw[-, draw=blue] (0.2,4) .. controls (0.21,2.5) and (0.5,1.5) .. (0.6,-0.2);
\draw[-, draw=blue] (0.3,4) .. controls (0.31,2.5) and (0.6,1.5) .. (0.7,-0.2);
\draw[-, draw=blue] (2.8,2.8)--(2.8,-0.0);
\draw[-, draw=red] (4, 0.1) .. controls (2.5, 0.11) and (1.5, 0.4) .. (-0.2, 0.5);
\draw[-, draw=red] (4, 0.2) .. controls (2.5, 0.21) and (1.5, 0.5) .. (-0.2, 0.6);
\draw[-, draw=red] (4, 0.31) .. controls (2.5, 0.32) and (1.5, 0.61) .. (-0.2, 0.72);
\draw[-, draw=red] (2.8,2.8)--(-0.0,2.8);
\draw[->, draw=black] (0.085,4) .. controls (0.3,0.5) and (0.5,0.3) .. (4,0.085);
\end{tikzpicture}
\caption{The first quadrant of the rectangle $P_0$, with stable and unstable foliations of 
Poincar\'e map $f = \Phi_h$
drawn vertically and horizontally, respectively. Also one of the integral curves is drawn.}
\label{fig:leaves2}
\end{center}
\end{figure} 

Let $\phi:Y^\tau \to \R$, $\phi(q) = \lim_{t\to 0} -\frac1t \log | D\Phi_t|_{W^{\u}}(q)|$ 
be the potential for the SRB measure of the flow.
Let $\mu_{\bar\phi}$ be the $F$-invariant equilibrium measure of the potential $\bar\phi:Y \to \R$,
$\bar\phi(q) = \int_0^{\tau(q)} \phi \circ \Phi_t(q) \, dt= -\log | D F|_{W^{\u}}(q)|$; so $\bar\phi$ is the
logarithm of the derivative in the unstable direction of the Poincar\'e map $F$.
This is at the same time the SRB-measure for $F$ and thus is absolutely continuous
conditioned to unstable leaves.

\begin{prop}\label{prop:tailtau}
Recall that $\beta = \frac{a_2+b_2}{2b_2} \in (\frac12, \infty)$.
There exists $C^* > 0$ such that
\begin{equation}\label{eq:asymp3}
\mu_{\bar \phi}(\{ \tau > t \}) = C^* t^{-\beta} (1+o(1))
\end{equation}
for the $F$-invariant SRB measure $\mu_{\bar\phi}$.
\end{prop}

\begin{proof}
The function $\tau$ is defined on $\Sigma \setminus P_0$ and $\tau \geq h_2 = h + h \circ f$ on
$Y_{\{\rf \geq 2\}} := f^{-1}(P_0) \setminus P_0$.
The set $Y_{\{\rf \geq 2\}}$ is a rectangle with boundaries consisting of two 
stable and two unstable leaves of the Poincar\'e map  $f$.
Let $W^{\u}(y)$ denote the unstable leaf of $f$ inside $Y_{\{\rf \geq 2\}}$
with $(0,y)$ as (left) boundary point.
Let $y_1 < y_2$ be such that $W^{\u}(y_1)$ and $W^{\u}(y_2)$ are the unstable
boundary leaves of $Y_{\{\rf \geq 2\}}$.

The unstable foliation of $\hat f = \Phi_1^{hor}$ does not entirely coincide with the unstable foliation of $f$.
Let $\hat W^{\u}(y)$ denote the unstable leaf of $\hat f$ 
with $(0,y)$ as (left) boundary point. Both $\hat W^{\u}(y)$ and $W^{\u}(y)$ are $C^1$ curves emanating from $(0,y)$; 
let $\bgamma(y)$ denote the angle between them. Then the lengths
\begin{eqnarray*}
\mbox{Leb}(W^{\u}(y) \cap \{ \tau > t\}) &=& |\cos \bgamma(y)|\ \mbox{Leb}(\hat W^{\u}(y) \cap \{ \tau > t\}) (1+o(1)) \\
&=& |\cos \bgamma(y)|\ \xi_0(y)\ t^{-\beta} (1+o(1))
\end{eqnarray*}
as $t\to\infty$, where the last equality and the notation $\xi_0(y)$ and  $\beta = (a_2+b_2)/(2b_2)$ come
from \eqref{eq:asymp0}.

We decompose the measure $\mu_{\bar\phi}$ on $Y_{\{\rf \geq 2\}}$ as 
$$
\int_{Y_{\{\rf \geq 2\}}} v\, d\mu_{\bar\phi} = 
\int_{y_1}^{y_2} \left( \int_{W^{\u}(y) } v\, d\mu^{\s}_{W^{\u}(y)} \right) d\nu^{\u}(y).
$$
The conditional measures $\mu_{W^{\u}(y)}$ on $W^{\u}(y)$
are equivalent to Lebesgue $m_{W^{\u}(y)}$ with density $h_0^{\u} = \frac{d\mu_{W^{\u}(y)}}{dm_{W^{\u}(y)}}$ tending to a 
constant $h_*(y)$ at the boundary point $(0,y)$.
Therefore, as $t \to \infty$,
\begin{eqnarray*}
 \mu_{\bar\phi}(\tau > t) &=& \int_{y_1}^{y_2} \mu_{W^{\u}(y)}(W^{\u}(y) \cap \{ \tau > t\} ) \, d\nu^{\u}(y) \\
  &=& \int_{y_1}^{y_2} h_0^{\u}\ m_{W^{\u}(y)}( W^{\u}(y) \cap \{ \tau > t\} ) \, d\nu^{\u}(y) \\
   &=& \int_{y_1}^{y_2} h_0^{\u} |\cos \bgamma(y)|\ m_{\hat W^{\u}(y)}(\hat W^{\u}(y) \cap \{ \tau > t\} ) (1+o(1)) \, d\nu^{\u}(y) \\
    &=& \int_{y_1}^{y_2} h_0^{\u} |\cos \bgamma(y)|\ \xi_0(y)\ t^{-\beta} (1+o(1)) \, d\nu^{\u}(y) 
    = C^* t^{-\beta}(1+o(1)),
\end{eqnarray*}
for $C^* = \int_{y_1}^{y_2} h_*(y) |\cos \bgamma(y)| \ \xi_0(y) \, d\nu^{\u}(y)$,
and using \eqref{eq:asymp0} in the third line.
This proves the result.
\end{proof}

\subsection{Main results for the Poincar{\'e} map $F$ and flow $\Phi_t$}
\label{subsec-mainres}

Throughout this section we assume the setup and notation of Subsection~\ref{subsec-suspAn}.
We emphasise that we are in the finite measure setting, so 
\begin{equation}\label{eq:tau*} 
 \tau^* := \int_Y \tau \, d\mu_{\bar\phi} < \infty.
\end{equation}
We recall that the natural potential associated to the SRB-measure for $F$ is
$\bar\phi = -\log | D F|_{W^{\u}}$, which is the induced version of $\phi = \lim_{t\to 0}\frac{-\log | D\Phi_t|_{W^{\u}}}{t}$.
Our main result 
can be viewed as a version of the results in~\cite{Sar06} for the flow $\Phi_t$; it gives gives the link 
between limit theorems for $(\Phi_t,\psi)$ for (unbounded from below) potentials $\psi$ on $\cM$
and pressure function $\mathcal P(\phi+s\psi), s\ge 0$.  
We let $\overline{\phi+s \psi}, s\ge 0$ be the family of  induced potentials  and denote the associate pressure 
function by  $\cP(\overline{\phi+s\psi})$.
Some background on equilibrium states and pressure function (along with their induced versions)
is recalled in Section~\ref{sec-back}.

\begin{thm}
\label{thm-conclAn} 
Suppose $\bar\psi(y) := \int_0^\tau \psi \circ \Phi_t \, dt =C'-\psi_0$ where $C'>0$ and $\psi_0$ is 
positive piecewise $C^1(Y)$.  
Let $\beta > 1$ be as in Proposition~\ref{prop:tailtau},
Furthermore, suppose that 
$\mu_{\bar\phi}(\psi_0 > t) \sim ct^{-\frac\beta\kappa}$ 
for some $c>0$ and $\kappa  \in (1/\beta, \beta)$.
Set $\psi^* =\int_Y \bar \psi\,d \mu_{\bar\phi}$ and let $\psi_T= \int_0^T \psi \circ \Phi_t \, dt$.
The following hold as $T\to\infty$, w.r.t.\, $\mu_{\bar\phi}$.

\begin{itemize}
\item[(a)] 

\begin{itemize}
\item[(i)] When $\beta<2$ and $\kappa\in (\beta/2,\beta)$,  set $b(T)$ such that $\frac{T}{(c b(T))^{\frac{\beta}{\kappa}}}\to 1$.
 
Then  $\frac{1}{b(T)}(\psi_T-\psi^*\cdot T)\to^d G_{\beta/\kappa}$,  where $G_{\beta/\kappa}$ is a stable law of index $\beta/\kappa$
and this is, further, equivalent to $\cP(\overline{\phi+s\psi})=\psi^*s+c s^{\beta/\kappa}(1+o(1))$, as $s\to 0$. 

\item [(ii)] When $\beta\leq 2$ but $\kappa=\beta/2$, set $b(T)\sim c^{-1/2} \sqrt{T\log T}$.

Then $\frac{1}{b(T)}(\psi_T-\psi^*\cdot T)\to^d \cN(0,1)$ and this is, further, equivalent to 
$\cP(\overline{\phi+s\psi})=\psi^* s+c s^2 \log(1/s)(1+o(1))$, as  $s\to 0$. 
\end{itemize}

\item[(b)] Suppose that $\beta>2$. Then there exists $\sigma\ne 0$ such that 
$\frac{1}{\sqrt T}(\psi_T- \psi^*\cdot T)\to^d \cN(0,\sigma^2)$ and this is, further, equivalent to $\cP(\overline{\phi+s\psi})=\psi^* s+\frac{\sigma^2}{2}s^2 (1+o(1))$ 
as  $s\to 0$. 
\end{itemize}
Moreover, $\cP(\phi+s\psi)=\frac{1}{\tau^*}\cP(\overline{\phi+s\psi}) (1+o(1))$.~\end{thm}

\begin{rmk}\label{rem:psi}  
If $\psi:\cM \to \R$ satisfies the following properties, then the condition on 
$\bar\psi$ in Theorem~\ref{thm-conclAn} holds. In local tubular coordinates
near $\{ p \} \times \T$, $\psi = g_0-g(W(x,y))$ where $g_0$ is H\"older function vanishing on a neighbourhood of $p$, and $W(x,y)$ is a linear combination of 
homogeneous functions depending only on $(x,y)$, such that the term $W_\rho(x,y)$ with lowest exponent $\rho$ satisfies
$W_\rho(0,1) \neq 0 \neq W_\rho(1,0)$, and $g:\R \to \R$ is analytic with $g(0) = 0$, $g'(0) \neq 0$. 
Then $\kappa = 1-\frac{\rho}{2}$ (see Proposition~\ref{prop:w-integral}), 
so assuming that this lowest
exponent $\rho \in (2(1-\beta), 2(\beta-1)/\beta)$, then $\kappa \in (\frac{1}{\beta},\beta)$.
\end{rmk}

\begin{pfof}{Theorems~\ref{thm-conclAn}}
This follow directly
from applying Theorem~\ref{thm-concl} $(\Phi_t,\phi)$, which uses a special 
case of Proposition~\ref{prop-limthF}.
Appendix~\ref{sec-ver} verifies all the abstract hypotheses of Proposition~\ref{prop-limthF} and Theorem~\ref{thm-concl} 
for the flow $\Phi_t$ with base map $F$ and roof function $\tau$ introduced in Section~\ref{sec-Aaflow}
and potential $\psi$ defined in the statement of Theorem~\ref{thm-conclAn}. The extra assumptions (including coboundary assumptions) 
that ensure $\sigma\ne 0$ in Proposition~\ref{prop-limthF}(ii) and Theorem~\ref{thm-concl} (b) are verified in 
Appendix~\ref{subsec-coboundary}.~\end{pfof}

\section{Background on equilibrium states for suspension flows}
\label{sec-back}

Here we outline the standard general theory of thermodynamic formalism for suspension flows. 
We start with a discrete time dynamical system $F:Y\to Y$.  Defining $\cM_F$ to be the set of $F$-invariant 
probability measures, for a potential $\psi:Y\to [-\infty, \infty]$ we define the pressure as
$$\cP(\psi):=\sup\left\{h(\mu)+\int\psi~d\mu:\ \mu\in \cM_F \text{ and }  -\int\psi~d\mu<\infty\right\}.$$
Given a roof function $\tau:Y\to \R^+$, let $Y^\tau=\{(y,z)\in Y\times\R: 0\le z\le \tau(y)\}/\sim$
where $(y,\tau(y))\sim(Fy,0)$.
Then the suspension flow $F_t:Y^\tau\to Y^\tau$ is given by
$F_t(y,z)=(y,z+t)$ computed modulo identifications.
 We will suppose throughout that $\inf\tau>0$, but note that the case $\inf\tau=0$ can also be handled, 
 see for example \cite{Sav98,IJT15}.  Barreira and Iommi \cite{BarIom06} define pressure as
\begin{equation}
\cP(\phi):=\inf\{u\in \R:\cP(\bar\phi-u\tau)\le 0\}\label{eq:flow pres}
\end{equation}
in the case that $(Y, F)$ is a countable Markov shift and the  potential $\phi:Y^\tau\to \R$ induces to  
$ \bar\phi:Y\to \R$, where 
 \begin{equation}\bar\phi(x):=\int_0^{\tau(x)}\phi(x, t)\, dt,
 \label{eq:ind pot}
 \end{equation}
has summable variations. This also makes sense for general suspension flows, 
so we will take \eqref{eq:flow pres} as our definition.
In \cite{AmbKak42}, there is a bijection between $F_t$-invariant probability measures $\mu$, and the  
corresponding $F$-invariant probability measures $\bar\mu$ given by the identification 
$$
\mu=\frac{\bar\mu\times m}{(\bar\mu\times m)(Y^\tau)},
$$
where $m$ is Lebesgue measure. That is, whenever there is such a $\mu$ there is such a $\bar\mu$, and vice versa. 
Moreover, Abramov's formula for flows \cite{Abr59b} gives the following characterisation of entropy:
\begin{equation}
\label{eq-Abr}
h(\mu)=\frac{h(\bar\mu)}{(\bar\mu\times m)(Y^\tau)},
\end{equation}
and clearly 
\begin{equation}\label{eq-projmeas}
\int\phi~d\mu=\frac{\int\bar\phi~d\bar\mu}{(\bar\mu\times m)(Y^\tau)} = \frac{\int\bar\phi~d\bar\mu}{\int \tau \, d\bar\mu}.
\end{equation}
One consequence of these formulas is that the pressure in \eqref{eq:flow pres} is independent of the choice of cross section $Y$.
This follows essentially 
from the fact that if we choose a subset of $Y$ and reinduce there, then 
Abramov's formula gives the same value for pressure (here we can use the discrete version of Abramov's formula \cite{Abr59a}).   
Note that we say `Abramov's formula' in both the continuous and discrete time cases as 
the formulas are analogous \cite{Abr59a, Abr59b}, and similarly for integrals, where the formula holds by the ergodic 
and ratio ergodic theorems.

We say that $\mu_\phi$ is an \emph{equilibrium state} for $\phi$ if $h(\mu_\phi)+\int\phi~d\mu_\phi=\mathcal P(\phi)$.  The same notion extends to the induced system $(Y, F, \bar\phi)$.  In that setting we may also have a \emph{conformal measure}  $m_{\bar\phi}$ for $\bar\phi$.  This means that  $m_{\bar\phi}(F(A))=\int_A e^{-\bar\phi}\, d m_{\bar\phi}$
for any measurable set $A$ on which $F$ is injective.
Later on, in order to link our limit theorem results with pressure, we will assume that $F:Y\to Y$ is Markov.  In that setting our assumptions on $\phi$ will be equivalent to assuming $\mathcal P(\phi)=0$, in which case we will have a equilibrium states and conformal measures $\mu_{\bar\phi}$ and $m_{\bar\phi}$.

\section{Abstract setup}
\label{sec-abstr}

We start with a flow $f_t:\cM \to \cM$, where $\cM$ is a manifold. Let $Y$ be a co-dimension $1$ section of $\cM$
and define $\tau:Y\to\R_{+}$ to be a roof function. We think of $\tau$ as a first return of $f_t$ to $Y$ and define $F:Y\to Y$ by $F=f_\tau$.
The flow $f_t$ is isomorphic with the suspension flow $F_t:Y^\tau\to Y^\tau$, 
 $Y^\tau=\{(y,z)\in Y\times\R: 0\le z\le \tau(z)\}/\sim$, as described in Section~\ref{sec-back}. Throughout, we assume that
$\tau$ is bounded from below.

Given the potential $\phi :\cM\to  \R$ and its induced version $ \bar\phi:Y\to \R$ defined in \eqref{eq:ind pot},  we assume that there is a conformal measure  $m_{\bar\phi}$
for $(F,\bar\phi)$. 
In the rest of this section we recall the abstract framework and hypotheses in~\cite{LT16} as relevant 
for studying limit theorems.

\subsection{Banach spaces and equilibrium measures for $(F,\bar\phi)$ and $(f_t,\phi)$. }
\label{subsec-Bsp}

Let $R_0:L^1(m_{\bar\phi})\to L^1(m_{\bar\phi})$ be the transfer operator for the first return 
map $F:Y\to Y$ w.r.t.\, conformal measure $m_{\bar\phi}$, defined by duality on $L^1(m_{\bar\phi})$ via the formula
$\int_Y R_0v\,w\,d m_{\bar\phi}= \int_Y v\,w\circ F\,d m_{\bar\phi}$ for every bounded measurable $w$.

We assume that
there exist two Banach spaces of distributions $\cB$, $\cB_w$ supported on $Y$ such that for some $\alpha,\alpha_1>0$

\begin{itemize}
\item[{\bf (H1)}\labeltext{{\bf (H1)}}{H1}] \begin{itemize}
\item[(i)] $C^\alpha\subset\cB\subset\cB_w\subset (C^{\alpha_1})'$, where 
$(C^{\alpha_1})'$ is the dual of $C^{\alpha_1} = C^{\alpha_1} (Y,\C)$.\footnote{We will use systematically a ``prime" to denote the topological dual.}
\item[(ii)]  $hg\in\cB$ for all $h\in\cB$, $g\in C^\alpha$.
\item[(iii)] The transfer operator $R_0$ associated with $F$ maps continuously from $C^\alpha$ to $\cB$ and
$R_0$ admits a continuous extension to an operator from $\cB$ to $\cB$, which we still call $R_0$.
\item[(iv)] The operator $R_0:\cB\to\cB$ has a simple eigenvalue at $1$
and the rest of the spectrum is contained in a disc of radius less than $1$.
\end{itemize}
\end{itemize}

A few comments on \ref{H1} are in order and here we just recall the similar ones in~\cite{LT16}.
We note that \ref{H1}(i) should be understood in terms of the following usual convention  (see, for instance,
\cite{GL06, DemersLiverani08}): there exist continuous injective linear maps $\pi_i$ such that $\pi_1(C^\alpha)\subset \cB$, $\pi_2(\cB)\subset \cB_w$ and $\pi_3(\cB_w)\subset (C^{\alpha_1})'$.
Throughout, we leave such maps implicit, but recall their meaning here. In particular, 
we assume that $\pi=\pi_3\circ\pi_2\circ \pi_1$ is the usual 
embedding, i.e., for all $h\in C^\alpha$ and  $g\in C^{\alpha_1}$
$$
\langle \pi(h),g\rangle_0=\int_Y h g\, d m_{\bar\phi}.
$$
Via the above identification, the conformal measure $m_{\bar\phi}$ can be identified with the constant function $1$ both in $(C^{\alpha_1})'$ and in $\cB$ 
(i.e., $\pi(1)=m_{\bar\phi}$). 
Also, by \ref{H1}(i), $\cB'\subset (C^\alpha)'$, hence the dual space can be naturally viewed as a space of distributions. Next, note that $\cB'\supset (C^{\alpha_1})''\supset C^{\alpha_1}\ni 1$, thus we have $m_{\bar\phi}\in\cB'$ as well.
Moreover, since $m_{\bar\phi}\in\cB$ and $\langle1,g\rangle_0=\langle g,1\rangle_0=\int g\, d\mu^{\s}_{\bar\phi}$, $m_{\bar\phi}$ can be viewed as the element $1$  of both spaces
$\cB$ and $(C^{\alpha_1})'$.

By \ref{H1}, the spectral projection $P_0$ associated with the eigenvalue $1$ is defined by
$P_0=\lim_{n\to\infty}R_0^n$. Note that by \ref{H1}(ii), for each $g\in C^\alpha$, $P_0 g\in\cB$ and
\[
\langle P_0g,1\rangle_0=m_{\bar\phi}(P_0 g)=\lim_{n\to\infty}m_{\bar\phi}(1\cdot R_0^n g)=m_{\bar\phi}(g)=\langle g,1\rangle_0.
\]
By  \ref{H1}(iv),  there exists a unique $\mu_{\bar\phi}\in\cB$ such that $R_0\mu_{\bar\phi}=\mu_{\bar\phi}$ 
and $\langle\mu_{\bar\phi},1\rangle =1$. Thus, $P_0g=\mu_{\bar\phi} \langle g,1\rangle_0$.
Also $R_0'm_{\bar\phi}=m_{\bar\phi}$ where $R_0'$ is dual operator acting on $\cB'$. Note that for any $g\in C^{\alpha_1}$,
\begin{align*}
|\langle \mu_{\bar\phi}, g\rangle_0|=|\langle P_01, g\rangle_0|=\left|\lim_{n\to\infty}R_0^n m_{\bar\phi}(g)\right|
= \lim_{n\to\infty}\left|m_{\bar\phi}(g\circ F^n)\right|\leq |g|_\infty.
\end{align*}
Hence $\mu_{\bar\phi}$  is a measure. For each $g\geq 0$, 
\[
\langle P_01, g\rangle_0=\lim_{n\to \infty} \langle R_0^n1, g\rangle=\lim_{n\to \infty} \langle 1, g\circ F^n\rangle_0\geq 0.
\]
It follows that $\mu_{\bar\phi}$ is a probability measure. 

Summarising the above, the eigenfunction associated with the eigenvalue $\lambda=1$
 is an invariant probability measure for $F$ given by $d\mu_{\bar\phi}=h_0 d m_{\bar\phi}$ where $P_01=h_0$. Using that  $\pi(1)=m_{\bar\phi}$ we also have  $P_01=\mu_{\bar\phi}$.
Under \ref{H1}, 
the standard theory summarised in Section~\ref{sec-back} applies $\cP(\bar\phi)=\log\lambda=0$. 
 So, under \ref{H1}, $\mu_{\bar\phi}$ is an equilibrium (probability) measure for $(F,\bar\phi)$. 
Further, via~\eqref{eq-projmeas}, 
\[
\mu_\phi=\frac{\mu_{\bar\phi}\times m}{(\mu_{\bar\phi}\times m)(Y^\tau)}
\]
gives an equilibrium (probability) measure for $(f_t, \phi)$. 

\subsection{ Transfer operator $R$ defined w.r.t. the equilibrium measure $\mu_{\bar\phi}$ }
\label{subsec-tropinv}
Given the equilibrium measure $\mu_{\bar\phi}\in\cB$, we 
let $R:L^1(\mu_{\bar\phi})\to L^1(\mu_{\bar\phi})$ be the transfer operator for the first return 
map $F:Y\to Y$ w.r.t.\, the invariant measure $\mu_{\bar\phi}$ given by
$\int_Y R v\,w\,d \mu_{\bar\phi}= \int_Y v\,w\circ F\,d \mu_{\bar\phi}$ for every bounded measurable $w$.

Under \ref{H1}(i)-(iv), we further assume
\begin{itemize}
\item[{\bf (H1)}(v)\labeltext{{\bf H1}(v)}{H1v}] The transfer operator $R$ maps continuously from $C^\alpha$ to $\cB$ and
$R$ admits a continuous extension to an operator from $\cB$ to $\cB$, which we still call $R$.
\end{itemize}

By \ref{H1}(iv) and \ref{H1v}, the operator $R:\cB\to\cB$ has a simple eigenvalue at $1$
and the rest of the spectrum is contained in a disc of radius less than $1$. 
While the spectra of $R_0$ and $R$ are the same,
the spectral projection $P=\lim_{n\to\infty}R^n$ acts differently on $\cB,\cB_w$. 
In particular, $P1=1$, while $P_01=\mu_{\bar\phi}$, $P_01=h_0$. 
Throughout the rest of this paper, for any $g\in C^\alpha$, we let
\[
\langle g, 1\rangle:=\langle g, P_0 1\rangle_0=\langle P_0 1, g\rangle_0=\int_Y g\, d \mu_{\bar\phi}
\]
and note that  $P_0g=h\langle g, 1\rangle_0$ and $Pg=\langle g, 1\rangle$.

\subsection{ Further assumptions on the transfer operator $R$}
\label{subsec-assumpop}
Given $R$ as defined in Subsection~\ref{subsec-tropinv}, for $u\ge 0$ and $\tau:Y\to\R_{+}$ we define the perturbed transfer operator
\[
\hat R(u) v := R(e^{-u\tau}v).
\]
By~\cite[Proposition 4.1]{BMT}, a general proposition on twisted transfer operators 
that holds independently of the specific properties of $F$, 
we can write for sufficiently small positive $u$,
\begin{equation}
\label{eq-RLapl}
\hat R(u) = r_0(u)\int_0^\infty M(t)\, e^{-ut}\, dt
\qquad \text{with} \qquad M(t) :=  R(\omega(t-\tau)),
\end{equation}
where $\omega:\R\to  [0, 1]$ is an integrable function with $\supp\omega \subset [-1,1]$ and $r_0$ is analytic 
such that $r_0(0)=1$.

\begin{rmk}\label{rem:r0}
In fact, $r_0(u)=1/\hat\omega(u)$ for $\hat\omega(u)=\int_{-1}^1 e^{-ut} \omega(t)\, dt$.
Therefore $\frac{d}{du}r_0(u)$ and $\frac{d}{du}r_0(u)$
are continuous functions, and for later use,
we set $\gamma_1 := \frac{d}{du} r_0(0)$, $\gamma_2 := \frac{d^2}{du^2} r_0(0)$.
\end{rmk}

Since it is short, for the reader's convenience we include the proof of~\eqref{eq-RLapl} 
(as in the proof of ~\cite[Proposition 4.1]{BMT}).

\begin{proof}
Let $\omega$ be an integrable function supported on $[-1, 1]$ 
such that $\int_{-1}^1 \omega(t)\, dt = 1$, and set $\hat\omega(u)=\int_{-1}^1 e^{-ut} \omega(t)\, dt$. Note that $\hat\omega(u)$ is analytic and $\hat\omega(0)=1$.
Since $\tau \geq h \ge 1$ (see the assumption on $h$ in Section~\ref{subsec-suspAn}) and $\supp\omega\subset [-1,1]$,
\[
\int_0^\infty \omega(t-\tau) e^{-ut}\, dt=e^{-u\tau}\int_{-\tau}^\infty \omega(t)\, e^{-ut}\, dt=e^{-u\tau}\hat\omega(u).
\]
Hence,
\[
\int_0^\infty R(\omega(t-\tau)v) e^{-ut}\, dt=R\left(\int_0^\infty \omega(t-\tau) v e^{-ut}\, dt\right)=\hat\omega(u)\hat R(u) v.
\]
Formula~\eqref{eq-RLapl} follows with $r_0(u)=1/\hat\omega(u)$, so $r_0(0) = 1$.
\end{proof}

For most results we require that
\begin{itemize}
\item[{\bf (H2)}\labeltext{{\bf (H2)}}{H2}]
There exists a function $\omega$ satisfying \eqref{eq-RLapl} and $C_\omega > 0$ such that

\begin{itemize}
\item[(i)] for any $t\in\R_{+}$ and for all $v\in C^\alpha$, $h\in\cB$, we have $\omega(t-\tau)h \in\cB_w$ and
\[
\langle \omega(t-\tau)vh, 1\rangle\le C_\omega  \|v\|_{C^\alpha}\|h\|_{\cB_w}\mu_{\bar\phi}(\omega(t-\tau));
\]
\item[(ii)] for all $T>0$, 
\[
\int_T^\infty \|M(t)\|_{\cB\to\cB_w}\, dt\le C_\omega \mu_{\bar\phi}(\tau>T).
\]
\end{itemize}
\end{itemize}

Further, we assume the usual Doeblin-Fortet inequality:
\begin{itemize}
\item[{\bf (H3)}\labeltext{{\bf (H3)}}{H3}] There exist $\sigma_0>1$ and $C_0, C_1>0$ such that for all $h\in\cB$,
for all $n\in\N$
and for all $u>0$,
\[
\|\hat R(u)^n h\|_{\cB_w}\le C_1 e^{-un} \|h\|_{\cB_w},\quad \|\hat R(u)^n h\|_{\cB}
\le C_0 e^{-un} \sigma_0^{-n}\|h\|_{\cB}+C_1\|h\|_{\cB_w}.
\]
\end{itemize}

\subsection{Refined assumptions on $\tau$}

For the purpose of obtaining limit laws for $F$ (and in the end $f_t$) we assume
that
\begin{itemize}
\item[{\bf (H4)}\labeltext{{\bf (H4)}}{H4}]
One of the following holds as $t\to\infty$,
\begin{itemize}
\item[(i)] $\mu_{\bar\phi}(\tau>t)=\ell(t) t^{-\beta}$, for some slowly varying function $\ell$ and $\beta\in(1,2]$.
When $\beta=2$, we assume $\tau\notin L^2(\mu_{\bar\phi})$ and $\ell$ is such that the function $\tilde\ell(t)=\int_1^t\frac{\ell(x)}{x}\, dx$
is unbounded and slowly varying. 

\item[(ii)] $\tau\in L^2(\mu_{\bar\phi})$. We do not assume\footnote{In our example, $\tau$ is only piecewise $C^1(Y)$.} $\tau\in\cB$, but require that  for any $h\in\cB$, $R(\tau h)\in\cB$. 
\end{itemize}
\end{itemize}

\section{Limit laws for $(\tau,F)$ }
\label{sec-lmtF}

In this section we obtain limit theorems for the Birkhoff sum $\tau_n=\sum_{j=0}^{n-1}\tau\circ F^j$.
Under the abstract assumptions formulated in Section~\ref{sec-abstr}, we obtain the asymptotics of the
leading eigenvalue  of $\hat R(u)$ (as in Subsections~\ref{subsec-eign1} and~\ref{subsec-eivclt} below). In turn,
 as clarified in Subsection~\ref{subs-prooflt}, the asymptotics of this (family of) eigenvalues give the asymptotics of  the Laplace transform $\mathbb{E}_{\mu_{\bar\phi}} (e^{-ub_n^{-1}\tau_n})$,
for suitable normalising sequences $b_n$,  proving the claimed limit theorems. 
Our result reads as
follows.

\begin{prop}
\label{prop-limthF}
Assume \ref{H1} and \ref{H2}.
The following hold as $n\to\infty$, w.r.t.\ any probability measure  $\nu$ such that $\nu\ll\mu_{\bar\phi}$.
\begin{itemize}
\item[(i)] Assume \ref{H4}(i).

When $\beta<2$, set $b_n$ such that $\frac{n\ell(b_n)}{b_n^\beta}\to 1$. Then  $\frac{1}{b_n}(\tau_n- \frac{\tau^* \cdot n}{\nu(Y)})\to^d G_\beta$, where $G_\beta$ is a stable law of index $\beta$.

When $\beta=2$, set $b_n$ such that $\frac{n\tilde\ell(b_n)}{b_n^2}\to c>0$.
Then $\frac{1}{b_ n}(\tau_n- \frac{\tau^* \cdot n}{\nu(Y)})\to^d \cN(0,c)$.

\item[(ii)] Assume \ref{H4}(ii) and $\tau-\tau^* \neq h - h \circ F$ for any $h \in \cB$. 
Then there exists $\sigma\ne 0$ such that 
$\frac{1}{\sqrt n}(\tau_n-\frac{\tau^* \cdot n}{\nu(Y)})\to^d \cN(0,\sigma^2)$.
\end{itemize}
\end{prop}

Note here that the slowly varying function $\ell$ from the general theory later on reduces
to $\ell(n) = C^*$, because the tail estimates \eqref{eq:asymp0} have this form.

\begin{rmk}
The above result holds for all piecewise $C^1(Y)$ observables $v$ in the space $\cB_0(Y)$ defined in Appendix~\ref{subsec-coboundary}.
For item (i), we need to assume that $v$ is in the domain of a stable law with index $\beta\in (1,2]$ 
with $v\notin L^2(\mu_{\bar\phi})$ and adjusted $C^* = C^*(v)$, 
while for item (ii) we assume $v\in L^2(\mu_{\bar\phi})$ and a non-coboundary condition precisely formulated 
in Proposition~\ref{prop-limthflow} (ii).
\end{rmk}

The rest of this section is devoted to the proof of the above result.

\subsection{Family of eigenvalues for $\hat R(u)$}
\label{subsec-eign1}

By equation~\eqref{eq-RLapl}, Remark~\ref{rem:r0} and \ref{H2}(ii), $\hat R(u)$ is an analytic family of bounded linear operators on $\cB_w$
for $u>0$ and this family can be continuously extended as $u\to 0$.  As in Remark~\ref{rem:r0}, 
$\lim_{u\to 0}\frac{d}{du}r_0(u), \frac{d^2}{du^2}r_0(u)$ exist
and  we let $\gamma_1 := \frac{d}{du} r_0(0)$, $\gamma_2 := \frac{d^2}{du^2} r_0(0)$.

\begin{lemma}
\label{lemma-cont1}
Assume \ref{H2}. Then 
\[
\|\hat R(u)-\hat R(0)\|_{\cB\to\cB_w}\ll u+\mu_{\bar\phi} (\tau>1/u)=:q(u).
\]
\end{lemma}
\begin{proof}
Using~\eqref{eq-RLapl}, we write
\begin{align*}
\hat R(u)&=(r_0(u)-1)\int_0^\infty  M(t) e^{-ut}\, dt+r_0(0)\int_0^\infty  M(t) (e^{-ut}-1)\, dt+r_0(0)\int_0^\infty  M(t)\, dt\\
&=(r_0(u)-1)\int_0^\infty  M(t) e^{-ut}\, dt+\int_0^\infty  M(t) (e^{-ut}-1)\, dt+\hat R(0).
\end{align*}
Since $|(r_0(u)-1|\le u|\gamma_1|\ll u$, we have that as $u\to 0$,
\begin{align*}
\|\hat R(u)-&\hat R(0)\|_{\cB\to\cB_w}\ll u\int_0^\infty \|M(t)\|_{\cB\to\cB_w}\, dt+\int_0^{1/u}(e^{-ut}-1)\|M(t)\|_{\cB\to\cB_w}\, dt\\
&+\int_{1/u}^\infty \|M(t)\|_{\cB\to\cB_w}\, dt\ll u +
\mu_{\bar\phi} (\tau>1/u)+u\int_0^{1/u}t\|M(t)\|_{\cB\to\cB_w}\, dt.
\end{align*}
We  continue using \ref{H2}(ii). Let $S(t)=\int_{t}^\infty \|M(x)\|_{\cB\to\cB_w}\, dx$ and note that 
\begin{align}
\label{eq-mst}
\|M(t)\|_{\cB\to\cB_w}\ll \int_t^{t+1}\| M(x)\|_{\cB\to \cB_w}\, dx \ll S(t+1)-S(t). 
\end{align}
Rearranging,
\begin{align*}
u&\int_0^{1/u}t\|M(t)\|_{\cB\to\cB_w}\, dt \ll u\int_0^{1/u}t(S(t+1)-S(t))\, dt\\
&\hspace{2cm}=u\Big(\int_1^{1/u+1}(t-1)S(t))\, dt-\int_0^{1/u}tS(t)\, dt\Big)\\
&\hspace{2cm}=u\int_0^1 tS(t)\, dt-u\int_1^{1/u+1}S(t)\, dt-u\int_{1/u}^{1/u+1}t S(t)\, dt\ll u+S(1/u).
\end{align*}
By  \ref{H2}(ii), $S(1/u)\ll \mu_{\bar\phi} (\tau>1/u)$ and the conclusion follows.~\end{proof}

By \ref{H1}(iv), $1$ is an eigenvalue for $\hat R(0)$. By \ref{H3}, there exists a family of eigenvalues $\lambda(u)$  well-defined for $u\in [0,\delta_0)$, 
for some $\delta_0>0$. Also,  for $u\in [0,\delta_0)$,  $\hat R(u)$
has a spectral decomposition  
\begin{align}
\label{eq-spdec}
\hat R(u)=\lambda(u) P(u)+ Q(u),\quad \|  Q(u)^n\|_{\cB}\leq C\sigma_0^n,
\end{align}
 for all $n\in\N$, some fixed $C>0$ and some $\sigma_0<1$.
Here, $P(u)$ is a family of rank one projectors and we  write $P(u)=\mu_{\bar\phi}(u)\otimes v(u)$, normalising such that $\mu_{\bar\phi}(u)(v(u))=1$, where
$\mu_{\bar\phi}(0)$ is the invariant measure and $v(0)=1\in\cB$.  Equivalently, we normalise such that $\langle v(u),1\rangle=1$ and write
\begin{equation}\label{eq:step1}
\lambda(u)=\langle \hat R(u)v(u), 1\rangle.
\end{equation}

The result below gives the continuity of the families $P(u)$ and $v(u)$ (in $\cB_w$).
\begin{lemma}
\label{lemma-contKL} Assume \ref{H1}, \ref{H2} and \ref{H3}. Let $q(u)$ be as defined in Lemma~\ref{lemma-cont1}.
 Then there exist $0 < \delta_1\le \delta_0$ and $C>0$ such that
\begin{align*}
\|P(u)-P(0)\|_{\cB\to\cB_w}\le C q(u)|\log(q(u))|
\end{align*}
for all $0 \leq u < \delta_1$. The same continuity estimate holds for $v(u)$ (viewed as an element of $\cB_w$).
\end{lemma}
\begin{proof}
Given the continuity estimate in Lemma~\ref{lemma-cont1}, we can apply~\cite[Remark 5]{KL99} 
(an improved version of ~\cite[Corollary 1]{KL99} under less general assumptions).
Assumption \ref{H3} here corresponds to~\cite[Hypotheses 2, 3]{KL99} and the estimate in Lemma~\ref{lemma-cont1} corresponds to~\cite[Hypothesis 5]{KL99}.
The extra assumptions of ~\cite[Remark 5]{KL99} are satisfied in the present setup, since by \ref{H3}, $\hat R (u)$ is a family of $\|\cdot\|_{\cB}$ contractions.
Going from $P(u)$ to $v(u)$ is a standard step (see, for instance,~\cite[Proof of Lemma 3.5]{LT16}).~\end{proof}

Since $\lambda(0)=1$ is a simple eigenvalue (by \ref{H1}), Lemma~\ref{lemma-contKL} ensures that $\lambda(u)$ is a family
of simple eigenvalues (see~\cite[Remark 4]{KL99}).  The next result gives precise information on the asymptotic and continuity properties of  $\lambda(u)$, $u\to 0$
(a version of the results  in~\cite{AaronsonDenker01}) in the present 
abstract framework.

\begin{lemma}\label{lemma-asymplambda0}
Assume \ref{H1} and \ref{H2}. Let $q(u)$ be as defined in Lemma~\ref{lemma-cont1}. Set $\Pi(u)=\mu_{\bar\phi} (1-e^{-u\tau})$. Then
as $u\to 0$,
\[
1-\lambda(u) =\Pi(u)\bigg(1+O\big(q(u)|\log(q(u))|\big)\bigg).
\]
\end{lemma}

\begin{proof}
By~\eqref{eq:step1}, $\lambda(u) v(u)=\hat R(u)v(u)$,  $u\in [0,\delta_0)$ with $\lambda(0)=1$ 
and $v(0)=\mu_{\bar\phi}$. Then
\[
\lambda(u)=\langle \hat R(u)v(u),1\rangle=\mu_{\bar\phi}(e^{-u\tau})+\langle (\hat R(u)-\hat R(0))(v(u)-v(0)),1\rangle.
\]
because $\langle v(u),1\rangle=1$. Since \ref{H2} holds (so, \eqref{eq-RLapl} holds),
\begin{align*}
\lambda(u) = \mu_{\bar\phi}(e^{-u\tau})+ V(u) &:= \mu_{\bar\phi}(e^{-u\tau})+
\int_0^\infty (e^{-ut}-1) \langle \omega(t-\tau) [v(u)-v(0)],1\rangle\, dt\\
&\hspace{2cm}+(r_0(u)-1)\int_0^\infty e^{-ut}\langle \omega(t-\tau) [v(u)-v(0)],1\rangle\, dt.
\end{align*}
Thus, $1-\lambda(u)=\Pi(u)-V(u)$. 
By \ref{H2}(i),
\[
\langle \omega(t-\tau)[v(u)-v(0)], 1\rangle\le C \|v(u)-v(0)\|_{\cB_w}\mu_{\bar\phi}(\omega(t-\tau)).
\]
Recalling that $\supp\omega\subset [-1,1]$, for all $u\in [0,\delta_0)$ and some $C_1>0$,
\begin{align*}
\left|\int_0^\infty e^{-ut}\langle \omega(t-\tau) [v(u)-v(0)],1\rangle\, dt \right| & 
\le C \|v(u)-v(0)\|_{\cB_w} \left|\int_0^\infty e^{-ut}\int_Y \omega(t-\tau)\, d\mu_{\bar\phi}\, dt\right|\\
&\hspace{-2.5cm}\le C \|v(u)-v(0)\|_{\cB_w}\left|\int_Y \int_{\tau-1}^{\tau+1} e^{-ut}\omega(t-\tau)\, dt\, d\mu_{\bar\phi}\right|\\
&\hspace{-2.5cm}\le C \|v(u)-v(0)\|_{\cB_w} \int_Ye^{-u\tau}\left|\int_{-1}^{1}\omega(t)e^{ut}\, dt \right|\, d\mu_{\bar\phi}\\
&\hspace{-2.5cm}\le C_1 \|v(u)-v(0)\|_{\cB_w} \int_Ye^{-u\tau}\, d\mu_{\bar\phi}=C_1 \|v(u)-v(0)\|_{\cB_w} (\Pi(u)+1).
\end{align*}

In the previous to last inequality, we have used that $\Big|\int_{-1}^{1}\omega(t)e^{ut}\, dt \Big| \ll 1$ 
for $u \in [0,\delta_0)$. By a similar argument,
\begin{align*}
\Big|\int_0^\infty (e^{-ut}-1)\langle \omega(t-\tau) &[v(u)-v(0)],1\rangle\, dt \Big|\le C_1 \|v(u)-v(0)\|_{\cB_w}\Pi(u).
\end{align*}
The previous two displayed equations  imply that for $u\in [0,\delta_0)$,
\begin{align*}
|V(u)| &\le C_1 \|v(u)-v(0)\|_{\cB_w} |r_0(u)-1|(\Pi(u)+1) +C_1 \|v(u)-v(0)\|_{\cB_w}\Pi(u)\\
&\le C_1\|v(u)-v(0)\|_{\cB_w}\Pi(u).
\end{align*}
By Lemma~\ref{lemma-contKL}, $\|v(u)-v(0)\|_{\cB_w}=O(q(u)|\log(q(u))|)$. 
Hence $V(u)=O(\Pi(u)q(u)|\log(q(u)|)$.~\end{proof}

A first consequence of the above result is
\begin{cor}
\label{cor-st}Assume \ref{H1} and \ref{H2}. The following hold as $u\to 0$.
\begin{itemize}
\item[(i)] If \ref{H4}(i) holds with $\beta<2$,  then $1-\lambda(u)=\tau^* u+u^\beta\ell(1/u)(1+o(1))$.
\item[(ii)] If \ref{H4}(i) holds with $\beta=2$, then $1-\lambda(u)=\tau^* u+u^2 L(1/u)(1+o(1))$, where
$L(t)=\frac{1}{2}\tilde\ell(t)$.
\end{itemize}
\end{cor}
\begin{proof} Under \ref{H4}(i) with $\beta<2$,  $\Pi(u)=\tau^* u+u^\beta\ell(1/u)(1+o(1))$; we refer to, 
for instance,~\cite[Ch. XIII]{Feller66} and~\cite{ADb}. This together with Lemma~\ref{lemma-asymplambda0} (with $q(u)\sim u$)
ensures that $1-\lambda(u)=\tau^* u+u^\beta\ell(1/u)+O(u^2\log(1/u))+O(u^{\beta+1}\ell(1/u))$. Item (i) follows.

Under \ref{H4}(i) with $\beta=2$,  it follows from~\cite[Theorem 3.1 and Corollary 3.7]{ADb} that  $\Pi(u)=\tau^* u+u^2 L(1/u)(1+o(1))$, where $L(t)\sim \ell(t)\log(t)$ with 
$\tilde\ell(t) = \int_1^t \ell(x)/x \, dx$ as in \ref{H4}(i) for $\beta=2$.
The estimate for $\Pi(u)$ together with Lemma~\ref{lemma-asymplambda0} (with $q(u)\sim  u$)
implies that $1-\lambda(u)=\tau^* u+u^2L(1/u)+ O(u\log(1/u)\Pi(u))$. 
 
In the case  of \ref{H4}(i) with $\beta=2$ and $\ell(1/u)\to\infty$ as $u\to 0$, we have $\log(1/u)=o(L(1/u))$. As a consequence, $u\log(1/u)\Pi(u)=O(u^{2}\log(1/u))=o(u^2L(1/u))$ and the conclusion follows.
In the general case (which allows $\ell$ to be asymptotically constant), for fixed small $\delta>0$ and  $\eps>0$, we write
\begin{align*}
u\log(1/u)\Pi(u)&=u\log(1/u)\Big(\Pi(u)-\int_{\tau\le \delta/u} (e^{-u^{2+\eps}\tau^{2+\eps}}-1)\,d\mu_{\bar\phi}\Big)\\
&\hspace{1cm} +u\log(1/u)\int_{\tau\le \delta/u} (e^{-u^{2+\eps}\tau^{2+\eps}}-1)\,d\mu_{\bar\phi}=
u\log(1/u)(I_1(u)+I_2(u)).
\end{align*}
First, 
\begin{align*}
u\log(1/u)|I_2(u)|&\ll u\log(1/u)u^{2+\eps}\int_{\tau\le \delta/u}\tau^{2+\eps}\, d\mu_{\bar\phi}\\
&\le \log(1/u)u^{3+\eps}u^{-2\eps}\delta^{2\eps}\int_Y\tau^{2-\eps}\, d\mu_{\bar\phi}=o(u^2L(1/u)).
\end{align*}
Next, compute that
\begin{align*}
I_1(u)&=\int_{\tau\le \delta/u} (e^{-u\tau}-e^{-u^{2+\eps}\tau^{2+\eps}})\,d\mu_{\bar\phi}+\int_{\tau> \delta/u} (e^{-u\tau}-1)\,d\mu_{\bar\phi}\\
&=\int_{\tau\le \delta/u} (e^{-u\tau}-e^{-u^{2+\eps}\tau^{2+\eps}})\,d\mu_{\bar\phi}+O(\mu_{\bar\phi}(\tau>\delta/u))\\
& =\int_{\tau\le \delta/u} e^{-u\tau}(1-e^{-u^{1+\eps}\tau^{1+\eps}})\,d\mu_{\bar\phi}+O(u^{2}\ell(\delta/u)).
\end{align*}
Let $G(x)=\mu_{\bar\phi}(\tau<x)$ and compute that
\begin{align*}
\int_{\tau\le \delta/u}& e^{-u\tau}(1-e^{-u^{1+\eps}\tau^{1+\eps}})\,d\mu_{\bar\phi}\ll u^{1+\eps}\int_{\tau\le \delta/u}\tau^{1+\eps} e^{-u\tau}\,d\mu_{\bar\phi}\\
&\ll u^{1+\eps}\int_0^\infty x^{1+\eps}e^{-ux}\, dG(x)=-u^{1+\eps}\int_0^\infty x^{1+\eps}e^{-ux}\, d(1-G(x))\\
&=u^{2+\eps}\int_0^\infty x^{1-\eps}(1-G(x))\, e^{-ux}\, dx+u^{1+\eps}\int_0^\infty x^{-\eps}(1-G(x))\, e^{-ux}\, dx\ll u^{1+\eps}.
\end{align*}
Altogether, $I_1(u)\ll u^{1+\eps}$. Hence, $u\log(1/u)|I_1(u)|=o(u^2L(1/u))$. 
This together with the estimate for $u\log(1/u)|I_2(u)$ implies that $u\log(1/u)\Pi(u)=o(u^2L(1/u))$
and item (ii) follows.~\end{proof}

\subsection{Estimates required for the CLT under \ref{H4}(ii)}
\label{subsec-eivclt}

For the CLT case we need the following
\begin{prop}
\label{prop-clt} Assume \ref{H1}, \ref{H2} and \ref{H4}(ii). Suppose that $\tau\neq h\circ F-h$, for any $h\in\cB$.
Then there exists $\sigma\ne 0$ such that $1-\lambda(u)=\tau^* u+\frac{\sigma^2}{2}u^2(1+o(1))$.
\end{prop}

We need to ensure that under the assumptions of Proposition~\ref{prop-clt}, which do not require
that $\tau\in\cB$, there exists the required $\sigma\ne 0$.
The argument goes by and large as~\cite[Proof of Theorem 3.7]{Gouezel04b} (which works with a different
Banach space) with the exception of estimating the second derivative of the eigenvalue $\tilde\lambda$ defined below.
The argument in~\cite[Proof of Theorem 3.7]{Gouezel04b} for estimating such a derivative does not directly apply  to our setup 
due to: i) the two Banach spaces $\cB, \cB_w$ at our disposal are not embedded in $L^p$, $p>1$;  ii) the presence of $\log q(u)$ in Lemma~\ref{lemma-contKL}.
Our estimates  below rely heavily on \ref{H2} and \ref{H4}(ii).

As usual, we can  can reduce the proof to the case of mean zero observables. 
Let $\tilde\tau=\tau-\tau^*$. We recall that under \ref{H4}(ii), $R(\tau h)\in\cB$, for every in $h\in\cB$ and the same holds for $\tilde\tau$.
By \ref{H1v}, $(I-R)$ is invertible on the space of 
functions inside $\cB$ of zero integral.
Thus, as in~\cite[Proof of Theorem 3.7]{Gouezel04b},  we can set
\begin{equation}\label{eq:a}
\aa := (I-R)^{-1}R\tilde\tau \in \cB. 
\end{equation}
Define $\tilde R(u)=R(e^{-u\tilde\tau})$. Clearly, $\tilde R(u)$ has the same continuity properties as $\hat R(u)$. Let $\tilde\lambda(u)$ be the associated 
family of eigenvalues. Recall that $\lambda(u)$ is the family of eigenvalues associated with $\hat R(u)$. Hence,  $\lambda(u)=e^{u\tau^*}\tilde\lambda(u)$. 
As in the previous sections, let $v(u)$ be the family of associated eigenvectors normalised such that $\langle  v(u),1\rangle=1$.
The next three results are technical tools required in the proof of Proposition~\ref{prop-clt}; the third, which has a longer proof, is postponed to 
Subsection~\ref{subsec-tech}.

\begin{lemma}
\label{lemma-exprlderiv} 
Suppose that \ref{H1}--\ref{H3} and \ref{H4}(ii) hold, and recall $\tilde\tau=\tau-\tau^*$.
Then
\[
\frac{d^2}{du^2}\tilde\lambda (u)=-\int_Y\tilde\tau^2\, d\mu_{\bar\phi} - 2\left\langle\frac{d}{du}\tilde R(u)\frac{d}{du}v(u) , 1\right\rangle +T(u),
\]
where $T(u)\to 0$, as $u\to 0$.
\end{lemma}

\begin{proof}Set $\tilde\Pi (u)=\mu_{\bar\phi}(1-e^{-u\tilde\tau})$.
By the calculation used in the proof of Lemma~\ref{lemma-asymplambda0}, $1-\tilde\lambda(u)=\tilde\Pi(u)-\langle (\tilde R(u)-\tilde R(0))(\tilde v(u)-v(0)),1\rangle$.
Differentiating twice,
\begin{align*}
\frac{d^2}{du^2}\tilde \lambda (u)-\frac{d^2}{du^2}\tilde\Pi (u)&= \left\langle \frac{d^2}{du^2}\tilde R(u)(v(u)-v(0),1\right\rangle+2\left\langle \frac{d}{du}\tilde R(u)\frac{d}{du} v(u),1\right\rangle\\
&\quad +\left\langle (\tilde R(u)-\tilde R(0))\frac{d^2}{du^2} v(u),1\right\rangle.
\end{align*}
Under \ref{H2}(i), arguments similar to the ones used in Lemma~\ref{lemma-asymplambda0}
together with Lemma~\ref{lemma-deriv1} (ii) imply that as $u\to 0$,
\begin{align*}
\left|\left\langle (\tilde R(u)-\tilde R(0))\frac{d^2}{du^2} v(u),1\right\rangle\right| &\ll \left\|\frac{d^2}{du^2}v(u)\right\|_{\cB_w} \int_0^\infty (e^{-ut}-1) \mu_{\bar\phi}(\omega(t-\tilde\tau))\, dt\\
&\hspace{-2cm}+|r_0(u)-1|\left\|\frac{d^2}{du^2}v(u)\right\|_{\cB_w} \int_0^\infty e^{-ut} \mu_{\bar\phi}(\omega(t-\tilde\tau))\, dt\ll u\log(1/u)=o(1).
\end{align*}
Lemmas~\ref{lemma-deriv1} (ii) and~\ref{lemma-contKL} together with $\tilde\tau^2\in L^1(\mu_{\bar\phi})$ imply that $| \langle \frac{d^2}{du^2}\tilde R(u)(v(u)-v(0),1\rangle|\ll u\log(1/u)=o(1)$.
Finally, $\frac{d^2}{du^2}\tilde\Pi (u)=-\int_Y\tilde\tau^2\, d\mu_{\bar\phi}(1+o(1))$.
The conclusion follows by putting the above estimates together.~\end{proof}

\begin{lemma}
\label{lemma-mainclt} 
Assume the setup of Proposition~\ref{prop-clt}.  Then 
\[\left\langle \frac{d}{du}\tilde R(u)\frac{d}{du}v(u),1\right\rangle=-\int\tilde\tau \aa\, d\mu_{\bar\phi}+D(u),
\] 
where $D(u)=o(1)$ as $u\to 0$, and $\aa$ is defined as in \eqref{eq:a}.
\end{lemma}

\begin{pfof}{Proposition~\ref{prop-clt}} By Lemmas~\ref{lemma-exprlderiv} and~\ref{lemma-mainclt},
\begin{align}
\label{eq-tildel}
\frac{d^2}{du^2}\tilde\lambda(u)=-\int_Y\tilde\tau^2\, d\mu_{\bar\phi}-2\int\tilde\tau \aa\, d\mu_{\bar\phi}+\tilde T(u),
\end{align}
where $\tilde T(u)\to 0$, as $u\to 0$.  Hence, 
$-\frac{d^2}{du^2}\tilde\lambda(u)|_{u=0}=\int_Y\tilde\tau^2\, d\mu_{\bar\phi}+2\int\tilde\tau \aa\, d\mu_{\bar\phi}$
and we can set $\sigma^2=\int_Y\tilde\tau^2\, d\mu_{\bar\phi}+2\int\tilde\tau \aa\, d\mu_{\bar\phi}$.
From here on the proof goes word for word as ~\cite[Proof of Theorem 3.7]{Gouezel04b}, which shows that given the previous formula for $\sigma$,
the only possibility for $\sigma=0$ is when  $\tilde\tau=h-h\circ F$, for some density $h\in\cB$. This is ruled out by assumption.

To conclude recall that $\lambda(u)=e^{u\tau^*}\tilde\lambda(u)$. By~\eqref{eq-tildel}, 
$1-\tilde\lambda(u)=\frac{\sigma^2}{2}u^2(1+o(1))$, as required.~\end{pfof}

\subsubsection{Proof of Lemma~\ref{lemma-mainclt}}
\label{subsec-tech}

\begin{pfof}{Lemma~\ref{lemma-mainclt}}
Although $\frac{d}{du}\tilde R(u)$ is bounded in $\cB_w$,  a priori $\frac{d}{du}v(u)$ is not; this a consequence of  Lemma~\ref{lemma-deriv1} (i).
However, we argue that due to \ref{H4}(ii), $\frac{d}{du}v(u)$ is well-defined at $0$ and as such, we get the claimed formula for
$\langle \frac{d}{du}\tilde R(u)\frac{d}{du}v(u),1\rangle$.
Let $W(u):=\frac{d}{du}\tilde R(u)$. Since $\tilde R(u) = R(e^{-u \tilde\tau})$,
 we have that for any $h\in\cB$,
\begin{equation}\label{eq:W0}
W(0)h:=\left.\frac{d}{du}\tilde R(u)\right|_{u=0} h=-R(\tilde\tau h).
\end{equation}
By  \ref{H4}(ii),  $W(0)h\in\cB$. Recall that $P(u)$ is the eigenprojection for $\hat R(u)$, so for $\tilde R(u)$ as well. For $\delta$
small enough (independent of $u$), we can write
\[
 P(u)=\int_{|\xi-1|=\delta}(\xi I-\tilde R(u))^{-1}\, d\xi.
\]
Recall that $v(u)=\frac{P(u)1}{\langle P(u)1, 1\rangle}=\frac{P(u)1}{m_{\bar\phi}(u)1}$; in particular, $\mu_{\bar\phi}(u)1=\langle P(u)1, 1\rangle$.
Compute 
\begin{align}
\label{eq-derv}
\frac{d}{du}v(u)=\frac{\frac{d}{du}P(u)1}{\mu_{\bar\phi}(u)1}-\frac{P(u)1}{(\mu_{\bar\phi}(u)1)^2}\frac{d}{du}\mu_{\bar\phi}(u)1.
\end{align}
Also, compute that by \eqref{eq:W0}
\begin{align}
\label{eq-derp}
\nonumber\frac{d}{du} P(u)&=\int_{|\xi-1|=\delta}(\xi I-\tilde R(u))^{-1}W(u)(\xi I-\tilde R(u))^{-1}\, d\xi\\
\nonumber&=\int_{|\xi-1|=\delta}(\xi I-R)^{-1}W(0)(\xi I-R)^{-1}\, d\xi +Z_0(u)\\
&=-\int_{|\xi-1|=\delta}(\xi I-R)^{-1}R\tilde\tau(\xi I-R)^{-1}\, d\xi +Z_0(u),
\end{align}
where the first term is well defined in $\cB$ by \ref{H4}(ii) and
\begin{align*}
\|Z_0(u)\|_{\cB_w}&\ll \left\|\int_{|\xi-1|=\delta} (\xi I-R)^{-1} \Big(W(u)-W(0)\Big)(\xi I-R)^{-1}\, d\xi \right\|_{\cB_w}\\
&\qquad+ \left\|\int_{|\xi-1|=\delta}\Big((\xi I-R(u))^{-1}-(\xi I-R)^{-1}\Big)W(u)(\xi I-R)^{-1}\, d\xi \right\|_{\cB_w}\\
&\qquad+ \left\|\int_{|\xi-1|=\delta}(\xi I-R)^{-1}W(u)\Big((\xi I-R(u))^{-1}-(\xi I-R)^{-1}\Big)\, d\xi \right\|_{\cB_w}.
\end{align*}

By Lemma~\ref{lemma-deriv1} (ii), $\frac{d^2}{du^2}\tilde R(u)$ is
bounded in $\|\cdot\|_{\cB_w}$; so, $\|W(u)-W(0)\|_{\cB_w}\ll u$. Hence, each one of three terms of the previous displayed equation contains an $O(u)$ factor.
By the argument recalled in the proof of Lemma~\ref{lemma-contKL}, for any $\xi$ such that $|\xi-1|=\delta$, the integrand of each one of these three terms is $o(1)$.
Altogether, $\|Z_0(u)\|_{\cB_w}=o(1)$. 
Also, by~\eqref{eq-derp} and Lemma~\ref{lemma-A1} below,
\begin{align}
\label{eq-derp2}
\frac{d}{du} P(u)1 = -\aa -\xi_0 \int_Y \aa \, d\mu_{\bar\phi}+Z_0(u),
\end{align}
for some $\xi_0 \in \C$ and $\aa = (I-R)^{-1} R \tilde \tau$ as in \eqref{eq:a}.
Plugging~\eqref{eq-derp2} into~\eqref{eq-derv},

\begin{align*}
\frac{d}{du}v(u)&=\frac{1}{\mu_{\bar\phi}(0)1}\left(-\aa -\xi_0 \int_Y \aa \, d\mu_{\bar\phi}\right)
-\frac{P(0)1}{(\mu_{\bar\phi}(0)1)^2}\left\langle-\aa -\xi_0 \int_Y \aa \, d\mu_{\bar\phi}, 1\right\rangle +Z_1(u)+Z_0(u)\\
&=-\aa +\tilde\xi\int_Y \aa \, d\mu_{\bar\phi} +Z_1(u)+Z_0(u),
\end{align*}
where $\tilde\xi$ is linear combination of $\xi_0$ and
\begin{align*}
\|Z_1(u)\|_{\cB_w} \ll &\, \frac{|\mu_{\bar\phi}(u)1-\mu_{\bar\phi}(0)1|}{m_{\bar\phi}(u)1} +\frac{\| P(u)1- P(0)1\|_{\cB_w}}{(m_{\bar\phi}(u)1)^2}+\left|\frac{1}{\mu_{\bar\phi}(u)1}-\frac{1}{\mu_{\bar\phi}(0)1}\right|\\
&+\left|\frac{1}{(\mu_{\bar\phi}(u)1)^2}-\frac{1}{(\mu_{\bar\phi}(0)1)^2}\right|.
\end{align*}
By Lemma~\ref{lemma-contKL}, $\| P(u)- P(0)\|_{\cB_w}\ll u\log(1/u)$ and as a consequence, 
$|\mu_{\bar\phi}(u)1-\mu_{\bar\phi}(0)1|\ll u\log(1/u)$. Thus, $\|Z_1(u)\|_{\cB_w}=o(1)$. 
Recalling that $\|Z_0(u)\|_{\cB_w}=o(1)$, we have
\[
\frac{d}{du}v(u)=-\aa +\tilde\xi\int_Y \aa \, d\mu_{\bar\phi} +Z(u),
\]
where  $\|Z(u)\|_{\cB_w}=o(1)$. Using this expression for $\frac{d}{du}v(u)$, we obtain

\begin{align*}
&\left\langle \frac{d}{du}\tilde R(u)\frac{d}{du}v(u),1\right\rangle =-\left\langle R\tilde\tau \frac{d}{du}v(u),1\right\rangle+\left\langle (W(u)-W(0))\, \frac{d}{du}v(u),1\right\rangle  \\
&\qquad=-\int_Y R\tilde\tau \aa \, d\mu_{\bar\phi} +\tilde\xi\int \tilde\tau \, d\mu_{\bar\phi}\int_Y \aa \, d\mu_{\bar\phi}+\left\langle R Z(u),1\rangle+\langle (W(u)-W(0))\, \frac{d}{du}v(u),1\right\rangle\\
&\qquad=-\int_Y \tilde\tau \aa \, d\mu_{\bar\phi} +\langle R Z(u),1\rangle+\left\langle (W(u)-W(0))\, \frac{d}{du}v(u),1\right\rangle,
\end{align*}
where we have used $\int \tilde\tau \, d\mu_{\bar\phi}=0$. Finally, using \ref{H2}, we compute that
\[ 
|\langle R Z(u),1\rangle|\ll \|Z(u)\|_{\cB_w}\int_0^\infty \int_Y \omega(t-\tilde\tau)\, d\mu_{\bar\phi}\, dt\ll \|Z(u)\|_{\cB_w}=o(1).
\]
Similarly, $|\langle (W(u)-W(0)\, w(u),1\rangle |=o(1)$, since we already know that $\|W(u)-W(0)\|_{\cB_w}=o(1)$. Thus,
\begin{align*}
\left\langle \frac{d}{du}\tilde R(u)\frac{d}{du}v(u),1\right\rangle 
=-\int_Y \tilde\tau \aa \, d\mu_{\bar\phi} +o(1),
\end{align*}
ending the proof.~\end{pfof}

\begin{lemma}
\label{lemma-A1} Assume the setup and notation of Lemma~\ref{lemma-mainclt}. Then there exists $\xi_0 \in \C$ with such that
\[
\int_{|\xi-1|=\delta} (\xi I-R)^{-1}R\tilde\tau\,(\xi I-R)^{-1} 1\, d\xi=\aa +\xi_0 \int_Y \aa \, d\mu_{\bar\phi}.
\]
\end{lemma}

\begin{proof}
We note that  for each $\xi$ such that $\xi-1|=\delta$, $R\tilde\tau\,(\xi I-R)^{-1} 1$ is well defined in $\cB$ by \ref{H4}(ii)
and that $\int_{|\xi-1|=\delta}R\tilde\tau\,(\xi I-R)^{-1} 1\, d\xi=0$. Hence, $\int_{|\xi-1|=\delta}(I-R)^{-1}R\tilde\tau\,(\xi I-R)^{-1} 1\, d\xi\in \cB$.

By the first resolvent identity and the Mean Value Theorem, 
\begin{align*}
 & \hspace{-3mm} \int_{|\xi-1|=\delta} (\xi I-R)^{-1}R\tilde\tau\,(\xi I-R)^{-1} 1\, d\xi \\
&=\int_{|\xi-1|=\delta}(I-R)^{-1}R\tilde\tau\,(\xi I-R)^{-1} 1\, d\xi
+ \int_{|\xi-1|=\delta}\Big((\xi I-R)^{-1}-(I-R)^{-1}\Big)R\tilde\tau\,(\xi I-R)^{-1} 1\, d\xi\\
&= \aa P(0)1 +  \int_{|\xi-1|=\delta} (1-\xi) (\xi I-R)^{-1}\aa\,(\xi I-R)^{-1} 1\, d\xi \\
&= \aa +\xi_1 \int_{|\xi-1|=\delta}(\xi I-R)^{-1}\aa\,(\xi I-R)^{-1} 1\, d\xi,
\end{align*}
where $\xi_1$  is a complex constant with $0<|\xi_1|\le\delta$.

Write $(\xi I-R)^{-1}-I=-(\xi I-R)^{-1}( (\xi-1) I -R)$. The Mean Value Theorem gives
\begin{align*}
\int_{|\xi-1|=\delta}&(\xi I-R)^{-1}\aa\,(\xi I-R)^{-1} 1\, d\xi=
\int_{|\xi-1|=\delta}(\xi I-R)^{-1}a\,\Big(I+(\xi I-R)^{-1}-I\Big) 1\, d\xi\\
&= \int_{|\xi-1|=\delta}(\xi I-R)^{-1}\aa\, d\xi  -
\int_{|\xi-1|=\delta}(\xi I-R)^{-1}\aa\,(\xi I-R)^{-1}( (\xi-1) I - R)1\, d\xi\\
& =P(0) \aa
-\int_{|\xi-1|=\delta}(\xi I-R)^{-1}\aa\, (\xi I-R)^{-1} (\xi-2) \, d\xi\\
&=\int_Y \aa \, d\mu_{\bar\phi}-(\xi_2-2)\int_{|\xi-1|=\delta}(\xi I-R)^{-1}\aa\,(\xi I-R)^{-1} 1\, d\xi,
\end{align*}
for some $\xi_2 \in \C$ with $|\xi_2-1| \leq \delta$.
Thus, 
\[
(\xi_2-1)\int_{|\xi-1|=\delta}(\xi I-R)^{-1}\aa\,(\xi I-R)^{-1} 1\, d\xi =\int_Y \aa \, d\mu_{\bar\phi}.
\]
Since $\int_Y \aa \, d\mu_{\bar\phi} > 0$, we have $\xi_2 \neq 1$. The conclusion follows with $\xi_0 =\xi_1(\xi_2-1)^{-1}$.~\end{proof}

\subsection{Proof of Proposition~\ref{prop-limthF}}
\label{subs-prooflt}

We start with part (i), i.e., the stable and non standard Gaussian laws.
Let $\tau_n$ and $b_n$ be as in Proposition~\ref{prop-limthF}.
Using~\eqref{eq-spdec} and Corollary~\ref{cor-st} (based on~\cite{ADb})
we obtain that the Laplace transform $\mathbb{E}_{\mu_{\bar\phi}} (e^{-ub_n^{-1}\tau_n})$
converges (as $n\to \infty$) to the Laplace transform of either the stable law or of the $\cN(0,1)$ law.
The conclusion w.r.t.\, $\mu_{\bar\phi}$ follows from the theory of Laplace transforms (as in~\cite[Ch. XIII]{Feller66}).
The conclusion w.r.t.\, $\nu$ follows from~\cite[Theorem 4]{Eagleson}.

For part (ii), i.e., the CLT,
let $\tau_n$ and $b_n$ be as in  Proposition~\ref{prop-limthF} under \ref{H4}(ii).
Proposition~\ref{prop-clt} shows that the Laplace transform $\mathbb{E}_{\mu_{\bar\phi}} (e^{-ub_n^{-1}\tau_n})$
converges (as $n\to \infty$) to the Laplace transform of  $\cN(0,\sigma^2)$.
The conclusion w.r.t.\, $\mu_{\bar\phi}$ follows from the theory of Laplace transforms (as in~\cite[Ch. XIII]{Feller66}).
The conclusion w.r.t.\, $\nu$ follows from~\cite[Theorem 4]{Eagleson}.

\section{Limit laws for the flow $f_t$}
\label{sec-lmtflow}

The result below generalises Proposition~\ref{prop-limthF} to the flow $f_t$. Given an $\alpha$-H\"older
observable $\w:\cM\to\R$, 
define $\bar\w:Y\to\R$ as in \eqref{eq:ind pot} and let $\w_T=\int_0^T \w \circ f_t\, dt$.

\begin{prop}
\label{prop-limthflow}
Assume \ref{H1}. Let $\w:\cM\to\R$ and let $\w^*=\int_Y \bar \w\,d \mu_{\bar\phi}$.
Suppose  that the twisted operator $\hat R_\w(u)v=R(e^{-u\bar \w}v)$ satisfies \ref{H2} 
(with $\tau$ replaced by $\bar\w$). 
Then the following hold as $T\to\infty$, w.r.t.\, $\mu_{\bar\phi}$ 
(or any probability measure $\nu\ll\mu_{\bar\phi}$).
\begin{itemize}
\item[(i)] Assume \ref{H4}(i) and $\mu_{\bar\phi}(\bar \w>T)\sim \mu_{\bar\phi}(\tau>T)$.

When $\beta<2$, set $b(T)$ such that $\frac{T\ell(b(T))}{b(T)^\beta}\to 1$.
Then  $\frac{1}{b(T)}(\w_T-\w^*\cdot T)\to^d G_\beta$,  where $G_\beta$ is a stable law of index $\beta$.

When $\beta=2$, set  $b(T)$ such that $\frac{T\tilde\ell(b(T))}{b(T)^2}\to c>0$.
 then $\frac{1}{b(T)}(\w_T-\w^*\cdot T)\to^d \cN(0,c)$.

\item[(ii)] Suppose that \ref{H4}(ii) holds. 
If $\bar \w\notin \cB$, we assume that $R\bar\w\in\cB$. We further assume that $\bar \w\neq h-h\circ F$ with $h\in\cB$. 
Then there exists $\sigma>0$ such that 
$\frac{1}{\sqrt T}(\w_T- \w^* \cdot T)\to^d \cN(0,\sigma^2)$. 
\end{itemize}
\end{prop}

\begin{rmk} The assumption $\mu_{\bar\phi}(\bar \w>T)\sim \mu_{\bar\phi}(\tau>T)$ is satisfied for all 
$\w$ bounded above with $\int_0^{h(x)} \w \circ f_t(p) \, dt = c > 0$ (where $p \in \Sigma$ is the fixed point
of the Poincar\'e section) and
such that $\int_0^{h(x)} \w \circ f_t(q) \, dt$ is $C^\alpha$ in $q$ near $p$.
The proof of this is similar to \cite[Proof of Theorem 1.3]{Gouezel04b} where the source of non-hyperbolicity 
was a neutral fixed point $x_0$ instead of a neutral periodic orbit $\Gamma = \{ p \} \times \T^1$.
Therefore, we just need to replace $g(x_0)$ in \cite[Proof of Theorem 1.3]{Gouezel04b}  by 
$\int_0^{h(p)} \w \circ f_t(p) \, dt$

In Theorem~\ref{thm-concl} below, we consider potentials of the form
$\bar\w=C'-C\tau^\kappa(1+o(1))$, $\kappa\in (0,\beta)$ and obtain specific form of  (i) and/or (ii) for different values of $\kappa$.
\end{rmk}

\begin{proof}
We use that $f_t:\cM\to \cM$ can be represented as a suspension flow $F_t:Y^\tau\to Y^\tau$.
Under the present assumptions, $\bar\w$  satisfies 
all the assumptions of Proposition~\ref{prop-limthF} (with $\tau$ replaced by $\bar\w$). 

Let $\bar \w_n=\sum_{j=0}^{n-1}\bar\w\circ F^j$ and recall $\tau_n=\sum_{j=0}^{n-1}\tau\circ F^j$.
Proposition~\ref{prop-limthF} (i) applies to $\bar \w_n$; the argument goes word for word as 
in the proof of Proposition~\ref{prop-limthF} (i) for $\tau_n$.
Under the present assumptions on $\bar \w$, item (ii) of Proposition~\ref{prop-limthF} applies to $\bar \w_n$ with
the argument used in the proof of Proposition~\ref{prop-clt} applying word for word with $\bar \w$ instead of $\tau$.
 
Item (i) follows from this together with Lemma~\ref{lemma-lift} below (correspondence between stable laws/non standard Gaussian for the base map $F$ and suspension flow.
Item (ii) follows in the same way using~\cite[Theorem 1.3]{MTor} instead of Lemma~\ref{lemma-lift}.~\end{proof}

The next result is a
version of ~\cite[Theorem 7]{Sar06} (generalising \cite[Theorem 1.3]{MTor}) for suspension flows
which holds in a very general setup; in particular, it is totally independent of method used to prove limit theorems
for the base map.

\begin{lemma}
\label{lemma-lift} Assume $\int_Y \tau d\mu_{\bar\phi}<\infty$ and let  $\w\in L^q(\mu_\phi)$, for $q>1$.
Suppose that there exists a sequence $b_n=n^{-\rho}\ell(n)$ for $\rho\in (1,2]$
and $\ell$ a slowly varying function such that $b_n^{-1}\left(\tau_n-n\int_Y \tau d\mu_{\bar\phi}\right)$
is tight on $(Y, \mu_{\bar\phi})$. Then the following are equivalent:
\begin{itemize}
\item[(a)]
 $b_n^{-1}\bar \w_n$ converges in distribution on  $(Y,\mu_{\bar\phi})$.

\item[(b)] $b(T)^{-1} \w_T$ converges in distribution on  $(Y^\tau, \mu_\phi)$,
where $b(T)=T^{-\rho}\ell(T)$.
\end{itemize}
\end{lemma}

\begin{proof}The fact the (b) implies (a) is obvious. 
The implication from (a) to (b) 
is contained in the proofs of~\cite[Theorem 1.3]{MTor} for the standard CLT case
and in~\cite[Theorem 7]{Sar06} for the stable and nonstandard CLT. We sketch the argument  for completeness.

Following~\cite{MTor}, let $n[x,T]$ denote the largest integer such that
$\tau_n(x)\leq T$; that is,
\[
\tau_{n[x,T]}(x)\leq T\leq \tau_{n[x,T]+1}(x).
\]
By the ergodic theorem, $\lim_{T\to\infty}\frac{1}{T}n[x,T]=\int_Y \tau\,d\mu_\phi =: \bar\tau$, a.e.
Hence, $n[x,T]=[\bar\tau T](1+o(1))$, a.e. as $T\to\infty$. 

Recall that $g:Y^\tau\to\R$ and 
define $\hat g_T(y,u)=g_T(y)$.
Recall $\bar g(y)=\int_0^{\tau(y)} g(y,u)\, du$.
By assumption, $b_n^{-1}\bar g_n$ converges in distribution on  $(Y,\mu_\phi)$.
Hence, (as in~\cite[Lemma 3.1]{MTor}),
$b_n^{-1} \hat g_n$ converges in distribution on  $(Y^\tau, \mu_\phi^\tau)$; this is a consequence
of~\cite[Theorem 4]{Eagleson}.

In what follows we adopt the convention $b_T=b(T)$. Write
\begin{align*}
\frac{g_T}{b_T}&=\frac{b_{[\bar\tau T]}}{b_T}\Big(\frac{\bar g_{[\bar\tau T]}}{b_{[\bar\tau T]}}+
\frac{1}{b_{[\bar\tau T]}}\Big(\bar g_{n[x,T]}-\bar v_{[\bar\tau T]}\Big)\Big)
+\Big(\frac{\hat g_T}{b_T}-\frac{g_{n[x,T]}}{b_{n[x,T]}}\Big)\\
&+\Big(\frac{\hat g_{n[x,T]}}{b_{n[x,T]}}-\tau_{n[x,T]}\circ F^{\tau_{n[x,T]}}\Big).
\end{align*}

Since $\ell$ is slowly varying, $\frac{b_{[\bar\tau T]}}{b_T}\to \bar\tau^\rho$.
Since $g\in L^1(\mu_\phi^\tau)$, ~\cite[Step 2
of the the proof of Theorem 7]{Sar06} (a generalization of~\cite[Lemma 3.4]{MTor})
applies
and thus, $\frac{1}{b_{[\bar\tau T]}}\Big(\bar g_{n[x,T]}-\bar g_{[\bar\tau T]}\Big)\to_d 0$
on   $(Y, \mu_\phi)$. 

Next, since $b_n^{-1}(\tau_n-n\mu_\phi(Y))$
is tight on $(Y, \mu_\phi)$, we have that $\tau\in L^{1/\rho}(\mu_\phi)$.
By assumption, $g\in L^q(\mu_\phi^\tau)$, for $q>1$.
Hence, the assumptions of~\cite[Lemma 2.1]{MTor}
are satisfied with $a=p=1/\rho$ and any $b:=q>1$ (with $a, b, p$ as there).
By ~\cite[Lemma 2.1 (b)]{MTor}, $\frac{g_T}{b_T}-\frac{\hat g_{n[x,T]}}{b_{n[x,T]}}\to_d 0$
on $(Y^\tau, \mu_\phi^\tau)$.

To conclude, note that $\hat g_{n[x,T]}(y,u)=g_{n[x,T]}(y)$. By  ~\cite[Step 3
of the the proof of Theorem 7]{Sar06},
$\frac{\hat g_{n[x,T]}}{b_{n[x,T]}}-\tau_{n[x,T]}\circ F^{\tau_{n[x,T]}}\to_d 0$
on   $(Y, \mu_\phi)$, as required.~\end{proof}

\section{Asymptotics of $\cP(\phi+s\psi)$ for the flow $f_t$ }
\label{sec-pressure}

In this section and the next we shall assume that $F:Y\to Y$ is Markov, 
which allows us to express our results in terms of pressure.
We will also assume that $P(\phi) = 0$ for all the potential functions $\phi$ involved.
\cite[Theorem 8]{Sar06} gives a link between the shape of the pressure of a given discrete time finite measure
dynamical system and an induced version (this was extended in \cite{BruTerTod17} 
to some infinite measure settings).
Here we give a version of this result in the abstract setup of 
Section~\ref{sec-abstr} along with suitable assumptions on a second potential $\psi$.  
Note that our assumptions are not directly comparable with those in \cite[Theorem 8]{Sar06}.

Throughout this section, we let $\phi:\cM\to\R$ and $\mu_\phi$ be as Section~\ref{sec-abstr}. 
In particular, we recall that $\tau^*= \int_Y \tau \, d\mu_{\bar\phi} < \infty$
and assume that the conformal measure $m_{\bar\phi}$ satisfies \ref{H1}.
This ensures that $\mu_{\bar\phi}$ is an equilibrium measure for $(F,\bar\phi)$ and $\mu_{\phi}$ is an equilibrium measure for $(f_t,\phi)$. 
We consider the potential $\psi :\cM\to  \R$ and its induced version $ \bar\psi:Y\to \R$ 
defined in~\eqref{eq:ind pot} and require:

\begin{itemize}
\item[{\bf (H5)}\labeltext{{\bf (H5)}}{H5}]
There exists $C, C'>0$ such that for $y\in Y$, $\bar\psi(y)=C'-\psi_0(y)$, where 
$\psi_0(y) = C \tau^\kappa(y) (1+o(1))$ for some $\kappa>0$. For $\kappa>1$ we further require that
$\int_Y\tau^\kappa\, d\mu_{\bar\phi}<\infty$.
\end{itemize}

Given $\psi:\cM\to\R$ and its induced version $\bar\psi$ on $Y$,
we define the `doubly perturbed' operator
$$
\hat R(u,s)v := R(e^{-u\tau}e^{s\bar\psi}v).
$$
Under \ref{H5}, we require the following extended version of \ref{H3}.
\begin{itemize}
\item[{\bf (H6)}\labeltext{{\bf (H6)}}{H6}]
There exist $\sigma_1>1$, constants $C_0, C_1, \delta>0$ 
such that for all $h\in\cB$,
for all $n\in\N$ and $u\ge 0$, $s\in [0,\delta)$,
\[
\|\hat R(u,s)^n h\|_{\cB_w}\le C_1\|h\|_{\cB_w},\quad \|\hat R(u,s)^n h\|_{\cB}
\le C_0\sigma_1^{-n}\|h\|_{\cB}+C_1\|h\|_{\cB_w}.
\]
\end{itemize}

To ensure that we can understand the pressure $\cP(\overline{\phi+s\psi})$ in terms of the eigenvalues of $\hat R(u,s)$, under \ref{H1} 
(in particular, for $\alpha$ as in \ref{H1}(i)) and \ref{H5} we require
that

\begin{itemize}
\item[{\bf (H7)}\labeltext{{\bf (H7)}}{H7}]
There exist $\gamma\in (0,1]$ (with $\gamma<1$ when $\kappa>1$)
and $C_2, C_3, \delta>0$ such that for all $u\ge 0$, $s\in (0, \delta)$, $h \in \cB$ and every element $a$ 
of the (Markov) partition $\cA$ we have:
\begin{equation}
\label{eq-contpsi0}
\|(e^{-u\tau+s\bar\psi}-1)1_a h\|_{\cB_w} \le \left(
C_2 s^\gamma\sup_a\psi_0^\gamma + C_3 u^\gamma \sup_a \tau^\gamma\right) \| h \|_{\cB_w}.
\end{equation}
\end{itemize}
In~\eqref{eq-contpsi0}, multiplication by $1_a$ means that we restrict $\tau,\bar\psi$ to $a$. The main result of this section reads as follows.

\begin{thm} 
\label{thm-relpres}
Assume 
\ref{H2} and suppose that $\psi:\cM\to\R$ satisfies \ref{H5} with  $C' > C \int \tau^\kappa \, d\mu_{\bar\phi}$.
Assume \ref{H6} and \ref{H7}. Then
$$
\cP(\phi+s\psi)=  \frac{1}{\tau^*}\cP(\overline{\phi+s\psi})(1+o(1))\text{ as } s\to 0^+.
$$
\label{thm:Sar thm8}
\end{thm}

\subsection{Technical tools, family of eigenvalues of $\hat R(u,s)$}
\label{subsec-techt}

Note  that $\hat R(u,0) v=\hat R(u) v$ and recall from Subsection~\ref{subsec-eign1} that under \ref{H1}--\ref{H3}, 
the family of eigenvalues $\lambda(u)$ is well-defined on $[0,\delta_0)$
with $\lambda(0)=1$. Recall that \ref{H6} and \ref{H7} hold for some $\delta>0$. Throughout the rest of this work we let $\delta_1=\min\{\delta, \delta_0\}$.
\begin{lemma}
\label{lemma-contdouble}
Assume \ref{H1}, \ref{H2}, \ref{H5} and \ref{H7}. 
Then for all $s\in (0,\delta_1)$, there exists $c>0$ such that
\[
\|\hat R(u,s)-\hat R(u,0)\|_{\cB\to \cB_w}\leq c s^\gamma.
\]
\end{lemma}
\begin{proof} Using~\eqref{eq-RLapl}, write
\begin{align}
\label{eq-rus}
\hat R(u,s)-\hat R(u,0)= r_0(u)\int_0^\infty R(\omega(t-\tau) (e^{s\bar\psi}-1)) e^{-ut}\, dt.
\end{align}

By \ref{H7}, $(e^{s\bar\psi}-1)1_a\in\cB$ for any element $a$ in the (Markov) partition $\cA$.
Without loss of generality we assume that $\supp \omega(t-\tau)$ is a subset of a 
finite union $\cup_{a \in A_t} a$ of elements in $\cA$. Thus, $\tau$ and $t$ are of the same order of magnitude
for $a \in A_t$ and $\|(e^{s\bar\psi}-1) 1_{\supp\, \omega(t-\tau)}\|_{\cB_w} \ll 
\max_{a \in A_t}\|(e^{s\bar\psi}-1) 1_a\|_{\cB_w}$. 

Note that 
\begin{align*}
\| R(\omega(t-\tau) (e^{s\bar\psi}-1))v\|_{\cB_w}\le C\| R(\omega(t-\tau))\|_{\cB\to\cB_w} 
\max_{a \in A_t}\|(e^{s\bar\psi}-1) 1_a\|_{\cB_w}.
\end{align*}
By \ref{H5} and \ref{H7} (with $u=0$),
\begin{align*}
\|(e^{s\bar\psi}-1) 1_{\supp\, \omega(t-\tau)}\|_{\cB_w} \ll \|(e^{s\bar\psi}-1) 1_a\|_{\cB_w} 
\ll s^\gamma \sup_{a \in A_t}\psi_0^\gamma \ll s^\gamma t^{\kappa\gamma}.
\end{align*}
Recall that $\gamma<1$ if $\kappa>1$. Putting the above together and recalling $R(\omega(t-\tau))=M(t)$, 
we obtain that for any $\eps\in (0,1-\gamma)$,
\begin{align*}
\|\hat R(u,s)-\hat R(u, 0)\|_{\cB\to\cB_w}
\ll s^\gamma\int_0^\infty t^{\kappa\gamma}  \| M(t)\|_{\cB\to\cB_w}\, dt
\ll s^\gamma\int_0^\infty t^{\kappa -\eps}  \| M(t)\|_{\cB\to\cB_w}\, dt.
\end{align*}

Proceeding as in the proof of Lemma~\ref{lemma-cont1}, in particular using~\eqref{eq-mst},
we have $\int_0^\infty t^{\kappa-\eps} \| M(t)\|_{\cB\to\cB_w}\, dt\ll \int_0^\infty t^{\kappa-\eps} (S(t+1)-S(t))\, dt$
for $S(t)=\int_t^\infty \| M(x)\|_{\cB\to\cB_w}\, dx$.
This gives
\begin{align*}
\int_0^\infty t^{\kappa-\eps} \| M(t)\|_{\cB\to\cB_w}\, dt
&\ll \int_0^\infty t^{\kappa-\eps} S(t)\, dt - \int_0^\infty t^{\kappa-\eps} S(t+1)\, dt \\
&= \int_0^\infty  t^{\kappa-\eps} S(t)\, dt - \int_1^\infty  t^{\kappa-\eps} (1-\frac1t)^{\kappa-\eps} S(t)\, dt \\
&\ll \int_0^1 t^{\kappa-\eps} S(t)\, dt + (\kappa-\eps)\int_1^\infty t^{\kappa-1-\eps} S(t)\, dt
\ll \int_1^\infty t^{\kappa-1-\eps} S(t)\, dt.
\end{align*}
By assumption, $\tau\in L^{\kappa}(\mu_{\bar\phi})$, so 
$t^\kappa\mu_{\bar\phi}(\tau>t)\le \int_{\tau\ge t}\tau^\kappa \, d\mu_{\bar\phi} 
\le  \int_Y\tau^\kappa\, d\mu_{\bar\phi}<\infty$.
By \ref{H2}(ii), $S(t)\ll \mu_{\bar\phi}(\tau>t)$.
Thus,
\begin{align}
\label{eq-tmt}
\int_0^\infty t^{\kappa-\eps} \| M(t)\|_{\cB\to\cB_w}\, dt
\ll \int_1^\infty t^{\kappa-1-\eps}\mu_{\bar\phi}(\tau>t) \, dt
\le \int_Y \tau^\kappa \, d\mu_{\bar \phi}  \int_1^\infty t^{-(1+\eps)} \, dt<\infty,
\end{align}
which ends the proof.~\end{proof}

Lemmas~\ref{lemma-cont1} and~\ref{lemma-contdouble} ensure that $(u,s)\to \hat R(u,s)$ is analytic in $u$ 
and continuous in $s$, for all  $s\in [0,\delta_0)$, when viewed as an operator from $\cB$ to $\cB_w$ with
\[
\|\hat R(u,s)-\hat R(0, 0)\|_{\cB\to\cB_w}\ll s^\gamma+q(u),
\]
where $q(u)$ is as defined in Lemma~\ref{lemma-cont1}.

Recall that $\lambda(u)$ is well-defined for all $u\in [0,\delta_1)$ with $\delta_1=\min\{\delta, \delta_0\}$.
By \ref{H6}, there exists a  family of eigenvalues $\lambda(u,s)$ well-defined for $u,s\in [0,\delta_1)$  with $\lambda(0,0)=1$. 
Throughout, let $v(u,s)$ be a family of eigenfunctions associated with $\lambda(u,s)$.  
The next result  gives the  continuity properties  of $v(u,s)$.
\begin{lemma}
\label{lemma-contdoublev}
Assume \ref{H1}, \ref{H2},  \ref{H5} and \ref{H6}. 
Then there exists $\delta_2\le \delta_1$ such that for all  $u,s\in [0,\delta_2)$, there exists $c>0$ such that
\[
\|v(u,s)-v(u,0)\|_{\cB_w}\leq c s^\gamma\log(1/s).
\]
\end{lemma}
\begin{proof} This follows from Lemma~\ref{lemma-contdouble} and \ref{H6} (which holds for $u\ge 0$ and $s\in [0,\delta_1)$)  together with the argument recalled in the proof of 
Lemma~\ref{lemma-contKL}.~\end{proof}

We recall $\Pi(u)=\mu_{\bar\phi} (1-e^{-u\tau})$and set $\Pi_0(s):=\mu_{\bar\phi} (e^{s\bar\psi}-1)$. The result below gives the 
asymptotic  behaviour and continuity properties of $\lambda(u,s)$.

\begin{lemma}\label{lemma-asymplambda2}Assume \ref{H1}, \ref{H2}, \ref{H5}, \ref{H6} (for some range of $\kappa$) and \ref{H7}. 
Then 
as $u, s\to 0$,
\[
1-\lambda(u, s) =\Pi(u)+\Pi_0(s)+D(u,s),
\]  where $|\frac{d}{du}D(u,s)|\ll (\log(1/u))^{\kappa\gamma+2}(s^\gamma+q(u))+\mu_{\bar\phi}(\tau>\log(1/u))$ 
with  $q(u)$ as defined in Lemma~\ref{lemma-cont1}.
\end{lemma} 

\begin{proof}
Proceeding as in the proof of Lemma~\ref{lemma-asymplambda0}, write $\lambda(u, s) v(u, s)=\hat R(u, s)v(u, s)$ 
for all $u, s\in [0,\delta_1)$ with $\lambda(0, 0)=1$ and $v(0, 0)=\mu_{\bar\phi}$. 
Normalising such that $\langle v(u, s),1\rangle=1$,
\begin{align*}
1 -\lambda(u,s) &= \Pi(u)+\Pi_0(s)-W(u,s)-V(u,s) \quad \text{ for } \\
W &= \int_Y (e^{-u\tau}-1)(e^{s\bar\psi}-1)\, d\mu_{\bar\phi} \ \text{ and }\
V(u,s) = \langle (\hat R(u, s)-\hat R(0, 0))(v(u, s)-v(0, 0)),1\rangle.
\end{align*}
We first deal with $V(u,s)$. Let $Q(u,s):=\langle (\hat R(u, 0)-\hat R(0, 0))(v(u, s)-v(0, 0)),1\rangle$. Using~\eqref{eq-rus},
\begin{align*}
V(u,s)= r_0(u)
\int_0^\infty e^{-ut} \langle (e^{s\bar\psi}-1)\omega(t-\tau) [v(u)-v(0)],1\rangle\, dt+Q(u,s).
\end{align*}
By the argument used in the proof of Lemma~\ref{lemma-asymplambda0}, 
\[
|Q(u,s)|\ll  \|v(u,s)-v(0, 0)\|_{\cB_w} \Pi(u).
\]
By \ref{H2}(i),
\begin{align*}
\langle(e^{s\bar\psi}-1) \omega(t-\tau)[v(u)-v(0)], 1\rangle&\ll  \|(e^{s\bar\psi}-1)(v(u,s )-v(0, 0)\|_{\cB_w}\mu_{\bar\phi}(\omega(t-\tau)),
\end{align*}
which together with \ref{H7} (with $u=0$) gives
\begin{align*}
\langle(e^{s\bar\psi}-1) \omega(t-\tau)[v(u)-v(0)], 1\rangle\ll s^\gamma t^{\kappa\gamma} \|v(u,s )-v(0, 0)\|_{\cB_w}\mu_{\bar\phi}(\omega(t-\tau)).
\end{align*}
Proceeding as in the proof of Lemma~\ref{lemma-asymplambda0} and using that $\int_Y \tau^\kappa\, d\mu_{\bar\phi}<\infty$, 
\begin{align*}
\Big|\int_0^\infty e^{-ut}&\langle (e^{s\bar\psi}-1)\omega(t-\tau) [v(u,s)-v(0, 0)],1\rangle\, dt \Big|\\
&\ll s^\gamma  \|v(u,s)-v(0,0)\|_{\cB_w} \int_Y \tau^{\kappa\gamma}  e^{-u\tau}\, d\mu_{\bar\phi}\ll s^\gamma  \|v(u,s)-v(0, 0)\|_{\cB_w}.
\end{align*}
By a similar argument,
\begin{align*}
\left|\int_0^\infty (e^{-ut}-1)\langle (e^{s\bar\psi}-1) \omega(t-\tau) [v(u, s)-v(0, 0)],1\rangle\, dt \right|\ll s^\gamma \|v(u,s)-v(0, 0)\|_{\cB_w}.
\end{align*}
By  Lemmas~\ref{lemma-contKL} and~\ref{lemma-contdoublev}, $\|v(u,s)-v(0, 0)\|_{\cB_w}\ll s^\gamma\log(1/s)+ q(u)|\log q(u)|$.
This together with the previous two displayed equations gives 
$|V(u,s)|\ll s^\gamma(s^\gamma\log(1/s)+ q(u)|\log q(u)|)$. 

We continue with $|\frac{d}{du}V(u,s)|$. Compute that
\begin{align*}
\frac{d}{du}V(u,s)=\left\langle \frac{d}{du}\hat R(u, s)(v(u, s)-v(0, 0)),1\right\rangle+\left\langle \left(\hat R(u, s)-\hat R(0, 0)\right)\frac{d}{du}v(u, s),1\right\rangle.
\end{align*}
By \ref{H5}, \ref{H7} and the argument used above in estimating $V(u,s)$, we get 
\begin{align*}
\left|\left\langle \frac{d}{du}\hat R(u, s)(v(u, s)-v(0, 0)),1\right\rangle\right| & \ll\left|\left\langle \frac{d}{du}\hat R(u, 0)(v(u, s)-v(0, 0)),1\right\rangle\right|\\
&\ll\|v(u,s)-v(0, 0)\|_{\cB_w}\Big|\int_0^\infty  t\omega(t-\tau) e^{-ut}\,dt \Big|\\
&\ll \|v(u,s)-v(0, 0)\|_{\cB_w}\ll \log(1/u)(s+q(u)).
\end{align*}
By \ref{H5}, \ref{H7} and Lemma~\ref{lemma-deriv1} (i), $\|\frac{d}{du}v(u, s)\|_{\cB_w}\ll \|\frac{d}{du}v(u, 0)\|_{\cB_w}\ll \log(1/u)$. Thus,
\[
\left|\left\langle (\hat R(u, s)-\hat R(0, 0))\frac{d}{du}v(u, s),1\right\rangle\right|\ll (s^\gamma+q(u))\log(1/u)
\]
and $|\frac{d}{du}V(u,s)|\ll (s^\gamma+q(u))\log(1/u)$. We briefly estimate $W(u,s)$. Compute that
\begin{align*}
\left|\frac{d}{du}W(u,s)\right| &\ll\int_{\{\tau<\log(1/u)\}}\tau e^{-\tau}|1-e^{s\bar\psi}|\, d\mu_{\bar\phi}+\int_{\{\tau\ge\log(1/u)\}}\tau \, d\mu_{\bar\phi}\\
&\ll s^\gamma\int_{\{\tau<\log(1/u)\}}\tau^{\kappa\gamma+1}\, d\mu_{\bar\phi}+\mu_{\bar\phi}(\tau\ge \log(1/u))\\
&\ll s^\gamma(\log(1/u))^{\kappa\gamma+2}+\mu_{\bar\phi}(\tau\ge\log(1/u)).
\end{align*}
The conclusion follows with $D(u,s)=-(W(u,s)+V(u,s))$.~\end{proof}

For a further technical result exploiting \ref{H7} we introduce the following notation.  
For $u, s\ge 0$, set $g=s\bar\psi-u\tau$, so $\hat R(u,s)=R(e^g)$.
For $N>1$ and $x \in [a]$, define the `lower' and `upper' flattened versions of $g$ as 
\begin{equation*}
g_N^-(x)=\begin{cases} g(x) & \text{ if } \tau(y)\le N \text{ for all } y\in a, \\
\inf_{y\in [a]}g(y) & \text{ otherwise},
\end{cases}
\end{equation*}
and 
\begin{equation*}
g_N^+(x)=\begin{cases} g(x) & \text{ if } \tau(y)\le N \text{ for all } y\in a, \\
\sup_{y\in [a]}g(y) & \text{ otherwise}.
\end{cases}
\end{equation*}
Since $\psi$ is bounded above by \ref{H5},  monotonicity of the pressure function
gives
\begin{equation}\label{eq:Pb}
\cP(\bar\phi+g_N^-)\le \cP(\bar\phi+g)
\le \cP(\bar\phi+g_N^+) \le \cP(\bar\phi) + s \, \sup \bar\psi < \infty
\end{equation}
for all $u,s \geq 0$.
For fixed $u,s\in [0,\delta_0)$, set $\hat R_N^\pm(u,s)=R(e^{g_N^\pm})$.  
By \ref{H6}, the associated eigenvalues $\lambda_N^\pm(u,s)$ exist 
for all $N$ and $u,s\in [0,\delta_0)$.

\begin{lemma}\label{lemma-trunc}
Assume \ref{H2}, \ref{H5}, \ref{H6} and \ref{H7}. Then for fixed $u,s\in [0,\delta_0)$, 
\[
\lim_{N\to\infty}|\lambda(u,s)-\lambda_N^\pm(u,s)|\to 0.
\]
\end{lemma}
The reason to introduce $g^\pm_N$ is that,
unlike $g$, the potentials $g_N^\pm$ have summable variation, and therefore,
$P(\bar\phi+g_N^\pm)=\log \lambda_N^\pm(u, s)$ by \cite[Lemma 6]{Sar99}.   Here \cite[Theorem 3]{Sar99} is used to equate Gurevich pressure in
\cite[Lemma 6]{Sar99} and variational pressure in this paper, and then \cite[Theorem 4]{Sar99}
together with the existence of the eigenfunctions $v^\pm_N(u,s)$ and eigenmeasures $m^\pm_N(u,s)$ 
of $R^\pm_N(u,s)$ and its dual, to conclude that $\bar\phi+g^\pm_N$ are positively recurrent.
From this, together with \eqref{eq:Pb} and Lemma~\ref{lemma-trunc}, we conclude
\begin{align}\label{eq-presseigv}
\cP(\overline{\phi+s\psi-u}) = \log \lambda(u, s).
\end{align}

We note that this connection between the eigenvalues of a perturbed transfer operator and the pressure can be seen in similar contexts in \cite{Sar06}.

\begin{pfof}{Lemma~\ref{lemma-trunc}}
Assume $u\ne 0$ and/or $s\ne 0$.
Proceeding similarly to the argument of Lemma~\ref{lemma-contdouble}, we write
\begin{align}
\label{eq-Npl}
\hat R(u,s)-\hat R_N^\pm(u,s)= \int_0^\infty R(\omega(t-\tau) (e^g -e^{\tilde g_N}))\, dt.
\end{align}
Without loss of generality we assume that $\supp \omega(t-\tau)$ is a subset of a 
finite union $\cup_{b\in B_t} b$ 
of elements $b \in \cA$. 
Note that for any $v\in C^\alpha(Y)$,
\begin{align*}
\| R(\omega(t-\tau) (e^g -e^{\pm g_N}))v\|_{\cB\to\cB_w}\le C\| R(\omega(t-\tau))\|_{\cB\to\cB_w}\max_{b\in B_t}\|(e^g -e^{\tilde g_N}) 
1_b v\|_{\cB_w}.
\end{align*}
We only have to consider those $b\in B_t$ with $b \cap \{\tau>N\}\ne \emptyset$, 
otherwise $(e^g -e^{g^\pm_N}) 1_b=0$.
This together with \ref{H7} gives
\begin{align*}
\max_{b\in B_t}\|(e^g -& e^{\pm g_N}) 1_{b}v\|_{\cB_w}
\ll \max_{b\in B_t}\sup_b|g -\pm g_N|^\gamma\ll t^{\gamma\kappa} 1_{\{t>N\}}.
\end{align*}
Recall $\gamma<1$ for $\kappa>1$. By \ref{H2} and the previous displayed equation,  for any $\eps \in (0,1-\gamma)$, 
\begin{align}
\label{eq-rus3}
\nonumber \| \hat R(u,s)-\hat R_N^\pm(u,s)\|_{\cB\to\cB_w}
&\ll  \int_N^\infty t^{\gamma\kappa}\| M(t)\|_{\cB\to\cB_w}~dt 
=\int_N^\infty t^{-(1-\gamma-\eps)\kappa} t^{\kappa-\eps} \| M(t)\|_{\cB\to\cB_w}~dt\\
&\le N^{-(1-\gamma-\eps)\kappa}\int_N^\infty t^{\kappa-\eps} \| M(t)\|_{\cB\to\cB_w}~dt.
\end{align}
By equation~\eqref{eq-tmt}, $\int_0^\infty t^{\kappa-\eps} \| M(t)\|_{\cB\to\cB_w}\, dt<\infty$.
Using~\eqref{eq-rus3} and~\eqref{eq-tmt},
\begin{align*}
\| \hat R(u,s)-\hat R_N^\pm(u,s)\|_{\cB\to\cB_w} \ll N^{-(1-\gamma-\eps)\kappa}.
\end{align*}
This together with the argument recalled in the proof of Lemma~\ref{lemma-contKL} gives 
\[
\|v(u,s)-v_N^{\pm}(u,s) \|_{\cB\to\cB_w}\ll N^{-(1-\gamma-\eps)\kappa}\log N.
\]
where for fixed $u, s\in [0,\delta_0)$, $v(u,s)$ and $v_N^\pm(u,s)$  are the eigenfunctions associated 
with $\lambda(u,s)$ and  $\lambda_N^\pm (u,s)$, respectively. Thus,
\begin{equation*}
\| \hat R(u,s)-\hat R_N^\pm(u,s) \|_{\cB\to\cB_w}\to 0, \quad \|v(u,s)-v_N^{\pm}(u,s) \|_{\cB_w} \to 0\mbox{ as } N\to\infty.
\end{equation*}
The conclusion follows from  these estimates together with the argument used in Lemma~\ref{lemma-asymplambda0} and the 
first part of the proof of Lemma~\ref{lemma-asymplambda2}.~\end{pfof}

\subsection{Proof of Theorem~\ref{thm-relpres}}
\label{subsec-pfmainthpress}

\begin{lemma}\label{lemma-posrec}
Assume that $\bar \psi \leq C'$ (as in \ref{H5}) and that $\cP(\phi+s\psi)>0$ for $s>0$. 
Then $\cP(\overline{\phi + s\psi-\cP(\phi + s\psi)})=0$. 
\end{lemma}

\begin{proof}
Set $u_0=\cP(\phi + s\psi)$ which is positive for $s > 0$ by assumption. 
We first show that the pressures we are dealing with are finite.
As in \cite[Lemma 2]{Sar99}, finiteness of the transfer operator acting 
on constant functions is sufficient to give finiteness of the pressure, 
i.e., since we are dealing with positive $u_0$, we just need to show $\sup_{x\in Y}| (\hat R(u, s)1)(x)|<\infty$ for $u \ge 0$.  We estimate
\begin{align*}
(\hat R(u, s)1)(x) & = \sum_{F(y)=x} e^{(\overline{\phi + s\psi-u})(y)}\lesssim  
\sum_{F(y)=x} e^{\overline{\phi}(y)} = (R_01)(x)<\infty,
\end{align*}
because $\bar\psi \leq C'$ by \ref{H5}.

To show that $\cP(\overline{\phi + s\psi-u_0})\le 0$,
suppose by contradiction that $\cP(\overline{\phi + s\psi-u_0})> 0$. 
We use the ideas of \cite[e.g.\ Theorem 2]{Sar99} to truncate the whole system, i.e., restrict to the system to the 
first $n$ elements of the partition $\cA$ and the corresponding flow.  
We write the pressure of this system as $\cP_n(\cdot)$, so \cite[Theorem 2]{Sar99} says that 
$\cP_n(\overline{\phi+s\psi-u_0})> 0$ for all large $n$.  By the definition of pressure, there exists a measure 
$\bar\mu$  so that
$$
h(\bar\mu)+\int \overline{\phi+s\psi-u_0}\, d\bar\mu>0.
$$
Note that $\int\tau~d\bar\mu<\infty$ since $\tau$ is bounded on this subsystem.  
Abramov's formula~\eqref{eq-Abr} implies that the projected measure $\mu$ 
(which relates to $\bar\mu$ via \eqref{eq-projmeas}) has 
$$h(\mu)+\int \phi + s\psi-u_0 ~d\mu>0,$$
contradicting the choice $u_0 = \cP(\phi + s\psi)$.

To show that
$\cP(\overline{\phi + s\psi-u_0})\ge 0$, we will use the continuity of the pressure function 
$u\mapsto \cP(\overline{\phi + s\psi-u})$ wherever this is finite (recall from above that we have finiteness for any $u\in [0, u_0)$).  
By the definition of pressure, for  any $\eps \in (0, u_0)$, there exists  $\mu'$ with
$$
h(\mu')+\int \phi+s\psi - (u_0-\eps)~d\mu'>0.
$$
By~\eqref{eq-projmeas}, any invariant probability measure for the flow corresponds to 
an invariant probability measure $\bar\mu'$ for the map $F$.
By Abramov's formula,
$$
h(\bar\mu')+\int \overline{\phi+s\psi - (u_0-\eps)}~d\bar\mu' =
\left(h(\mu')+\int \phi+s\psi - (u_0-\eps)~d\mu'\right) \int \tau~d\bar\mu'  >0
$$
regardless of whether $\int \tau~d\bar\mu'$ is finite or not.
By continuity in $u$, $\cP(\overline{\phi + s\psi-u_0}) \ge 0$ as required.~\end{proof}

The following result is a consequence of \eqref{eq-presseigv} and of Lemma~\ref{lemma-asymplambda2}. 
\begin{lemma}
\label{lemma-derivpress}
Assume the setting of Lemma~\ref{lemma-asymplambda2}. Suppose that there exists $a>0$ such that $s=O(u^a)$, as $u\to 0$. Then
as $u\to 0^+$,
\[
\frac{d}{du}\cP(\overline{\phi+s\psi-u})=- \tau^* (1+o(1)).
\]
\end{lemma}

\begin{proof} By \eqref{eq-presseigv} and Lemma~\ref{lemma-asymplambda2}, 
\[
\frac{d}{du}\cP(\overline{\phi+s\psi-u})=\frac{d}{du} \log \lambda(u,s)=-\frac{d}{du}\left(\Pi(u)+D(u,s)\right),
\]
where $\frac{d}{du} D(u,s)=O\Big((\log(1/u))^{\kappa\gamma+2}(s^\gamma+q(u))+\mu_{\bar\phi}(\tau>\log(1/u))\Big)$.
Also, compute that
\[
\frac{d}{du}\Pi(u)=\int_Y\tau e^{-u\tau}\, d\mu_{\bar\phi}=\int_Y\tau \, d\mu_{\bar\phi}+\int_Y\tau (e^{-u\tau}-1)\, d\mu_{\bar\phi}.
\]

Since $\tau |e^{-u\tau}-1|$ is bounded by the integrable function $\tau$ and it converges pointwise to $0$, it follows from the Dominated Convergence Theorem that
$\int_Y\tau (e^{-u\tau}-1)\, d\mu_{\bar\phi}\to 0$, as $u\to 0^+$. Hence, $\frac{d}{du}\Pi(u)=\tau^*(1+o(1))$.

Since $\tau\in L^1(\mu_{\bar\phi})$ by assumption, $\mu_{\bar\phi}(\tau>\log(1/u))\to 0$, as $u\to 0$.
To conclude, note that since for some $a>0$, $s=O(u^a)$, as $u\to 0$, we have $D(u,s)=o(1)$. 
 ~\end{proof}

We can now complete

\begin{pfof}{Theorem~\ref{thm-relpres}}
Set $r(u,s) = \frac{d}{du} \cP(\overline{\phi + s \psi - u})$.
By Lemma~\ref{lemma-derivpress}, as $u\to 0$ and $s=O(u^a)$ 
for some $a>0$,
\begin{equation*}
r(u,s)=-\tau^* (1+o(1)).
\end{equation*}
For any small $u_0 > 0$, integration gives
\begin{align}
\label{eq-overl}
\cP(\overline{\phi + s\psi-u_0}) -  \cP(\overline{\phi+s\psi})  
 = \int_0^{u_0} r(u,s) \, du =  -\tau^* u_0(1+o(1)).
\end{align}

 The convexity of $s \mapsto \cP(\phi+s\psi)$ implies that 
 $\frac{d}{ds} \cP(\phi+s\psi) = \int \psi \, d\mu_\phi = \frac1{\tau^*} \int \bar\psi \, d\mu_{\bar\phi}$.
 Thus $\cP(\phi+s\psi) \ge  \frac{s}{\tau^*}\int \bar\psi \, d\mu_{\bar\phi}$ since
 $\int \bar\psi \, d\mu_{\bar\phi} > 0$; this is  guaranteed by our assumption that $C' > C \int \tau^\kappa \, d\mu_{\bar\phi}$.
 Hence, $u_0 = u_0(s) =\cP(\phi+s\psi)\gg s$ and such $u_0$ satisfies the assumptions of 
 Lemma~\ref{lemma-derivpress} with $a=1$. Thus,~\eqref{eq-overl} holds for this $u_0$.

By  Lemma~\ref{lemma-posrec} applied to $u_0 = u_0(s) =\cP(\phi+s\psi)$,
we obtain $\cP(\overline{\phi+s\psi-u_0}) = 0$. Thus, the left hand side  of~\eqref{eq-overl} is
$-\cP(\overline{\phi + s\psi)}$. By assumption, $u_0(s)>0$, for $s>0$. 
The continuity property of the pressure function gives $u_0(s)\to 0$ as $s\to 0$.
Hence,~\eqref{eq-overl} applies to $u_0(s)=\cP(\phi+s\psi)$. Thus,
$\cP(\phi+s\psi)=\frac{1}{\tau^*}\cP(\overline{\phi+s\psi}) (1+o(1))$, which ends the proof.~\end{pfof}

\section{Pressure function and limit theorems}
\label{sec-concl}

For $\psi: \cM\to\R$, define $\bar\psi:Y\to\R$ as in \eqref{eq:ind pot}
and let $\psi_T=\int_0^T \psi\circ f_t\, dt$. Throughout this section we assume that $\bar\psi$
satisfies \ref{H5}; in particular, we require that $\bar\psi=C'-\psi_0$,  
with $\psi_0 \sim C\tau^\kappa(1+o(1))$, $\int \tau^\kappa\, d\mu_{\bar\phi}<\infty$ for some $C,C'>0$.
For $s$ real, $s\ge 0$, define 
$$
\hat R_\psi(s) v := R(e^{s\bar\psi}v)=e^{sC'}R(e^{-s\psi_0}v) =: e^{sC'}\hat R_1(s).
$$
As in Subsection~\ref{subsec-assumpop} (when making assumptions on $\hat R(u)$), 
we write
\begin{align}
\label{eq-R0Lapl}
\hat R_1(s) =r_0(s)\int_0^\infty R(\omega(t-\psi_0)) e^{-st}\, dt =
r_0(s)\int_0^\infty M_0(t)\, e^{-st}\, dt,
\end{align}
where $M_0(t) := R(\omega(t-\psi_0))$ and $\omega:\R\to[0, 1]$ is an integrable function with $\supp\omega \subset [-1, 1]$ and 
$r_0(0)=1$.
Similarly to \ref{H2}, we require that
\begin{itemize}
\item[{\bf (H8)}\labeltext{{\bf (H8)}}{H8}]
There exists a function $\omega$ satisfying ~\eqref{eq-R0Lapl} and $C_\omega>0$ such that 
\begin{itemize}
\item[(i)] for any $t\in\R_{+}$ and for all $h\in\cB$, we have $\omega(t-\psi_0)h \in\cB_w$ and for all $v\in C^\alpha(Y)$,
\[
\langle \omega(t-\psi_0)vh, 1\rangle\le C_\omega \|v\|_{C^\alpha(Y)}\|h\|_{\cB_w}\mu_{\bar\phi}(\omega(t-\psi_0)).
\]

\item[(ii)] there exists $C>0$ such that for all $T>0$, 
\[
\int_T^\infty \|M_0(t)\|_{\cB\to\cB_w}\, dt\le C \mu_{\bar\phi}(\psi_0>T).
\]
\end{itemize}
\end{itemize}
\begin{rmk}
In practice, checking \ref{H8} requires no extra difficulty compared to checking \ref{H2}. But in the generality of the present abstract framework, 
there is no obvious way of deriving \ref{H8} from \ref{H2}.
\end{rmk}

Summarising the arguments used in the proof of Proposition~\ref{prop-limthF} and Theorem~\ref{thm-relpres} together 
with Proposition~\ref{prop-limthflow}, in this section we obtain

\begin{thm}
\label{thm-concl}
Assume \ref{H1} and \ref{H8} and $C' > C \int \tau \, d\mu$ as in Theorem~\ref{thm-relpres}. Let $\psi:\cM\to\R$ and assume that $\bar\psi$
satisfies \ref{H5}, \ref{H6} and \ref{H7} (with $u=0$). 
Set $\psi^* =\int_Y \bar\psi\, d\mu_{\bar\phi}$.
The following hold as $T\to\infty$.

\begin{itemize}
\item[(a)] Suppose that \ref{H4}(i) holds and that $\mu(\psi_0>t)\sim\mu(\tau^\kappa>t)$ with
$\kappa$ in (some subset of) $(0,\beta)$.

\begin{itemize}
\item[(i)] When $\beta<2$ and $\kappa\in (\beta/2,\beta)$, set $b(T)$ such that $T\ell(b(T))/b(T)^{\frac{\beta}{\kappa}}\to 1$.  Then
$\frac{1}{b(T)}(\psi_T-\psi^*\cdot T)\to^d G_{\beta/\kappa}$,  where $G_{\beta/\kappa}$ is a stable law of index $\beta/\kappa$.
This is further  equivalent to $\cP(\overline{\phi+s\psi})=\psi^*s+s^{\beta/\kappa} \ell(1/s)(1+o(1))$, as $s\to 0$. 

\item [(ii)] If  $\beta\le 2$ but $\kappa=\beta/2$, set $b(T)$ such that $T\tilde\ell(b(T))/b(T)^{\frac{2}{\kappa}}\to c>0$.
Then $\frac{1}{b(T)}(\psi_T-\psi^*\cdot T)\to^d \cN(0,c)$  and this is further  equivalent to 
$\cP(\overline{\phi+s\psi})=\psi^*s+s^2 L(1/s)(1+o(1))$, as  $s\to 0$ with $L(x)=\frac{1}{2}\tilde\ell(x)$.
\end{itemize}

\item[(b)] Suppose that \ref{H4}(ii) holds and further assume that $\mu(\psi_0>t)\ll\mu(\tau^\kappa>t)$ with
$\kappa\in(0,\beta/2)$, and  $R\bar \psi\in\cB$. 
Then there exists $\sigma\ne 0$ such that 
$\frac{1}{\sqrt T}(\psi_T- \psi^*\cdot T)\to^d \cN(0,\sigma^2)$ and this  is further  equivalent to  $\cP(\overline{\phi+s\psi})=\psi^*s+\frac{\sigma^2}{2}s^2 (1+o(1))$ 
as  $s\to 0$. 
\end{itemize}
Moreover, if \ref{H6} and \ref{H7} hold for all $u,s \in [0,\delta)$, then in both  cases (a) and (b) we have $\cP(\phi+s\psi)=\frac{1}{\tau^*}\cP(\overline{\phi+s\psi}) (1+o(1))$, for $\tau^*=\int_Y \tau\,d \mu_{\bar\phi}$.~\end{thm}

\begin{proof} We start by noting that  the proofs of the results used in the proofs of 
Theorems~\ref{thm-relpres} and Propositions~\ref{prop-limthF} and~\ref{prop-limthflow}
essentially go through when replacing $\tau$ by $\tau^\kappa$.  We briefly summarize  the required arguments.

For $s>0$, let $\lambda(s)$, $\lambda_1(s)$ be the families of eigenvalues associated with $\hat R_\psi(s)$ and $\hat R_1(s)$, respectively. Note that 
$\lambda(s)=e^{sC'}\lambda_1(s)$. Hence,
\begin{equation}
\label{eq-l0l1}
\lambda(s)-1=sC'+\lambda_1(s)-1+s^2C'+O(s^3).
\end{equation}
Further, by~\eqref{eq-presseigv},  $\cP(\overline{\phi+s\psi})=\log\lambda(s)$. 
Similar to the proofs of Proposition~\ref{prop-limthF} and Proposition~\ref{prop-limthflow}, to conclude that
(a) and/or (b) hold, we just need to estimate $1-\lambda(s)$.

For the proof of (a) and (b), we note that under \ref{H5}--\ref{H8} (where \ref{H8} and \ref{H6} 
with $u=0$ are the analogues of \ref{H2} and \ref{H3} with $\tau$ there replaced by $\psi_0$), 
Lemma~\ref{lemma-asymplambda0} gives
$1-\lambda_1(s)=\Pi_1(s)(1+O(q_1(s)|\log(q_1(s)|))$, 
where $q_1(s)=s+\mu(\psi_0>1/s)=s+s^{\beta/\kappa}\ll s^{\beta/\kappa}$ and $\Pi_1(s)=\mu_{\bar\phi} (e^{s\psi_0}-1)$ . Hence,
\[
1-\lambda_1(s)=\Pi_1(s)(1+O(s^{\beta/\kappa}\log(1/s))).
\]

In case (a) (i), the argument recalled in the proof of Corollary~\ref{cor-st} (i) 
gives $\Pi_1(s)=s\int_Y\psi_0\,d\mu_{\bar\phi}+s^{\beta/\kappa} \ell(1/s)(1+o(1))$. This together with~\eqref{eq-l0l1} gives
\[
\lambda(s)-1=\left(C'-\int_Y\psi_0\,d\mu_{\bar\phi}\right)s+s^{\beta/\kappa} \ell(1/s)(1+o(1))=\psi^*s+s^{\beta/\kappa} \ell(1/s)(1+o(1)).
\]
As a consequence, $\cP(\overline{\phi+s\psi})=\psi^*s+s^{\beta/\kappa} \ell(1/s)(1+o(1))$.
Similarly, by the argument used in  the proof of Proposition~\ref{prop-limthF} (i), this expansion of the eigenvalue/pressure is equivalent to $\frac{1}{b(n)}(\bar\psi_n-\psi^*\cdot n)\to^d G_{\beta/\kappa}$, where $\bar\psi_n=\sum_{j=0}^{n-1} \bar\psi\circ F^j$.
Further, by the argument used in Proposition~\ref{prop-limthflow} (i) with $\beta<2$ this is further equivalent to $\frac{1}{b(T)}(\psi_T-\psi^*\cdot T)\to^d G_{\beta/\kappa}$, which completes the proof of (a) (i).

The proof of a(ii) goes similar with the versions of the proofs of  Proposition~\ref{prop-limthF} (i), Proposition~\ref{prop-limthflow} (i)  with $\beta<2$ replaced by the argument used in the proofs of Proposition~\ref{prop-limthF} (i),
Proposition~\ref{prop-limthflow} (i)  with $\beta=2$.

The proof of  (b) goes again similarly with  the proofs of Proposition~\ref{prop-limthF} (i), Proposition~\ref{prop-limthflow} (i)  replaced by the argument used in the proofs of Proposition~\ref{prop-limthF} (ii),
Proposition~\ref{prop-limthflow} (ii) . 

Finally, if \ref{H6} for all $u,s\in [0,\delta_0)$ and \ref{H7} also hold, then Theorem~\ref{thm-relpres} applies 
and ensures that $\cP(\phi+s\psi)=\frac{1}{\tau^*}\cP(\overline{\phi+s\psi}) (1+o(1))$, ending the proof.~\end{proof}

\appendix

\section{Derivatives of $\hat R(u)$ and $v(u)$}
\label{app-deriv}

Recall $\hat R(u)$  can be written in the form \eqref{eq-RLapl} in Section~\ref{subsec-assumpop} and that $v(u)$ is its normalised eigenfunction.

\begin{lemma}
\label{lemma-deriv1}Let $\kappa\ge 1$ and assume that $\tau^\kappa\in L^1(\mu_{\bar\phi})$. Suppose that \ref{H1}--\ref{H3} hold.
The following hold as $u\to 0$.
\begin{itemize}
\item[(i)] $\|\frac{d}{du}\hat R(u)\|_{\cB_w}\ll 1$ and $\|\frac{d}{du} v(u)\|_{\cB_w}\ll \log(\log(1/u))$.
\item[(ii)] If $\kappa\ge 2$ then $\|\frac{d^2}{du^2}\hat R(u)\|_{\cB_w}\ll 1$ and $\|\frac{d^2}{du^2} v(u)\|_{\cB_w}\ll \log(1/u)$.
\end{itemize}
\end{lemma}

\begin{proof} Given \eqref{eq-RLapl},  we write
\[
\frac{d}{du}\hat R(u)=\frac{d}{du}r_0(u)\int_0^\infty  M(t) e^{-ut}\, dt-r_0(u)\int_0^\infty t M(t) e^{-ut}\, dt.
\]
 As in Remark~\ref{rem:r0},  as $u\to 0$, $\frac{d}{du}r_0(u)$ is bounded and we take $\frac{d}{du}r_0(u)=\gamma_1$. Under \ref{H2}, the first term is bounded in $\|\cdot\|_{\cB_w}$. 
Also, by~\eqref{eq-tmt} with $\kappa=1$, $\int_0^\infty t\| M(t)\|_{\cB\to\cB_w}\, dt<\infty$. Item (i) follows.

For the estimate on $\frac{d}{du}v(u)$, we recall $P(u)=\mu_{\bar\phi}(u)\otimes v(u)$, with $\mu_{\bar\phi}(u)(v(u))=1$. Hence,  we can write
$v(u)=\frac{P(u)1}{\mu_{\bar\phi}(u)(1)}$. Since $\|\frac{d}{du}\hat R(u)\|_{\cB_w}\ll 1$, using \ref{H1}, \ref{H3} and arguments similar to~\cite{KL99}
(see also~\cite[Proof of Corollary 3.11]{LT16})
we obtain
\[
\left\|\frac{d}{du} P(u)\right\|_{\cB_w}\ll \log(\log(1/u)).
\]
The bound $\log(\log(1/u))$ is far from optimal (given  that $\|\frac{d}{du}\hat R(u)\|_{\cB_w}\ll 1$), but this suffices for the present purpose. The same estimate holds for $|\frac{d}{du}\mu_{\bar\phi}(u)(1)|$.
These together with the formula $v(u)=\frac{P(u)1}{\mu_{\bar\phi}(u)(1)}$ give the second part of item (i).

For the estimates on the second derivatives of $\hat R(u)$, we just need to differentiate once more, recall that as in Remark~\ref{rem:r0}, $\frac{d^2}{du^2}r_0(u)=\gamma_2$, as $u\to 0$ and repeat the argument above with $\kappa=2$.
The estimate  $\|\frac{d^2}{du^2} v(u)\|_{\cB_w}\ll \log(1/u)$ follows from item (i) together with \ref{H1}, \ref{H3} and, again, 
arguments similar to~\cite{KL99} and~\cite[Proof of Corollary 3.11]{LT16}.~\end{proof}

\section{Checking \ref{H1}--\ref{H8} for the almost Anosov flow}
\label{sec-ver}

We start with a technical result that will be essential in verifying \ref{H2} and \ref{H6}.

\begin{lemma}\label{lem:supinf}
Assume that $w(x,y)$ is a homogeneous function of degree $\rho$ as 
in Proposition~\ref{prop:w-integral}.
Then
$$
 \sup_k \left\{ \sup_{\{\rf = k\}} \tau - \inf_{\{\rf = k\}}\tau \right\} < \infty.
$$
\end{lemma}

\begin{proof}
Recall that $\hat\tau(x,y) = \min\{ t > 0 : \Phi_t^{hor}(x,y) \in \hat W^{\s}\}$ and that $|\tau - \hat\tau| = O(1)$.
From Lemma~\ref{lem:tau} we know that $\rf = \hat\tau + \Theta(\hat\tau) + O(1)$.
We will first show that if $\hat\tau$ varies between, say, $n$ and $n+1$, then the corresponding $\rf$
varies by an amount of $O(1)$ as well.

Indeed, take $(x,y)$ and $(x', y')$ arbitrary such that $n \leq \hat\tau(x,y) , \hat\tau(x',y') \leq n+1$.
Define $T = \hat\tau(x,y)$ and abbreviate $q(t) = (x(t), y(t)) = \Phi_t^{hor}(x,y)$ for $0 \leq t \leq T$.
There is $t_0 \in \R$ with $|t_0| \leq 1$ such that the $y$-component $\Phi_{-t_0}^{hor}(\hat x,\hat y)$ is $y$.
Also write $q'(t) = (x'(t),y'(t)) = \Phi_{t-t_0}^{hor}(x',y')$, and the Euclidean distance $\eps(t) = |q'(t)-q(t)|$.

Clearly $| w(q'(t)) - w(q(t))| \ll | \nabla w(q(t))| \eps(t)$, so
$\eps := \eps(0) \ll T^{-(1+\beta)}$ by \eqref{eq:asymp0}. 
However, since $\hat q(0)-q(0)$ is roughly in the unstable direction,
$\eps(t)$ will increase to size $O(1)$ as $t$ increases to $\hat\tau(q')$.
This is too large, but we can set $T_1 := \min\{ t > 0 : x(t) = y(t)\}$, and then estimate $\eps(T_1)$ as follows.
Take $T'_1 = \min\{ t > 0 : x'(t) = y'(t)\}$, then $|T'_1 - T_1| = O(1)$ and we can write 
$q(T_1) = (\delta,\delta)$, $q'(T'_1) = (\delta',\delta')$ with $\eps' := \delta'-\delta$.
Furthermore, since $L$ from \eqref{eq:lf} is constant on integral curves, we have
\begin{equation}\label{eq:L1}
L(x,y) = x^uy^v \left(\frac{a_0}{v} x^2+\frac{b_2}{u} y^2\right) = \delta^{u+v+2}
\left(\frac{a_0}{v}+\frac{b_2}{u}\right),
\end{equation}
and
$$
L(q'(0)) = L(x+\eps, y) = (x+\eps)^uy^v \left(\frac{a_0}{v} (x+\eps)^2+ \frac{b_2}{u} y^2\right) = 
(\delta+\eps')^{u+v+2} \left(\frac{a_0}{v}+\frac{b_2}{u}\right).
$$
Taking the difference of both expressions, we obtain
$$
\frac{\eps}{x}\, u\, L(x,y) \left(1 + \frac{2a_0}{u a_0 x^2 + v b_2 y^2} + O\left(\frac{\eps}{x}\right) \right) =
\frac{\eps'}{\delta} \, (u+v+2)\, L(x,y)\left(1 + O\left(\frac{\eps'}{\delta}\right)\right).
$$
Divide by $(u+v+2)\, L(x,y)/\delta$ and ignore the quadratic error terms.
Then we have 
$$
\eps' \sim \frac{u}{u+v+2} \left(1+\frac{2a_0}{v b_2y^2} \right)\frac{ \delta\ \eps}{x}.
$$
But $\delta$ can be estimated in terms of $x$ using \eqref{eq:L1}, and since $\frac{u}{u+v+2} = \frac{1}{2\beta}$
and $\frac{a_0 u}{b_2v} = \frac{c_0}{c_2}$, 
we obtain
$$
\eps' \sim \frac{\eps}{2\beta} x^{\frac{1}{2\beta}-1} y^{1-\frac{1}{2\beta}} 
\left( \left(1+\frac{c_0}{c_2}\right)\left(1+\frac{c_0}{c_2 y^2}\right)\right)^{-\frac{1}{u+v+2}} \qquad \text{ as } n \to \infty.
$$
Rewriting $x$ and $\eps$ in terms of $T$ using \eqref{eq:asymp0}, we obtain
$\eps' \ll T^{-\frac32}$.
Next, because the vector-valued function
$\nabla w$ is homogeneous with exponent $\rho-1$,
Proposition~\ref{prop:w-integral} gives for $\rho < 2$: 
\begin{eqnarray*}
\int_0^{T_1} | w(q'(t))-w(q(t)) | \, dt &\ll &\sup_{t \in [0,T_1]} \eps(t) \int_0^{T_1} | \nabla w(q(t)) | \, dt + O(1) \\
&\ll& \eps'\ T^{1-\frac{\rho-1}{2}} + O(1) = T^{-\frac{\rho}{2}} + O(1).
\end{eqnarray*}
The integral $\int_{T_1}^T \| w(q'(t))-w(q(t)) \| \, dt$ is dealt with in the same way, but now integrating backwards 
from $\Phi_{\hat\tau(x,y)}^{hor}(x,y)$ and $\Phi_{\hat\tau(x',y')}^{hor}(x',y')$.
Therefore $\int_0^T w(q'(t))-w(q(t)) \, dt$  varies at most $O(1)$ on the region
$\{ (x,y) : n \leq \hat\tau(x,y) \leq n+1\}$.

Since also $\int_0^T w(q(t)) \, dt = o(T)$, we can find $N \in \N$ independent of $n$ such that
$\tau = \int_0^T 1+w(q(t)) \, dt + O(1)$ varies by at least $1$ (but at most by $O(N)$) over $\{ n \leq \hat\tau \leq n+N\}$,
and therefore $\rf$ varies by at least $1$ on this region.
It follows that for each $k$, $\{ \rf = k\} \subset \{ n \leq \hat\tau \leq n+N\}$ for some $n=n(k)$ and 
$\sup_{\{\rf = k\}} \tau - \inf_{\{\rf = k \}} \tau$ is bounded, uniformly in $k$.

The proof for $\rho \geq 2$ goes likewise.
\end{proof}

\subsection{Verifying \ref{H1}: recalling previously used Banach spaces}

\paragraph{Charts and distances:}

Throughout, $\cW^{\s}$ denotes the set of {\em admissible leaves}, which consists of 
maximal stable leaves in elements of the partition 
$\cY = \{ Y_j \}_j = \{ P_i\}_i \vee \{ \{ \rf = k\} \}_{k \geq 2}$.
It is convenient to arrange the enumeration of $\cY$ such that
$Y_j = \{ \rf = j\}$ for $j \geq 2$, and use indices $j \leq 1$ for the remaining (parts of) $P_i$.
To define distances between leaves, as in \cite[Section 3.1]{BT17}, we use charts 
$\chi_j:[0,L_{\u}(Y_j)] \times [0,1] \to Y_j$, where $L_{\u}(Y_j)$ 
is the length of the (largest) unstable leaf in $Y_j$. 
For any leaf $W \subset Y_j$, we have the parametrisation
\begin{equation}\label{eq-chart}
\chi_j^{-1}(W)=\{(g_{Y_j,W}(\rho)), \rho \in [0,1]\},
\end{equation}
so $g_{Y_j}$ is a parametrisation of $W$ in the chart.
Suppressing $Y_j$ and using $g$ and $\tilde g$ for close-by stable leaves
$W \subset Y_k$ and $\tilde W\in\cW^{\s} \in Y_\ell$ we define their by
$$
d(W,\tilde W)= 
 \begin{cases}
 \sup_{x \in [0,1]}|g(x)-\tilde g(x)| & \text{ if } k = \ell;\\
 \sup_{x \in [0,1]}|g(x)-L_{\u}(Y_k)-\tilde g(x)| + \sum_{j=k+1}^{l-1} L_{\u}(Y_j) & \text{ if } 2 \leq k < \ell;\\
 \infty & \text{ otherwise,}
 \end{cases}
$$
where $g$ and $\tilde g$ are parametrisations of $W$ and $\tilde W$ in the appropriate chart.
Here an empty sum $\sum_{j=k+1}^k$ is $0$ by convention.

From here on, in the notation $\int_W \cdot~d\mu^{\s}$, the measure $\mu^{\s}$ refers to the SRB measure $\mu_{\bar\phi}$
conditioned on the stable leaf $W$: similarly for $m^{\s}$ and Lebesgue measure. Hence the norms defined below are w.r.t.\ the invariant measure,
not the conformal measure as in \cite{BT17}. It is possible to do this due to Lemma~\ref{lem:Rint},
specifically \eqref{eq:Ku}.
\vspace{-1ex}

\paragraph{Definition of the norms:}

Given $h\in C^1(Y,\C)$, define the \emph{weak norm} by
\begin{equation}\label{eq-weaknorm}
\|h\|_{\cB_w}:=\sup_{W\in\cW^{\s}}\;\sup_{|\varphi|_{ C^1(W)}\leq 1 }\int_W h\varphi\, d\mu^{\s}.
\end{equation}

Given $q\in [0,1)$ we define the \emph{strong stable norm} by

\begin{equation}\label{eq-strongnormst}
\|h\|_{\s}:=\sup_{W\in\cW^{\s}}\; \sup_{|\varphi|_{ C^q(W)}\leq 1 }\int_W h\varphi\, d\mu^{\s}.
\end{equation}

Finally, we define the \emph{strong unstable norm} by
\begin{equation}\label{eq-strongnormunst}
\|h\|_{\u}:=\sup_\ell \sup_{W,\tilde W\in\cW^{\s} \cap Y_\ell} \
\sup_{\stackrel{|\varphi|_{C^1(W)}, |\tilde\varphi|_{C^1(\tilde W)} \leq 1}{d(\varphi,\tilde\varphi) \leq d(W,\tilde W)}}
\frac{1}{d(W,\tilde W)}\left|\int_{W} h\varphi\, d\mu^{\s} -\int_{\tilde W} h\tilde\varphi\, d\mu^{\s}\right|,
\end{equation} 
where 
$d(\varphi,\tilde\varphi) = | \varphi \circ \chi_\ell( g(\eta), \eta) - \tilde\varphi \circ \chi_\ell(\tilde g(\eta), \eta)|_{C^1([0,1])}$.

The \emph{strong norm} is defined by $\|h\|_{\cB}=\|h\|_{\s}+\|h\|_{\u}$.

\paragraph{Definition of the Banach spaces:}
We define $\cB$ to be the completion of $C^1$ in the strong norm and $\cB_w$ to be the completion in the weak norm.
As clarified in \cite[Lemma 3.2, Lemma 3.3, Proposition 3.2]{BT17}, \ref{H1} holds with $\alpha = \alpha_1 = 1$.

The spaces $\cB$ and $\cB_w$ defined above are simplified versions of functional spaces defined
in~\cite{DemersLiverani08}, adapted to the setting of $(F, Y)$.
The main difference in the present setting is the simpler definition of admissible leaves and 
the absence of a control on short leaves. This is possible due to the Markov structure of 
the diffeomorphism.

As this is the same Banach space as in \cite{BT17}, most parts of \ref{H1} can be taken from there. 
For \ref{H1v}, which is concerned with the transfer operator w.r.t.\ $\mu_{\bar\phi}$,
we first need a lemma.

\begin{lemma}\label{lem:Rint}
 Let $\bgamma = \bgamma(q)$ be the angle between the stable and unstable leaf at $q$ and 
 $h_0^{\u} = \frac{d\mu^{\u}}{dm^{\u}}$ be the density of $\mu_{\bar\phi}$ conditioned 
 on unstable leaves.
 Then
 \begin{equation}\label{eq:Rint}
   \int_{W^{\s}} R v \cdot \varphi \, d\mu^{\s} = \sum_j \int_{W_j}  v 
 \cdot \varphi \circ F \cdot K^{\u}_{W_j} \, d\mu^{\s},
 \end{equation}
 where the sum is over all preimage leaves $W_j = F^{-1}(W^{\s}) \cap \{ \rf = j\}$ and
\begin{equation}\label{eq:Ku}
 K^{\u}_{W_j}  =  \frac{J_{\mu_{\bar\phi}}^{\s} F}{J_{\mu_{\bar\phi}}F} =
 \frac{1}{J_{m_{\bar\phi}}^{\u} F} \frac{(h_0^{\u} \sin \bgamma) \circ F}{h_0^{\u} \sin \bgamma}
\end{equation}
is piecewise $C^1$ on unstable leaves. 
\end{lemma}

\begin{proof}
 We have the pointwise formulas for $R_0:L^1(m_{\bar\phi})\to L^1(m_{\bar\phi})$ 
 and $R:L^1(\mu_{\bar\phi})\to L^1(\mu_{\bar\phi})$:
$$
R_0v = \Id_Y \frac{v}{|DF|}\circ F^{-1},\qquad Rv = \Id_Y \frac{h_0 \circ F^{-1}}{h_0} \, \frac{v}{|DF|}\circ F^{-1}, 
$$
where $|DF|=J_{m_{\bar\phi}} = |\det(DF)|$.
This also shows that $J_{\mu_{\bar\phi}}F = J_{m_{\bar\phi}}  \frac{h_0}{h_0 \circ F^{-1}}$ and analogous formulas
hold for $J_{\mu^{\s}}F$ and $J_{\mu^{\u}}F$.
Integration over a stable leaf $W \in \cW^{\s}$ with preimage leaves $W_j$ gives
\begin{eqnarray*}
\int_W R_{\mu_{\bar\phi}}v \, \varphi \, d\mu^{\s} &=&
     \int_W \frac{h_0 \circ F^{-1}}{h_0} \, \frac{v}{|DF|}\circ F^{-1}  \, \varphi \, h_0^{\s} \, dm^{\s} \\
&=& \sum_j \int_{W_j} \frac{J_{m^{\s}}F}{J_{m_{\bar\phi}}F} \,  \frac{h_0}{h_0 \circ F}\, v\, \varphi \circ F \, \frac{h_0^{\s} \circ F}{h_0^{\s}} \, h_0^{\s} dm^{\s} \\
&=& \sum_j \int_{W_j} \frac{J_{\mu^{\s}}F}{J_{\mu_{\bar\phi}}F} \, v\, \varphi \circ F \,  d\mu^{\s}
=  \sum_j \int_{W_j} v \, \varphi \circ F \, K^{\u}_{W_j}\, d\mu^{\s}.
\end{eqnarray*}
Since $d\mu^{\u} = h_0^{\u} \, dm^{\u}$ for a $C^1$ density $h_0^{\u}$, and 
$dm_{\bar\phi} = dm^{\s} dm^{\u} \sin \bgamma$, whence
$J_{m_{\bar\phi}}F = J_{m^{\s}}F \cdot J_{m^\u}F \cdot \frac{\sin \bgamma \circ F}{\sin \bgamma}$, the other formula
$K^{\u}_W = \frac{1}{J_{m_{\bar\phi}}^{\u} F} \frac{(h_0^{\u} \sin \bgamma) \circ F}{h_0^{\u} \sin \bgamma}$
follows as well.
Then \cite[Equation (18)]{BT17},
leading to the parametrisation of the unstable foliation of the induced map over $\hat f = \Phi_1^{hor}$,
shows that this foliation is $C^1$. The analogous statement holds for the stable foliation.
To obtain the unstable/stable foliations of $F$, we need to flow a bounded and piecewise $C^1$
amount of time (see Lemma~\ref{lem:supinf}). Therefore $q \mapsto \bgamma(q)$ is piecewise $C^1$.
Hence the latter expression of $K^{\u}_W$ shows that it is piecewise $C^1$ on unstable leaves.
\end{proof}

Now \ref{H1}(v) follows as in \cite[Lemma 3.3]{BT17}, with $J_{\mu_{\bar\phi}}$ and $K^{\u}_{W_j}$ instead of
$|DF|^{-1}$ and $|DF|^{-1} J_{W_j}F$; this can be done because $K^{\u}_{W_j}$ is piecewise $C^1$ on unstable leaves by 
Lemma~\ref{lem:Rint}.

\subsection{Verifying \ref{H2} and \ref{H8}}

Let $\omega:\R \to [0,1]$ be a function with $\supp \omega \subset [-1,1]$ and $\int \omega(x) \, dx = 1$.
To fix a choice that suffices for the present purpose we take $\omega = 1_{[0,1]}$;
then $\gamma_1 = \frac12$ and $\gamma_2 = \frac16$ (see Remark~\ref{rem:r0}) and in particular,
 $\omega (t-\tau)= 1_{\{t\le\tau\le t+1\}}$.
The first result below verifies assumption \ref{H2}(i).

\begin{prop}\label{prop:check1}
 There is $C_\omega$ such that
 $\int_Y \omega(t-\tau) v\, d\mu_{\bar\phi} \leq C_\omega \| v \|_{\cB_w} \int_Y \omega(t-\tau) d\mu_{\bar\phi}$ 
 for all $t \geq 0$ and $v \in \cB_w$.
\end{prop}

\begin{proof} 
 By Lemma~\ref{lem:supinf}, there is $N$ (independent of $t$) and $k(t)$ such that 
 $$
 E :=  \{ q \in f^{-1}(P_0)\setminus P_0 : k(t) \leq \rf(q) \leq k(t)+N \}
 $$
 contains $\supp(\omega(t-\tau))$.
 We split $v = v^+-v^-$ into its positive and negative parts and treat them separately.
 Also we assume without loss of generality that $0 \leq \omega \leq 1$ 
 (otherwise we split $\omega$ in a positive and negative part as well).
 Then $\int \omega(t-\tau) \, v^+ \, d\mu_{\bar\phi} \leq \int_E v^+ \, d\mu_{\bar\phi}$.
 
 Let $\cW^{\s}$ denote the stable foliation of $F$ and decompose the measure $\mu_{\bar\phi}$ 
 as  $\int v^+ \, d\mu_{\bar\phi} = \int_{\cW^{\s}} \int_{W^{\s}} v^+ d\mu_{W^{\s}} \, d\nu^{\u}$.
 Note that $E$ is the union of leaves in $\tilde
\cW^{\s}$, so $\Id_E$ is constant on each $W^{\s} \in \cW^{\s}$. Take $\varphi:E \to \R^+$ arbitrary such that 
$\varphi|_{W^{\s}} \in C^1(W^{\s})$ with $\| \varphi \|_{C^1(W^{\s})} \leq 1$ for each $W^{\s} \in \cW^{\s}\cap E$.
Similar to \cite[Proposition 3.1]{BT17}, we get
 \begin{eqnarray*}
 \int_E v^+ \varphi \, d\mu_{\bar\phi}
 &=& \int_{\cW^{\s} \cap E} \int_{W^{\s}} v^+ \varphi \, d\mu_{W^{\s}} \, d\nu^{\u} 
 \leq \int_{\cW^{\s} \cap E} \| v^+ \|_{\cB_w} \| \varphi \|_{C^1(W^{\u})} \, d\nu^{\u}\\ 
 &\leq& \| v^+\|_{\cB_w} \mu_{\bar\phi}(E) 
 \ll \| v^+\|_{\cB_w} \int_{k(t)}^{k(t)+N} \int_Y \omega(s-\tau)\, d\mu_{\bar\phi} \, ds \\
 &\ll& \| v^+\|_{\cB_w} N \int_Y \omega(t-\tau)\, d\mu_{\bar\phi}.
 \end{eqnarray*}
The same holds for $v^-$, and this ends the proof.
\end{proof}

The next result  verifies assumption \ref{H2}(ii).

\begin{prop}\label{prop:check2}
Let 
$M(t) = R(\omega(t-\tau))$.
 Then $\int_T^\infty \| M(t) \|_{\cB \to \cB_w} \, dt \ll T^{-\beta}$.
\end{prop}

\begin{proof}
 We need to estimate the $\cB_w$-norm of $M(t)v = R(\omega(t-\tau) v)$.
 This means that we need to take some leaf $W^{\s} \in \cW^{\s}$, the collection
 of stable leaves in $Y$ stretching across an element of the Markov partition of $F$,
 and compute (as in Lemma~\ref{lem:Rint}) that
 $$
 \int_{W^{\s}} M(t) v \cdot \varphi \, d\mu^{\s} = \sum_j \int_{W^{\s}_j} \omega(t-\tau) v 
 \cdot \varphi \circ F \cdot K^{\u}_{W_j}\,  d\mu^{\s}.
 $$
 As in the proof of Proposition~\ref{prop:check1}, we can split $v = v^+-v^-$ and use that
 $\{ t-1 \leq \tau \leq t+1\} \subset E :=  \{k(t) \leq \rf \leq k(t)+N\}$.
Thus, following \cite[Proposition 3.2 ({\em weak norm})]{BT17} and \cite[Equation (42)]{BT17}, 
the integral for $v^+$ is bounded by
 \begin{eqnarray*}
  \sum_{j=k(t)}^{k(t)+N} \int_{W_j} v^+ J_{W^j} F |DF|^{-1} \cdot \varphi \circ F \, d\mu^{\s}
  &\leq& \| v^+\|_{\cB_w} |\varphi|_{C^1(W^{\s})} \sum_{j=k(t)}^{k(t)+N} \int_{W_j} K^{\u}_{W_j} \, d\mu^{\s} \\
   &\ll& \| v^+\|_{\cB} \sum_{j=k(t)}^{k(t)+N} \mu_{\bar\phi}(\{\rf = j\})\\
   &\ll& \| v^+\|_{\cB} \, N\, \mu_{\bar\phi}(\{\rf = k(t)\}).
 \end{eqnarray*}
For $v^+$ and $v^-$ together, this gives 
$\int_T^\infty \| M(t) \|_{\cB \to \cB_w} dt  \ll \int_T^\infty  \mu_{\bar\phi}(\{\rf = k(t)\})\, dt
 \ll \mu_{\bar\phi}(\{\tau \geq T \})$
and the proposition follows.
\end{proof}

Assumption \ref{H8} is the same as \ref{H2}, with $\tau$ replaced by $\psi_0$ and $e^{-u\tau}$ by $e^{-s\psi_0}$. 
Its verification is entirely analogous to the above.

\subsection{Verifying that $R\zeta\in\cB$ for a large class of $\zeta$ (including $\bar\psi$ in Theorem~\ref{thm-conclAn} (b)) and completing
the verification of \ref{H4}}
\label{subsec-coboundary}

Assumption \ref{H4}(i) and the tail estimate part of \ref{H4}(ii) (so that $\tau \in L^2(\mu_{\bar\phi})$) is verified in Proposition~\ref{prop:tailtau}. 
To check the remaining part of \ref{H4}(ii) and assumption on $\bar\psi$ in Theorem~\ref{thm-conclAn} (b)
we give a more general result in Lemma~\ref{lem:Rtau-well-def} below. This result ensures that
given $\bar\psi$ as in Theorem~\ref{thm-conclAn} (b) we have $R\bar\psi\in\cB$. This is needed
to verify the abstract assumptions of  Theorem~\ref{thm-concl}(b) for  $\bar\psi$ as in Theorem~\ref{thm-conclAn} (b).

\begin{lemma}\label{lem:Rtau-well-def}
For $0 < \kappa < \beta$ and $\zeta: Y \to \R$ piecewise $C^1$, define 
$$
\| \zeta \|_0 := \sup_{j \geq 1} \frac{1}{j^{\kappa}}\left( \|\zeta \, 1_{Y_j} \|_{\cB_w} + \frac{1}{j^{1+\beta}} \| \zeta\, 1_{Y_j} \|_{\u} \right)
$$
and the Banach space $\cB_0 = \{ \zeta:Y \to \R : \zeta \text{ is piecewise } C^1 \text{ and } \| \zeta \|_0 < \infty\}$.
Then $R(\cB_0) \subset \cB$.
\end{lemma}

\begin{proof}
Recall that, by Proposition~\ref{prop:w-integral} and Lemma~\ref{lem:tau}, $\rf = O(\tau)$ and vice versa.
Abbreviate $Y_j = \{ \rf = j\}$; for large $j$ these are strips close to the stable manifold
$W^{\s}_p$ of the neutral fixed point, and bounded by stable and unstable curves 
and both $F^{-1}$ and $|DF|^{-1}$ are $C^1$ on each $\overline{F(Y_j)}$.

For the stable norm $\| \ \|_{\s}$, choose an arbitrary stable leaf $W$ and $q$-H\"older 
function $\varphi \in C^q(W)$ with $|\varphi|_{C^q(W)} \leq 1$.
Let $W_j = F^{-1}(W) \cap Y_j$.
By Lemma~\ref{lem:Rint}, and the fact that $K^{\u} \ll j^{-(1+\beta)}$ and $|\varphi \circ F|_{C^1} \ll |\varphi|_{C^1}$
on each $W_j$, we have
\begin{eqnarray} \label{eq-useinh7}
 \int_W R\zeta \, \varphi\, d\mu^{\s} &=& \sum_j \int_{W_j} \zeta\, \varphi \circ F \,K^\u \,  d\mu^{\s}
 \ll \sum_j j^{-(1+\beta)} \int_{W_j} \zeta \, \varphi \circ F \, d\mu^{\s} \nonumber \\
 &\ll& \sum_j j^{-(1+\beta)} \| \zeta\, 1_{Y_j} \|_{\cB_w} \ll \sum_j j^{-1+\kappa-\beta} < \infty.
\end{eqnarray}
The stable part $\| \, \|_{\s}$ of $\| \, \|_{\cB}$ is treated in the same way.

Now for the unstable part $\| \, \|_{\u}$, let $\varphi$ such that $|\varphi|_{C^1(W)}\leq 1$. 
Using \cite[Equation (43)]{BT17}, which expresses $\int_{W_j} h \phi\, dm$ in terms of the parametrisation of $W_j$,
for any nearby leaves $W,\tilde W\in Y$, we define a diffeomorphism $v:\tilde W \to W$ as in
\cite[Proof of Proposition 3.2 ({\em unstable norm part})]{BT17}. Let $v_j = F^{-1} \circ v \circ F:\tilde W_j \to W_j$
be the corresponding bijection between the preimage leaves $W_j, \tilde W_j \subset Y_j$.
Then we compute,
\begin{align*}
 \left| \int_{\tilde W}  R\zeta \varphi\, d\mu^{\s} - \int_W R\zeta \varphi\, d\mu^{\s} \right| 
 &\leq \sum_j \left| \int_{\tilde W_j} \zeta \varphi \circ F \, K^{\u}_{\tilde W_j} \, d\mu^{\s} 
 -  \int_{W_j} \zeta \varphi \circ F \, K^{\u}_{W_j}  \, d\mu^{\s}  \right| \\
&\leq  \sum_j \int_{\tilde W_j} \zeta |\varphi \circ v \circ F - \varphi \circ F| \, K^{\u}_{\tilde W_j}  \, d\mu^{\s} \\ 
 &\quad + \sum_j \int_{\tilde W_j} \zeta \, |\varphi \circ F| 
   \, \Big| K^{\u}_{W_j}  \circ v_j - K^{\u}_{\tilde W_j} \Big|  \, d\mu^{\s}  \\
  &\quad + \sum_j \int_{W_j} |\zeta \circ v_j - \zeta| |\varphi \circ F| \, K^{\u}_{W_j}  \, d\mu^{\s} \\
 &= S_1 + S_2 + S_3.
\end{align*}
Next
$$
S_1 \leq | \varphi |_{C^1}\, d(W, \tilde W) 
\sum_j |\zeta|_{\tilde W_j}|_\infty \, |K^{\u}_{\tilde W_j} |_\infty 
\ll  \| \varphi \|_{C^1}\, d(W, \tilde W), 
$$
because as in the first part of this proof, the sum in the above expression is bounded.

For the sum $S_2$, using \eqref{eq:Ku} we split 
\begin{align*}
 \Big|K^{\u}_{\tilde W_j} - K^{\u}_{W_j} \circ v_j \Big| 
 &= \frac{1}{J^{\u}_{m_{\bar\phi}}} \frac{| (h^{\u} \sin \bgamma) \circ F \circ v_j -(h_0^{\u} \sin \bgamma) \circ F |}{h_0^{\u} \sin \bgamma}\\
& \quad + \frac{1}{ J^{\u}_{m_{\bar\phi}} } \left| \frac{J^{\u}_{m_{\bar\phi}} }{J^{\u}_{m_{\bar\phi}} \circ v_j} - 1 \right| 
\frac{(h_0^{\u} \sin \bgamma) \circ F \circ v_j}{h_0^{\u} \sin \bgamma}\\
& \quad + \frac{1}{J^{\u}_{m_{\bar\phi}}} \frac{(h^{\u} \sin \bgamma) \circ F \circ v_j}{(h_0^{\u} \sin \bgamma)\circ v_j}
\left| \frac{(h^{\u} \sin \bgamma) \circ v_j}{h_0^{\u} \sin \bgamma} - 1 \right|.
\end{align*}
By distortion estimate \cite[Equation (40)]{BT17}, this is bounded by 
$C d(W,\tilde W)\frac{1}{J^{\u}_{m_{\bar\phi}}}$ for some uniform distortion constant $C > 0$.
Therefore
$$
S_2 \leq C \, d(W,\tilde W) \sum_j \|\zeta\, 1_{Y_j} \|_{\cB_w} \, \Big| \frac{1}{J^{\u}_{m_{\bar\phi}}} \Big|_\infty
\ll d(W,\tilde W) \sum_j j^\kappa j^{-(1+\beta)} < \infty
$$
as before.

Now for $S_3$, the weighted $\| \, \|_{\u}$ part of the norm $\| \ \|_0$ gives
 $|\zeta \circ v_j - \zeta| \ll j^{\kappa+1+\beta} d(W_j, \tilde W_j) \ll j^{\kappa+1+\beta} d(W, \tilde W) L_{\u}(Y_j)$.
As in the first half of the proof,
$|K^{\u}_{W_j}|_\infty \ll j^{-(1+\beta)}$. We have $L_{\u}(Y_j) \ll \tau^{1+\beta}$
due to the small tail estimates \eqref{eq:asymp0}, so
$$
S_3 \ll |\varphi|_\infty d(W, \tilde W) \sum_j j^{\kappa} L_{\u}(Y_j)
\ll |\varphi|_\infty d(W, \tilde W).
$$
Therefore $\|R\zeta\|_\cB < \infty$ and the proof is complete.
\end{proof}

\begin{cor}\label{cor:Rtau} 
$R(\tau^\kappa h) \in \cB$ for each $0 < \kappa < \beta$ and $h \in \cB$ 
(in particular $R(\tau h) \in \cB$ as required for \ref{H4}(ii)).
\end{cor}

Since $C' - \bar \psi \sim C \tau^\kappa$, this corollary can also be used to show that $R(\bar\psi h) \in \cB$.

\begin{proof}
Since $\tau$ is the return time of a $C^1$ flow to a $C^1$ Poincar\'e section, it is piecewise $C^1$, 
and we know $\tau|_{Y_j} \ll j$. 
Taking $h \in \cB$ and $\varphi \in C^1(W)$ for an arbitrary stable leaf $W \subset Y_j$, we have
$\int_W R(\tau^ h) \, \varphi \, d\mu^{\s} \leq j^\kappa \int_W h \varphi \, d\mu^{\s} = j^\kappa | h \|_{\cB_w}$,
so $\frac{1}{j^\kappa} \| \tau^\kappa h \|_{\cB_w} \ll \| h \|_{\cB_w} < \infty$.

Using \eqref{eq:asymp0}, we have $\eps \leq \xi(y,T+1) - \xi(y,T) \ll T^{-(1+\beta)}$
for $T = \tau(x,y) + O(1)$ and $\xi(y,T)$ plays the role of $x$.
By \eqref{eq:asympt2}, $x = \xi_0(y) \tau(x,y)^{-\beta}(1+o(1))$ as $x \to 0$, so
$\eps \ll x^{\frac{1+\beta}{\beta}}$. 
Using \eqref{eq:asympt2} again,
we can estimate that as $\tau = \tau(x) \to \infty$ (so as $x \to 0$ and $\eps = o(x)$),
\begin{eqnarray}\label{eq:taudiff}
|\tau^\kappa(x+\eps,y) - \tau^\kappa(x,y)| &=& \xi_0(y)^\kappa |(x+\eps)^{-\frac{\kappa}{\beta}} - x^{-\frac{\kappa}{\beta}}|(1+o(1)) \nonumber \\
&=& \xi_0(y)^\kappa x^{-\frac{\kappa}{\beta}} \frac{\kappa}{\beta} \frac{\eps}{x}\left(1+o(1)+O\left(\frac{\eps}{x}\right)\right) \nonumber \\
&=& \frac{\kappa}{\beta} \xi_0(y)^{-\frac{\kappa}{\beta}}  \tau(x,y)^{\kappa+\beta} \eps \left(1+o(1)+O\left(\frac{\eps}{x}\right)\right).
\end{eqnarray}
Let $W, \tilde W$ nearby stable leaves in $Y_j$ with $\eps = d(W,\tilde W)$ and $v:\tilde W \to W$ a diffeomorphism.
Then, for $\varphi\in C^1(W)$, $\tilde\varphi \in C^1(\tilde W)$ with $d(\varphi, \tilde\varphi) \leq d(W,\tilde W)$,
\begin{eqnarray*}
 \int_{\tilde W} \tau^k \,  h \, \tilde\varphi \, d\mu^{\s} -  \int_{W} \tau^k \,  h \, \varphi \, d\mu^{\s}
 &\ll& \int_{\tilde W} |\tau^\kappa - \tau^\kappa \circ v_j| \,  h \, \varphi \, d\mu^{\s}
 + j^{\kappa} \left(  \int_{\tilde W} h \, \tilde\varphi \, d\mu^{\s} -  \int_{W}  h \, \varphi \, d\mu^{\s} \right) \\
 &\ll& j^{\kappa+\beta} \| h \|_{\cB_w} + j^{\kappa} \| h \|_{\u}.
\end{eqnarray*}
Therefore $\frac{1}{j^{\kappa+\beta}} \| \tau^\kappa h \|_{\u} \ll \| h \|_{\cB_w} < \infty$,
showing that $\tau^\kappa h \in \cB_0$.
Combined with Lemma~\ref{lem:Rtau-well-def}, the result follows.
\end{proof}

As in the statement of Proposition~\ref{prop-limthF} (ii) and Proposition~\ref{prop-clt} we need

\begin{cor}\label{cor:nocoboundary}
$\tilde\tau := \tau-\tau^*$ cannot be written as $h\circ F-h$ for any $h\in\cB$.
\end{cor}

\begin{proof}
By Corollary~\ref{cor:Rtau}, $R\tilde\tau\in\cB$.
Because $\tau(z) \to \infty$ as $z \to W^s(p)$, $z \in F^{-1}(P_0) \setminus P_0$, we have
$\sup_{W \in \cW^{\s}} \int_W \tilde\tau \, d\mu^{\s} = \infty$.
Therefore $\tilde \tau \notin \cB_w$ and hence $\tilde \tau$ is not a coboundary.
\end{proof}

\subsection{Verifying \ref{H7}}
Extending the inequality $|1-e^{-x}| \leq x^\gamma$ for all $\gamma \in (0,1]$ and $x \geq 0$,
we can find $C_\gamma$ depending only on $s\sup_a\bar\psi$ such that
$$
|e^{-u\tau + s\bar\psi}-1| 1_a \leq C_\gamma (u\tau + s\psi_0)^\gamma
\leq C_\gamma \left(u^\gamma \sup_a \tau^\gamma + s^\gamma \sup_a \psi_0^\gamma\right).
$$
Therefore, for each $h \in \cB$, and $\varphi \in C^1(W)$,
$W \in \cW^{\s}$,
$$
\int_W \left| (e^{-u\tau + s\bar\psi}-1) 1_a h \, \varphi \right| \, d\mu^{\s}
\leq C_\gamma \left(u^\gamma \sup_a \tau^\gamma + s^\gamma \sup_a \psi_0^\gamma\right) \int_W  h \, \varphi \, 
d\mu^{\s},
$$
so \ref{H7} follows for $C_2 = C_3 = C_\gamma$.

\subsection{Verifying \ref{H6} (and thus, \ref{H3})}

In this section we verify \ref{H6} for $\kappa>1/\beta$.

\begin{prop}\label{prop:LY}
Assume that $\bar\psi$ satisfies \ref{H5} and let $\hat R(u,s)v = R(e^{-u\tau}e^{s\bar\psi} v)$.
Then there exists $\sigma_1 \in (0,1)$ and $\delta, C_0, C_1 > 0$ such that 
for all $h\in\cB$, $n\in\N$, $0 \leq s < \delta$, and $u \geq 0$,
\[
\| \hat R(u,s)^n h\|_{\cB_w} \le C_1 e^{-un} \|h\|_{\cB_w},\qquad 
\| \hat R(u,s)^n h\|_{\cB} \le e^{-un}(C_0 \sigma_1^{-n}\|h\|_{\cB}+C_1 \|h\|_{\cB_w}).
\]
\end{prop}

Before turning to the proof, we need another lemma (which we will apply with $g \equiv 1$, but the general $g$
is needed for the induction in the proof). 

\begin{lemma}\label{lem:sum}
Assume that $\bar\psi$ satisfies \ref{H5} (in particular, $\bar\psi < C'-C^{-1}\tau^\kappa < \infty$
for some $C',C>0$ and $\kappa > 1/\beta$).
There exists $N \in \N$ depending only of $F:Y \to Y$ such that for
all positive integrable functions $g$ that are bounded and bounded away from zero,
the following holds.
There exist $\delta > 0$ such that for all $s \in [0,\delta]$,
all admissible stable leaves $W \subset Y$ and all $n \in N$,
$$
\sum_{\stackrel{W' \text{\tiny\, component}}{\text{\tiny of } F^{-n}(W)}} \int_{W'} 
e^{s\bar\psi_n} \, g \circ F^n \, K^\u_{W'}F^n \, d\mu^{\s} \leq  e^{sN \sup \bar\psi} \int_W g\, d\mu^{\s},
$$
where $K^\u_{W'}F^n$ equals the analogue of $K^\u_{W'}$ from \eqref{eq:Ku} for the iterate $F^n$ on the preimage leaf $W'$. 
\end{lemma}

\begin{proof}
Let $P_1$ be the partition element containing the images $F(\{ \rf = k\})$ of the strips
$\{ \rf=k\}_{k \geq 2}$, see Figure~\ref{fig:leaves2}. Let $N \in \N$ be such that $F^{N-1}(P_1)$ intersects each
$P_i$, $i \geq 1$. 
Let $W$ be an arbitrary stable leaf in $P_i$ for some $i \geq 1$,
and let $\cW$ be the collection of preimage leaves under $F^{-N}$.
By a change of coordinates,
$$
\sum_{W' \in \cW} \int_{W'} g \circ F^N \,  K^\u_{W'}F^n \, d\mu^{\s} = \int_W g \, d\mu^{\s} 
$$
for any integrable function $g$.

Given $a \geq 1$ to be chosen later, set 
$\cW^+ = \{ W' \in \cW : \inf_{W'} \tau^\kappa \geq 2NaC \sup \bar \psi \}$ and $\cW^- = \cW \setminus \cW^+$.
Then there is $\eps > 0$ (depending only on the geometry of the Markov map $F$ and $\frac{\sup g}{\inf g}$) such that
$\sum_{W' \in \cW^+} \int_{W'} g \circ F^N\,  K^\u_{W'}F^N \, d\mu^{\s} \geq \eps \int_W g\, d\mu^{\s}$.
Since 
$\int_{W'}\, d\mu^{\s}$ is proportional to the measure of the element in $\bigvee_{i=0}^{N-1} F^{-i}(\cP)$
that $W'$ belongs to,
Proposition~\ref{prop:tailtau} gives $\eps \gg \mu_{\bar\phi}(\{ \tau > (2NaC \sup \bar\psi)^{\frac{1}{\kappa}} \}) 
\gg Ba^{-\frac{1}{\beta'}}$ for some $B > 0$ and $\beta' := \kappa\beta > 1$.
We have $\bar\psi_N \leq (N-1) \sup\bar\psi - C^{-1}\tau^\kappa \leq (N-1-2Na) \sup\bar\psi \leq -Na\sup\bar\psi$ on $\cW^+$,
and $\bar\psi_N \leq N \sup \bar \psi$ on $\cW^-$.
Therefore
\begin{align*}
 \sum_{W' \in \cW}  & \int_{W'} e^{s\bar\psi_N} \, g \circ F^N \, K^{\u}_{W'}F^N\, d\mu^{\s}  \\
&\leq \sum_{W' \in \cW^+} \int_{W'}  e^{- s N a \sup \bar\psi}  \,g \circ F^N \, K^{\u}_{W'}F^N\,   d\mu^{\s} 
+ \sum_{W' \in \cW^-} \int_{W'} e^{s N \sup \bar\psi} \, g \circ F^N \, K^{\u}_{W'}F^N\,  d\mu^{\s} \\
& \leq \eps e^{-s N a \sup \bar\psi} \int_W g \, d\mu^{\s} + (1-\eps) e^{s N \sup \bar\psi} \int_W g \, d\mu^{\s} \\
&= \left(\eps x^{-a} + (1-\eps)x \right) \int_W g \, d\mu^{\s} =: b(x) \int_W g \, d\mu^{\s} 
\end{align*}
for $x = e^{s N \sup \bar\psi}$.
Clearly $b(1) = 1$ and $b'(x) = -a\eps x^{-(a+1)} + (1-\eps) < 0$ whenever $0 < x \leq (a\eps)^{\frac{1}{1+a}}$.
Since $\eps \geq B a^{-\frac{1}{\beta'}}$, we find
$b'(x) < 0$ for all $x \leq a_0 :=  B^{-\frac{1}{1+a}} a^{\frac{\beta'-1}{\beta'(1+a)}}$.
Choose $a > \max\{1,  B^{\frac{\beta'}{\beta'-1}} \}$, so that $a_0 > 1$.
This means that $b(x) \leq 1$ for all $1 \leq x \leq a_0$, i.e., for all 
$0 \leq s \leq \frac{\log a_0}{N \sup \bar\psi}$.

Now for general $n = pN + q$ with $0 \leq q < N$, we use induction on $p$.
Let $\cW_p(W)$ be the collection of preimage leaves of $W$ under $F^{-pN}$.
Assume by induction that
$$
\sum_{W' \in \cW_{p-1}(W_1)} \int_{W'} 
e^{s\bar\psi_{(p-1)N}} \, g \circ F^{(p-1)N} \, K^{\u}_{W'}F^{(p-1)N} \,  d\mu^{\s} \leq \int_W g \, d\mu^{\s}
$$
for all stable leaves $W_1$ and $g$ as above. Then
\begin{align*}
 \sum_{W' \in \cW_p(W)} &\int_{W'} 
e^{s\bar\psi_{pN}} \, g \circ F^{pN} \, K^{\u}_{W'}F^{pN} \,  d\mu^{\s}  \\
&\leq \sum_{W_1 \in \cW_1(W)} \sum_{W' \in \cW_{p-1}(W_1)}
\int_{W'} e^{s\bar\psi_{(p-1)N}} K^{\u}_{W'}F^{(p-1)N} \left( 
e^{s\bar\psi_N} \, g \circ F^N K^{\u}_{W'}F^N \right) \circ F^{(p-1)N}  \, d\mu^{\s} \\
&\leq \sum_{W_1 \in \cW_1(W)} \int_{W_1}
e^{s\bar\psi_N} \, g \circ F^N \, K^{\u}_{W_1}F^N\,  d\mu^{\s}  \leq \int_W g\, d\mu^{\s}.
\end{align*}

This way we proved the lemma for all multiples of $N$.
For $0 < q < N$, we get an extra factor $e^{sq \sup \bar\psi}$. This proves the lemma.
\end{proof}

\begin{pfof}{Proposition~\ref{prop:LY}}
This proof goes as in \cite[Proposition 3.2]{BT17}, but with some changes.
The factor $z^n$ is to be replaced with $e^{-u\tau_n}$ where $\tau_n = \sum_{i=0}^{n-1} \tau \circ F^i$
and the factor $e^{s\bar\psi_n}$ for $\psi_n = \sum_{i=0}^{n-1} \psi \circ F^i$ is dealt 
with using Lemma~\ref{lem:sum}.
Let $W$ and $\tilde W$ be two stable leaves in the same partition element of $\{ P_i \}_{i \geq 1}$.
Since $\tau$ and $\psi$ are not constant on partition elements, we get a third and fourth term
in the {\em strong unstable norm} part of \cite[Proposition 3.2]{BT17}
expressing the difference of $\tau_n$ on nearby preimage leaves $\tilde W_j$ and $W_j = v_j(\tilde W_j)$ of $W$ and $\tilde W$:
\begin{equation}\label{eq:S3}
S_3 = \sum_j \left| \int_{\tilde W_j} h \tilde\varphi \circ F^n
e^{s \bar\psi_n}\,  (e^{-u \tau_n}-e^{-u(\tau_n \circ v_j)} ) \, K^{\u}_{W_j}F^n\, d\mu^{\s} \right|
\end{equation}
and
\begin{equation}\label{eq:S4}
 S_4 = \sum_j \left| \int_{\tilde W_j} h \tilde\varphi \circ F^n 
 e^{-u\tau_n \circ v_j}\, (e^{s\bar\psi_n}-e^{s \bar\psi_n \circ v_j})\,  K^{\u}_{W_j}F^n\, d\mu^{\s}  \right|,
\end{equation}
for some $\tilde\varphi$ with $|\tilde\varphi|_{C^1(\tilde W)} \leq 1$.

For $S_3$, without loss of generality, we can assume that $\tau_n \circ v_j \geq \tau_n$, 
so $|e^{-u \tau_n}-e^{-u(\tau_n \circ v_j)}| \leq ue^{-u\tau_n} (\tau_n \circ v_j - \tau_n)$. 
By \eqref{eq:taudiff}, we have $|\tau(x+\eps,y) - \tau(x,y)| = 
\beta^{-1} \xi_0(y)^{-\frac{1}{\beta}}  \tau(x,y)^{1+\beta} \eps (1+o(1)+O(\frac{\eps}{x}))$.
Applied to $\eps_i := d(F^i(W_j), F^i(\tilde W_j)) \ll \tau^{-(1+\beta)} \circ F^i \eps_{i+1}$ 
this gives
\begin{align}\label{eq:epsi}
\sum_{i=0}^{n-1} | \tau \circ F^i \circ v_j - \tau \circ F^i| 
&\ll \sum_{i=0}^{n-1} \tau^{1+\beta} \circ F^i\ d(F^i(W_j), F^i(\tilde W_j)) \nonumber \\
&\ll \sum_{i=0}^{n-1} \eps_{i+1} \ll \eps_n =  d(W,\tilde W), 
\end{align}
and therefore $|e^{-u\tau_n}-e^{-u \tau_n \circ v_j} | \ll u e^{-u\tau_n} d(W,\tilde W)$.
Combining the above with \eqref{eq:S3}, we obtain 
\begin{align*}
S_3 &\leq  u\ d(W,\tilde W) \sum_j \left| \int_{\tilde W_j} h \tilde\varphi \circ F^n
e^{-u\tau_n  + s \sup \bar\psi_n} \,  K^{\u}_{W_j}F^n\, d\mu^{\s} \right| \\
&\ll  \| h \|_{\cB_w} \, u e^{-un}\, d(W,\tilde W) \sum_j \left| \int_{\tilde W_j} 
e^{s\bar\psi_n} \,  K^{\u}_{W_j}F^n\, d\mu^{\s}  \right| \\
&\leq  \| h \|_{\cB_w} \, u e^{-un}\, d(W,\tilde W) e^{sN \sup \bar\psi},
\end{align*}
where we used that $\tau_n \geq n$, and the last step follows from
Lemma~\ref{lem:sum} with the $N$ (independent of $n,u,s$) taken from that lemma too.

For $S_4$, using \ref{H5} and \eqref{eq:taudiff}, we obtain:
\begin{eqnarray*}
|e^{s\bar\psi(x+\eps)}-e^{s\bar\psi(x))}| 
&\ll& e^{s \sup \bar \psi} e^{-s\tau^\kappa/C} s|\tau^\kappa(x+\eps,y) - \tau^\kappa(x,y)| \\
&=& e^{s \sup \bar \psi} s e^{-s\tau^\kappa/C}\, \frac{\kappa}{\beta} \, \xi_0(y) \tau^{\kappa+\beta} \eps (1+o(1)+O(\eps)) \\
&\ll& e^{s \sup \bar \psi} \underbrace{s e^{-s\tau^\kappa/C} \tau^{\kappa-1}}_K  \, \tau^{1+\beta} \eps,
\end{eqnarray*}
where we compute the supremum over $\tau$ to conclude that 
$K \leq e^{-\kappa} (C\kappa)^{\frac{\kappa-1}{\kappa}} s^{\frac{1}{\kappa}}$.
Apply the above with $\eps = \eps_i := d(F^i(W_j), F^i(\tilde W_j))$ as in \eqref{eq:epsi}.
Then we can estimate $S_4$ in the same way as $S_3$:
\begin{eqnarray*}
S_4 &\leq&  s^{\frac{1}{\kappa}}\ d(W,\tilde W) \sum_j \left| \int_{\tilde W_j} h \tilde\varphi \circ F^n
e^{-u\tau_n + s\bar\psi_n} \,  K^{\u}_{W_j}F^n\, d\mu^{\s}  \right| \\
&\ll&  \| h \|_{\cB_w} e^{-un} e^{s N \sup \bar\psi}\, s^{\frac{1}{\kappa}} \, d(W,\tilde W).
\end{eqnarray*}
\end{pfof}


\begin{thebibliography}{BMT1}

\bibitem[AD01a]{AaronsonDenker01} J.~Aaronson, M.~Denker.
{Local limit theorems for partial sums of stationary sequences generated by Gibbs-Markov maps.} 
\emph{Stoch.\ Dyn.} \textbf{1} (2001) 193--237.

\bibitem[AD01b]{ADb} J.~Aaronson, M.~Denker.
\emph{A local limit theorem for stationary processes in the domain of attraction of a normal distribution.}  
In N. Balakrishnan, I.A. Ibragimov, V.B. Nevzorov, eds.,
Asymptotic methods in probability and statistics with applications. International conference, St. Petersburg, Russia, 1998, Basel: Birkh{\"a}user,
 (2001) 215--224.

\bibitem[Ab59a]{Abr59a} L.M.~Abramov. 
{The entropy of a derived automorphism.}
\emph{Dokl.\ Akad.\ Nauk SSSR} \textbf{128} (1959)  647--650.

\bibitem[Ab59b]{Abr59b} L.M.~Abramov. 
{On the entropy of a flow}. 
\emph{Dokl.\ Akad.\ Nauk SSSR} \textbf{128} (1959) 873--875.

\bibitem[AK42]{AmbKak42}  W.~Ambrose, S.~Kakutani. 
{Structure and continuity of measurable flows}. 
\emph{Duke Math.\ J.\  } \textbf{9}  (1942) 25--42.

\bibitem[BF13]{BF13} T.\ Barbot, S.\ Fenley,
{Pseudo-Anosov flows in toroidal manifolds}
\emph{Geom.\ Topol.} {\bf 17} (2013), 1877--1954. 

\bibitem[BI06]{BarIom06} L.~Barreira, G.~Iommi, 
{Suspension flows over countable Markov shifts}.
\emph{J.\ Stat.\ Phys.\ } \textbf{124} (2006) 207--230.

\bibitem[BS02]{BS02} M.\ Brin, G.\ Stuck
{\em Introduction to dynamical systems,}
Cambridge Univ. Press, Cambridge (2002).

\bibitem[B19]{Bruin} H.\ Bruin,
{\em On volume preserving almost Anosov flows,}
Preprint 2019 arXiv:1908.05675 

\bibitem[BMT18]{BMT18} H.\ Bruin, I.\ Melbourne, D.\ Terhesiu. 
{Rates of mixing for nonMarkov infinite measure semiflows}. 
\emph{Trans.\ Amer.\ Math.\ Soc.\ } \textbf{371} (2019) 7343--7386. 

\bibitem[BMT]{BMT} H.~Bruin, I.~Melbourne, D.~Terhesiu.
Lower bounds for non-Markovian finite measure semiflows and flows. 
In progress.

\bibitem[BT17]{BT17} H.\ Bruin, D.\ Terhesiu,
{Regular variation and rates of mixing for infinite measure preserving 
almost Anosov diffeomorphisms}. 
\emph{Ergodic Theory Dynam.\ Systems} (appeared online on August 10 2018, doi.org/10.1017/etds.2018.58). 

\bibitem[BTT18]{BruTerTod17} H.\ Bruin, D.\ Terhesiu, M.\ Todd.
{The pressure function for infinite equilibrium measures}.
\emph{Israel J.\ Math.} \textbf{32} (2019) 775--826.

\bibitem[DL08]{DemersLiverani08} M.~Demers, C.~Liverani. 
{Stability of statistical properties in two dimensional piecewise hyperbolic maps}.
\emph{Trans.\ Amer.\ Math.\ Soc.\ } \textbf{360} (2008) 4777--4814.

\bibitem[D98]{Dolgopyat98} D.~Dolgopyat.
{On the decay of correlations in Anosov flows}.
\emph{Ann.\ of Math.} \textbf{147} (1998) 357--390.

\bibitem[D03]{Dolgopyat03} D.~Dolgopyat.
{Limit theorems for partially hyperbolic systems}.
\emph{Trans.\ Amer.\ Math.\ Soc.\ } \textbf{356} (2003) 1637--1689.

\bibitem[E04]{Eagleson} G.\ K.\ Eagleson.
{Some simple conditions for limit theorems to be mixing}. 
\emph{Teor. Verojatnost.} \textbf{356} (2004) 1637--1689.

\bibitem[F66]{Feller66} W.~Feller.
\emph{An Introduction to Probability Theory and its Applications, II}, 
Wiley, New York, 1966.

\bibitem[G04a]{Gouezel04} S.~Gou{\"e}zel, 
{Sharp polynomial estimates for the decay of correlations}.
\emph{Israel J.\ Math.\ } \textbf{139} (2004) 29--65.

\bibitem[G04b]{Gouezel04b} S.~Gou{\"e}zel. 
{Central limit theorems and stable laws for intermittent maps}.
\emph{Prob.\ Th.\ and Rel.\ Fields} \textbf{1} (2004) 82--122.

\bibitem[GL06]{GL06} S.\ Gou{\"{e}}zel, C.\ Liverani.
{Banach spaces adapted to {A}nosov systems}.
\emph{Ergodic Theory Dynam.\ Systems} {\bf 26} (2006) 189--217.

\bibitem[H00]{Hu00} H.~Hu. 
{Conditions for the existence of SBR measures of ``almost Anosov'' diffeomorphisms}. 
\emph{Trans.\ Amer.\ Math.\ Soc.\ } \textbf{352} (2000) 2331--2367.

\bibitem[HY95]{HY95} H.\ Hu,  L-S.\ Young. 
{Nonexistence of SBR measures for some diffeomorphisms that are "almost Anosov",}
\emph{Ergodic Theory Dynam.\ Systems} {\bf 15} (1995) 67--76.

\bibitem[IJT15]{IJT15} G.\ Iommi, T.\ Jordan, M.\ Todd. 
Recurrence and transience for suspension flows.
\emph{Israel J.\ Math.\ }  {\bf 209} (2015) 547--592. 

\bibitem[KL99]{KL99} G.~Keller, C.~Liverani. 
{Stability of the spectrum for transfer operators.} 
\emph{Annali Della Scuola Normale Superiore di Pisa, Classe di Scienze} \textbf{19} (1999) 141--152.

\bibitem[L04]{Liverani04} C.~Liverani. {On contact Anosov flows.}
\emph{Ann.\ of Math.} {\bf 159}  (2004) 1275--1312.

\bibitem[LT16]{LT16} C.\ Liverani, D.\ Terhesiu. 
{Mixing for some non-uniformly hyperbolic systems}.
\emph{Annales Henri Poincar\'e} {\bf 17} no. 1 (2016) 179-226.

\bibitem[MTe17]{MelTer17} I.\ Melbourne, D.\ Terhesiu. 
{Operator renewal theory for continuous time dynamical systems with finite and infinite measure}. 
{\em Monatsh. Math.} \textbf{182} (2017) 377--431.

\bibitem[MTo04]{MTor} I. \ Melbourne,  A.\ T{\"o}r{\"o}k.
{Statistical limit theorems for suspension flows}.
\emph{Israel J.\ Math.} \textbf{144} (2004) 191--209. 

\bibitem[MV19]{MV19} I.\ Melbourne, P.\ Varandas.
{Convergence to a L\'evy process in the Skorohod ${\mathcal M}_1$ and ${\mathcal M}_2$ topologies for nonuniformly hyperbolic systems, including billiards with cusps.}
Appeared online in {\em Commun.\ Math.\ Phys.\ } DOI 10.1007/s00220-019-03501-9

\bibitem[N76]{N76} D.\ Naugler. 
{Equivalence of suspensions and manifolds with cross section. Dynamical systems.}
\emph(Proc.\ Internat. Sympos., Brown Univ., Providence, R.I., 1974) 
\textbf{Vol. II,} pp. 29–31. Academic Press, New York, 1976. 

\bibitem[S99]{Sar99} O. Sarig.
{Thermodynamic formalism for countable Markov shifts}.
\emph{Ergodic Theory Dynam.\ Systems} \textbf{19} (1999), 1565--1593.

\bibitem[S06]{Sar06} O.\ Sarig. 
{Continuous phase transitions for dynamical systems.} 
\emph{Commun.\ Math.\ Phys.\ } \textbf{267} (2006) 631--667.

\bibitem[Sav98]{Sav98} S.\ Savchenko.
{Special flows constructed from countable topological Markov chains}.
\emph{Funct.\ Anal.\ Appl.\ } \textbf{32} (1998) 32--41.

\bibitem[Z07]{Zwe07} R.\ Zweim\"uller.
{Mixing limit theorems for ergodic transformations,}
 \emph{J.\ Theoret.\ Probab.} \textbf{20} (2007) 1059--1071.


\end{thebibliography}
\end{document}